\documentclass[review]{elsarticle}

\usepackage{hyperref}
\usepackage[tbtags]{amsmath}    %%% cette option permet que le numero des equations `splittees'
\usepackage{amsfonts,amssymb,amsthm,euscript,makeidx,pifont,boxedminipage,fullpage}

\usepackage[latin1]{inputenc}
\usepackage{ucs}
\usepackage[english]{babel}
\usepackage{amsmath}

\journal{}

\newcommand{\bc}{\begin{center}}
\newcommand{\ec}{\end{center}}
\newcommand{\1}{\mathbf{1}}
\newcommand{\E}{\mathbb{E}}

\renewcommand{\P}{\mathbb{P}}

\newcommand{\R}{\mathbb{R}}
\newcommand{\N}{\mathbb{N}}

\newcommand{\cF}{\mathcal{F}}
\newcommand{\cL}{\mathcal{L}}

\newcommand{\cP}{\mathcal{P}}
\newcommand{\cU}{\mathcal{U}}

\newcommand{\cS}{\mathcal{S}}
\newcommand{\cH}{\mathcal{H}}
\newcommand{\cV}{\mathcal{V}}
\newcommand{\cR}{\mathcal{R}}

\newcommand{\pare}[1]{\left ( #1 \right )}

\newcommand{\sgn}{{\rm sgn}}
\newcommand{\cA}{\mathcal{A}}
\newcommand{\ds}{\displaystyle}

\newtheorem{theorem}{Theorem}[section]
\newtheorem{lemma}[theorem]{Lemma}
\newtheorem{proposition}[theorem]{Proposition}
\newtheorem{corollary}[theorem]{Corollary}
\newtheorem{definition}[theorem]{Definition}
\newtheorem{remark}[theorem]{Remark}

\begin{document}

\begin{frontmatter}
\title{Time inhomogeneous Stochastic Differential Equations involving the local time of the unknown process, and associated parabolic operators}
%\runtitle{Time inhomegeneous SDE with local time}

%% or include affiliations in footnotes:
\author[mymainaddress]{Pierre \'Etoré\corref{mycorrespondingauthor}}
\cortext[mycorrespondingauthor]{Corresponding author}
\ead{pierre.etore@univ-grenoble-alpes.fr}

\author[Migueladdress]{Miguel Martinez}
\ead{miguel.martinez@u-pem.fr}

\address[mymainaddress]{Université Grenoble Alpes - LJK - Bâtiment IMAG, 
700 avenue centrale,
38041 St Martin-d'Hères, France.}
\address[Migueladdress]{Universit\'e Paris-Est Marne-la-Vall\'ee - LAMA - Labex B\'ezout,
 5  Bld Descartes, Champs-sur-marne, \\
 77454 Marne-la-Vall\'ee Cedex 2, France. }

\begin{abstract}
%In this paper we study time-inhomogeneous versions of one-dimensional Stochastic Differential Equations (SDE) involving the Local Time of the unknown process. We prove existence and uniqueness results for this type of equations. Then we explore their link with parabolic Partial Differential Equations with transmission conditions, for which we give our proper treatment, using in particular the seminal paper by O.A. Ladyzhenskaya et al. (1966). Then we use these results in order to characterize the solution of the studied SDE as a Markov process.
In this paper we study time inhomogeneous versions of one-dimensional Stochastic Differential Equations (SDE) involving the Local Time of the unknown process on curves. After proving existence and uniqueness for these SDEs under mild assumptions, we explore their link with Parabolic Differential Equations (PDE) with transmission conditions. We study the regularity of solutions of such PDEs and ensure the validity of a Feynman-Kac representation formula.
These results are then used to characterize the solutions of these SDEs as time inhomogeneous Markov Feller processes.  
\end{abstract}


\end{frontmatter}

\section{Introduction}

\subsection{Presentation}
In a seminal paper on the subject \cite{legall}, J.-F. Le Gall gives necessary and sufficient conditions for pathwise uniqueness property of
time homogeneous
 one-dimensional Stochastic Differential Equations involving the Local Time (SDELT) of the unknown process, namely 
\begin{equation}
\label{eq:X-legall}
dX_t=\sigma(X_t)dW_t+\int_{\mathbb R} L^{x}_t(X)\nu(dx),\quad t\in [0,T],\quad X_0=x_0.
\end{equation}
Here $T>0$ denotes the time horizon, $x_0\in \R$ is the starting point, $\sigma~:~{\mathbb R}\rightarrow {\mathbb R}_+^{\ast}$ is a given bounded measurable function, $\nu(dx)$ is a given bounded measure on ${\mathbb R}$, and $(L_t^x(X))_{t\in [0,T]}$ stands for the symmetric local time of the unknown process $(X_t)_{t\in [0,T]}$ at point $x$. Together  with results on the existence of a weak solution for \eqref{eq:X-legall}, these results on pathwise uniqueness allow to assert that \eqref{eq:X-legall} possesses a unique strong  solution.

Note that when the measure $\nu(dx)$ is sufficiently regular and can be decomposed into $\nu(dx) = \frac{b(x)}{\sigma^2(x)}dx + \sum_{i=1}^I\beta_i \delta_{x_i}(dx)$ (for some integer $I$ and coefficients $\beta_i\in {\mathbb R}$ and a bounded measurable function $b~:~{\mathbb R}\rightarrow {\mathbb R}$), the stochastic differential equation \eqref{eq:X-legall} simplifies to 
\begin{equation}
\label{eq:X-legall-2}
dX_t=\sigma(X_t)dW_t+b(X_t)dt+\sum_{i=1}^I\beta_i dL^{x_i}_t(X),\quad t\in [0,T],\quad X_0=x_0
\end{equation}  
thanks to the occupation time formula.
In the case where $\sigma \equiv 1$, $b\equiv 0$, $I=1$, $x_1=0$, and $\beta_1\in(-1,1)$, we recover the celebrated Skew Brownian motion, which has been an endless subject of study on its own right over the recent past years (see the survey \cite{lejay-2006}; see \cite{zhang} for an example of application).

Solutions of one-dimensional SDELTs such as \eqref{eq:X-legall-2} are known to be related to operators of the form
\begin{equation}
\label{eq:op-div}
\frac{\rho}{2}\nabla\cdot(a\nabla)+b\nabla
\end{equation}
where $\rho a=\sigma^2$ and the jumps $a(x_i+)-a(x_i-)$ are proportional to $\beta_i$, $1\leq i\leq I$, in the sense of \eqref{eq:betai-poura} (see the forthcoming Subsection \ref{ssec:eq-form} for details). Note that in \eqref{eq:op-div} the $\nabla$-sign can stand either for the weak derivative, for example when one studies the problem in an $L^2$-context with the help of Dirichlet forms (see for instance \cite{ouknine-russo-trutnau}), or for the classical derivative,  when one works with Feller semigroups. Note that both approaches require to carefully specify the domain of the operator, guaranteeing that for any function~$\varphi$ in this domain, the weak derivative of $a\nabla\varphi$ exists.

Further, assuming the coefficients $\sigma$ and $b$ are smooth outside the points of singularity $x_i$, $1\leq i\leq I$, one can establish,  via a Feynman-Kac formula, the link
between the process $X$ and the classical solution $u(t,x)$ of some parabolic Partial Differential Equation (PDE) with transmission conditions (the so-called {\it
Diffraction} or {\it transmission parabolic problem}): the PDE satisfied by $u(t,x)$ involves the operator \eqref{eq:op-div}, and~$u(t,x)$ has to satisfy at any time $t$ the transmission condition
$$
a(x_i+)u'_x(t,x_i+)=a(x_i-)u'_x(t,x_i-)$$
for any $1\leq i\leq I$.
 In particular, this link opens an extended broadcast of applications such as dispersion across interfaces \cite{waymire}, diffusions in porous media \cite{lejay-geo2}, magneto-electroencephalography \cite{faugeras} (see also \cite{lejay-2006} and the references therein).
 
For proofs stating - in a time homogeneous context - the link between solutions of \eqref{eq:X-legall-2}, operators of the form~\eqref{eq:op-div}, and solutions of PDE involving transmission conditions, one may refer to the seminal papers \cite{portenko}\cite{mastrangello}, the overviews \cite{lejay-2006}, \cite{talay-survey}, and also to the series of works \cite{martinez04a}, \cite{martinez06a}, \cite{etore05a}, \cite{etore06a}, \cite{martinez12a}, where numerical schemes are presented and studied. Note that this kind of questions still seems to rise a lot of interest (see the recent papers \cite{engelbert}, \cite{mazzonetto1}).

In this paper we aim at generalizing
this family of results in a time inhomogeneous context.
 Our starting point is the study of a time inhomogeneous version of \eqref{eq:X-legall-2}, namely
\begin{equation}
\label{eq:X}
dX_t=\sigma(t,X_t)dW_t+b(t,X_t)dt+\sum_{i=1}^I\beta_i(t)dL^{x_i}_t(X),\quad t\in [0,T],\quad X_0=x_0.
\end{equation}
Here the generalization is three fold : first the coefficients $\sigma$ and $b$ are now allowed to depend on time, second the coefficients $\beta_i$ are no longer constant but are also allowed to depend on time, and third the functions $x_i~:~t\mapsto x_i(t)$ are now time-curves, so that $(L_t^{x_i}(X))_{t\in [0,T]}$ stands for the (symmetric) local time of the unknown process $(X_t)_{t\in [0,T]}$ along the time-curve $x_i$ (see the next subsection for a precise definition).

Particular versions of \eqref{eq:X} have already been examined in the literature. The so-called Inhomogeneous Skew Brownian motion 
(ISBM)
$$
dX_t=dW_t+\beta(t)dL^0_t(X)
$$
 first appears in the seminal paper \cite{weinryb} where a pathwise uniqueness result is proved. On ISBM see also the very recent papers \cite{etoremartinez2},\cite{ouknine-skewinho}. Besides some SDEs involving the local time of the unknown process on a curve appear in \cite{russo-trutnau}, \cite{trutnau1}. For example in \cite{trutnau1} the author investigates the question of the existence of a weak solution to
\begin{equation}
\label{eq:trutnau}
dR_t=\sigma\sqrt{|R_t|}dW_t+\frac{\sigma^2}{4}(\delta-bR_t)dt+(2p-1)dL^\gamma_t(R),\quad t\in[0,\infty),\quad R_0=r,
\end{equation}
where $r,\sigma,\delta>0$, $b\geq 0$, $p\in(0,1)$ and $\gamma:\R_+\to\R_+$ is assumed to be in $H^1_{loc}(\R_+)$. 
The question of pathwise uniqueness for \eqref{eq:trutnau} has been investigated separately by the same author in \cite{trutnau2}.
 Regarding existence, the difficulty relies in the fact that the coefficients in \eqref{eq:trutnau} are irregular and not bounded - even after some transformations applied to this SDE. The author manages to overcome this difficulty with the help of the machinery of generalized Dirichlet forms (see \cite{stannat}; in fact \cite{russo-trutnau} gives a general framework that was applied again in \cite{trutnau1}).   Note also that  the setting of generalized Dirichlet forms allows to study equation~\eqref{eq:trutnau}, with a curve that has only a weak regularity (by contrast in our study the curves $x_i$ will be assumed to be of class~$C^1$). However, this approach has some limitations: for example 
the choice of $\gamma$ is in fact restricted by some monotonicity assumptions, and 
the starting point $r>0$ in \eqref{eq:trutnau} can only be taken outside an exceptional set (a set of capacity zero). Note that assumptions and techniques in \cite{trutnau2} differ from the ones in \cite{trutnau1}. For example in \cite{trutnau2} the curve $\gamma$ is assumed to be continuous and locally of bounded variation and the monotonicity assumptions are dropped.

Here we will work in a more classical setting. Our coefficients $\sigma$ and $b$ will be always bounded, and $\sigma$ will be always uniformly strictly positive (however, we stress that $\sigma$ and $b$ can present discontinuities). When turning to PDE issues we will require smoothness of the coefficients outside the interfaces, in order to deal with classical solutions of PDE (and not only weak ones).  We will not allow the curves to cross, nor to touch. These assumptions will allow to study \eqref{eq:X} in full generality (multiple curves, time-dependent $\beta_i$'s etc...) with the help of classical stochastic analysis. We believe this is the first attempt in this direction.

Note that in the case we examine the use of generalized Dirichlet forms would probably allow to get alternate proofs and relax the assumptions on the coefficients. But
in our opinion this topic surely requires further investigations (see also our comments in
Subsection \ref{ss:sol-faible}).

The content and organization of the paper are the followings.

In  Section \ref{sec:prelim}, we give preliminary material for the study of equation \eqref{eq:X}. First, we recall results on the related martingale problem. Second, we recall  some pathwise uniqueness results to be found in \cite{legall} (available in a time inhomogeneous context).  Then we present the recent It\^o-Peskir formula (see \cite{peskir}). This formula is a kind of generalization of the Itô-Tanaka formula to time-dependent functions. We provide a slight adaptation of the It\^o-Peskir formula (in the case where the curves are $C^1$ functions). Since we aim at studying the generator of the solutions of equation \eqref{eq:X},
we also give introductory material to the semigroups associated to time inhomogeneous Markov processes and Feller evolution systems.

In  Section \ref{sec:SDE} we use the result of Peskir to prove a change of variable formula, that will be of crucial use in the rest of the paper. Then we give conditions for the equation \eqref{eq:X} to admit a weak or strong solution, to enjoy pathwise uniqueness. The method follows closely Le Gall \cite{legall} by the mean of a space transform that eliminates the local times. But as the local times are now taken on curves and the $\beta_i$'s are time-dependent, we have to use the It\^o-Peskir formula, at places where Le Gall uses the classical It\^o-Tanaka formula.

Section \ref{sec:feynman} is devoted to the proof of a Feynman-Kac representation linking the solution of \eqref{eq:X} and the solution of a parabolic partial differential equation with transmission conditions along the curves $x_i$. It is assumed that the solution of the parabolic PDE with transmission conditions is smooth enough in order to apply the change of variable formula of Section \ref{sec:SDE}.

Section \ref{sec:EDP} is devoted to the study of the parabolic PDE with transmission conditions appearing in the previous section. We first study its weak interpretation and manage to show, by adapting the arguments in \cite{lions-magenes}, that a weak solution exists.
As regarding classical solutions, we rely on the main result of the reference article \cite{lady1}, where the coefficient $\rho$ in \eqref{eq:op-div} is constantly equal to one and the sub-domains are cylindrical (non-moving interfaces). For the sake of completeness, we give hints of the main steps of the proof given in \cite{lady1}. Again, using the fact that the space dimension is one, and space transform techniques, we manage to generalize the result to the solution of the parabolic PDE with transmissions conditions, with $\rho\neq 1$ and moving interfaces. Thus, we fully prove that the solution of the parabolic PDE with transmission conditions is smooth enough to assert the validity of the Feynman-Kac representation given in the previous section (see the conclusion at the end of Section \ref{sec:EDP}).

Section \ref{sec:markov} is an attempt to characterize the Markov generator of the solution $X=(X_t)_{t\in [0,T]}$ to \eqref{eq:X}. We first give a set of sufficient conditions for $X$ to be a Feller time inhomogeneous Markov process (see Subsection~\ref{ss:markov-inho} for a definition).  Then, we manage to identify fully the generator of $X$ in the case of non-moving interfaces. The case of moving interfaces seems more difficult to handle since we do no longer have the continuity of the time derivative of the associated parabolic transmission problem.

An Appendix contains detailed material regarding the It\^o-Peskir formula and PDE technical aspects.

\vspace{0.2cm}
Some notations frequently used in the paper are introduced in the next subsection.

\subsection{Notations}

In the following notations an interval $[0,T]\subset\R_+$ is given and kept fixed (with $0<T<\infty$).

For any semi-martingale $X$ the process $L^0_.(X)=(L^0_t(X))_{t\in [0,T]}$ is the symmetric local time at point~$0$ of~$X$.
And for any continuous function of bounded variation $\gamma:[0,T]\to\R$ we denote by
$L^\gamma_.(X)$ the process defined by
$$
L^\gamma_t(X)=L^0_t(X-\gamma),\quad\forall t\in [0,T].$$
So that
$$
\forall t\in [0,T],\quad L^\gamma_t(X)=\P-\lim_{\varepsilon\downarrow 0}\frac{1}{2 \varepsilon }
\int_0^t \1_{|X_s-\gamma(s)|<\varepsilon}\,d\langle X\rangle_s,$$
(see \cite{RY}, Exercise VI-1-25, and \cite{peskir}).

\vspace{0.2cm}
For any topological spaces $U,V$ we denote by $C(U)$ the set of continuous $\R$-valued functions on $U$, and by $C(U,V)$ the set of continuous functions from $U$ to $V$.

$C_b(U)$ denotes the set of continuous bounded functions on $U$.

$C^p(U)$, $p\in\bar{\N}=\N\cup\{\infty\}$, denotes the set of continuous functions on $U$ with continuous derivatives up to order~$p$.

\vspace{0.2cm}

$C_0(\R)$ denotes the set on continuous functions on $\R$ vanishing at infinity.

%We denote by $C_0(\R)$  the space of $\R$-valued continuous functions of $\R$, vanishing at infinity.
\vspace{0.2cm}

We denote $E=[0,T]\times\R$ and $E^\circ=[0,T)\times\R$.
\vspace{0.1cm}

\vspace{0.1cm}
Let $F\subset E$ be an open subset of $E$. We denote by $C^{p,q}(F)$ the set of continuous functions on $F$, with continuous derivatives up to order $p$ in the time variable, and up to order $q$ in the space variable (with the convention that for example $q=0$ corresponds to the continuity w.r.t. the space variable).

We denote by $C_0(E)$ the space of $\R$-valued continuous functions of $E$, vanishing at infinity, i.e. when $|x|\to\infty$, $(t,x)\in E$. We will denote this space $C_0$ in short when this causes no ambiguity. The spaces $C_0(\R)$ and $C_0(E)$ are endowed with the corresponding supremum norm, for which we use the common notation $||\cdot||_\infty$ (which norm is meant will be made clear from the context ).

\vspace{0.1cm}
We denote by $C^{\infty,\infty}_c(E)$ the set of $\R$-valued functions of $E$ that are  $C^{\infty,\infty}(E)$, and of compact support with respect to the space variable (i.e. for 
$f\in C^{\infty,\infty}_c(E)$,
 for any $t\in[0,T]$, the function $f(t,\cdot)$ is of compact support). 

We denote  by $C^{\infty,\infty}_{c,c}(E)$ the set of $\R$-valued functions of $E$ that are in $C^{\infty,\infty}(E)$, and of compact support 
$K\subset (0,T)\times \R$.

Note that throughout the whole text, for a space-time function $g\in C(E)$, we will denote  by $g'_x(t,x)$, $g''_{xx}(t,x)$ and $g'_t(t,x)$ its classical partial derivatives at point $(t,x)\in E$, whenever they exist.

\vspace{0.2cm}

For a function in $L^2(\R)$ we denote by $\frac{\mathrm{d}f}{\mathrm{dx}}$ its first derivative in the distribution sense. We denote by $H^1(\R)$ the usual Sobolev space of those functions $f$ in $L^2(\R)$ such that $\frac{\mathrm{d}f}{\mathrm{dx}}$ belongs to $L^2(\R)$.
We denote by $H^{-1}(\R)$ the usual dual space of $H^1(\R)$.

We denote $L^2(0,T;L^2(\R))$ the set of measurable functions $f(t,x)$ s.t.
$$
\int_0^T\int_\R|f(t,x)|^2dxdt<\infty.$$
For $f\in L^2(0,T;L^2(\R))$ we denote by $||f||^2$ the above quantity.

We denote by  $L^2(0,T;H^1(\R))$ the set of mesurable functions 
$f(t,x)$ such that for any $t\in[0,T]$ the function $f(t,\cdot)$ is in $H^1(\R)$ and 
$$
\int_0^T\int_\R|f(t,x)|^2dxdt+\int_0^T\int_\R\big| \frac{\mathrm{d}f}{\mathrm{dx}}(t,x) \big|^2dxdt<\infty.$$

For a function $f\in L^2(0,T;L^2(\R))$ we denote by $\frac{\mathrm{d}f}{\mathrm{dt}}$ its first derivative with respect to time in the distribution sense (see Remark \ref{rem:der-t-distru} for some details). 

We will denote by $H^{1,1}(E)$ the set of functions in $L^2(0,T;H^1(\R))$ such that $\frac{\mathrm{d}f}{\mathrm{dt}}$ belongs to
$ L^2(0,T;L^2(\R))$. 
It is equipped with the norm $f\mapsto\Big( ||f||^2 + \big|\big| \frac{\mathrm{d}f}{\mathrm{dx}}  \big|\big|^2 
+ \big|\big| \frac{\mathrm{d}f}{\mathrm{dt}} \big|\big|^2 
   \Big)^{1/2}$.
   
   Finally we will denote by $H^{1,1}_0(E)$ the closure in 
   $H^{1,1}(E)$ of $C^{\infty,\infty}_{c,c}(E)$
    with respect to the just above defined norm.
Note that for $\varphi\in H^{1,1}_{0}(E)$ we have $\varphi(0,\cdot)=\varphi(T,\cdot)=0$, and $\lim_{|x|\to\infty}\varphi(t,x)=0$, $t\in[0,T]$.

For $0<m<M<\infty$ we denote by $\Theta(m,M)$ the set of functions $\sigma:[0,T]\times\R\to [m,M]$ that are measurable. We denote
by $\Xi(M)$ the set of functions $b:[0,T]\times\R\to [-M,M]$ that are measurable. 

\vspace{0.2cm}

Let $I\in \N^*=\N\setminus\{0\}$. For each $1\leq i\leq I$, let $x_i:[0,T]\to\R$ be a continuous function of bounded variation,
and assume that $x_i(t)<x_j(t)$ for all $t\in [0,T]$ and all $1\leq i<j\leq I$.

Given such a family $(x_i)_{i=1}^I$ we will denote $D^x_0=\{(t,z)\in [0,T] \times\R:\,z<x_1(t)\}$, $D^x_I=\{(t,z)\in [0,T]\times\R:\,z>x_I(t)\}$ and, for any $1\leq i\leq I-1$,
$D^x_i=\{(t,z)\in [0,T] \times\R:\,x_i(t)<z<x_{i+1}(t)\}$.

We will denote 
\begin{equation}
\label{eq:def-Delta-x}
\Delta_{{\bf x}}=\{(t,x_i(t)):0\leq t\leq T\}_{i=1}^I\subset E
\end{equation}
(this will be clear from the context which family $(x_i)_{i=1}^I$ is dealt with).

%%Given such a family $(x_i)_{i=1}^I$ we will denote $D^x_0=\{(t,z)\in [0,T] \times\R:\,z<x_1(t)\}$, $D^x_I=\{(t,z)\in [0,T]\times\R:\,z>x_I(t)\}$ and, for any $1\leq i\leq I-1$,
%%$D^x_i=\{(t,z)\in [0,T] \times\R:\,x_i(t)<z<x_{i+1}(t)\}$.

\vspace{0.1cm}
We will say that a space-time function $\sigma$ in $\Theta(m,M)$ satisfies the 
${\bf H}^{(x_i)}$-hypothesis if:
$$
 \sigma\in C^{0,1}(E\setminus\Delta_{\bf x}), \quad \max_{1\leq i\leq I}\sup_{t \in [0,T]}\sup_{x_i(t)< x< x_{i+1}(t)}|\sigma'_x(t,x)|<\infty$$
$$
\text{and}\quad \sup_{t\in [0,T]}\sup_{ x< x_1(t)}|\sigma'_x(t,x)|<\infty,\quad \sup_{t\in [0,T]}\sup_{ x> x_I(t)}|\sigma'_x(t,x)|<\infty.
$$

We define the ${\bf AJ}^{(x_i)}$-hypothesis (AJ for Average Jumps) in the following way: a bounded space-time function $\sigma$ satisfies the ${\bf AJ}^{(x_i)}$-hypothesis if,
$$
\exists 0<C<\infty, \;\; \infty>C\int_0^T\sum_{x\leq z\leq y}|\sigma^2(s,z+)-\sigma^2(s,z-)|ds
\geq \sum_{x\leq z\leq y}|\sigma^2(t,z+)-\sigma^2(t,z-)|, $$
for all $x,y\in\R$, $t\in [0,T]$.

\begin{remark}
This roughly speaking, means that the size of the jumps of $\sigma^2$ are not allowed to go too far from a kind of time-averaged size jump. See Remark \ref{rem:surHhyp} below for a comment on why this technical hypothesis is needed.
\end{remark}

  A space-time function $g$  in $\Theta(m,M)$, in $\Xi(M)$ or in
$C_c(E)$ will be said to satisfy the ${\bf H}^{(t)}$-hypothesis if

$$
 g\in C^{1,0}(E\setminus\Delta_{\bf x}), \quad \max_{1\leq i\leq I}\sup_{t \in [0,T]}\sup_{x_i(t)< x< x_{i+1}(t)}|g'_t(t,x)|<\infty$$
$$
\text{and}\quad \sup_{t\in [0,T]}\sup_{ x< x_1(t)}|g'_t(t,x)|<\infty,\quad \sup_{t\in [0,T]}\sup_{ x> x_I(t)}|g'_t(t,x)|<\infty
$$
(this hypothesis will be used for the study of the PDE aspects).

Note that the same kind of notations will be used for a family $y_i:[0,T]\to\R$, $1\leq i\leq I$, satisfying the same assumptions (for example in Corollary \ref{cor-peskirmulti} below).

Finally, we fix notations for two sets of type $\Delta_{\bf x}$ that play a special role in the sequel. Those are
\begin{equation}
\label{eq:def-Delta}
\Delta=\{(t,i):\,0\leq t\leq T\}_{i=1}^I\subset E\quad\text{ and }\quad \Delta_0=\{(t,0):\,0\leq t\leq T\}.
\end{equation}

\vspace{0.2cm}
For any function $f:\R\to\R$ and any $x\in\R$ such that $f(x+)=\lim_{y\downarrow x}f(y)$  and
$f(x-)=\lim_{y\uparrow x}f(y)$ both exist,
we will sometimes use the following notations :
$$
f_\pm(x):= \frac{f(x+)+f(x-)}{2}\quad\text{and}\quad \vartriangle f(x)=\frac{f(x+)-f(x-)}{2}.$$
In particular if $f:\R\to\R$ is differentiable, except on a finite number of points $x_1<\ldots <x_I$, where $f'(x_i\pm)$, $1\leq i\leq I$
exist, note that the function $f'_{\pm}$ is defined on the whole real line and represents the absolute part of $f'(dx)$, the derivative of $f$ in the generalized sense; in other words,
$$
f'(dx)=f'_{\pm}(x)dx+\sum_{i=1}^I2\vartriangle f(x_i)\delta_{x_i}(dx).
$$

\section{Preliminaries and known results concerning the stochastic aspects of the problem}
\label{sec:prelim}

\subsection{Well-posedness of the martingale problem associated to discontinuous coefficients}
\label{ss:pbmart}

Of crucial importance is the following result, to be found in \cite{stroockvar}.

\begin{theorem}[\cite{stroockvar}, Exercise 7.3.3]
\label{thm:pbmart}
Let $\bar{\sigma}\in\Theta(\bar{m},\bar{M})$ and $\bar{b}\in\Xi(\bar{M})$ (for some $0<\bar{m}<\bar{M}<\infty$). Then the martingale problem associated to $\bar{\sigma}^2$ and $\bar{b}$ is well-posed.
\end{theorem}

The first important consequence of this result is that the for any $(s,y)\in E$ the SDE
\begin{equation*}
%\label{eq:Y}
dY_t=\bar{\sigma}(t,Y_t)dW_t+\bar{b}(t,Y_t)dt,\quad t\in [s,T],\quad Y_s=y
\end{equation*}
has a weak solution (\cite{stroockvar}, Theorem 4.5.1), unique in law (\cite{stroockvar}, Theorem 5.3.2).  The second one is that this weak solution is (time inhomogeneous) Markov (\cite{stroockvar}, Theorem 6.2.2; see also the forthcoming Subsection~\ref{ss:markov-inho} for comments on time inhomogeneous Markov processes).

\begin{remark}
Note that the result of Theorem \ref{thm:pbmart} is available for time-dependent coefficients, only because the dimension of the space variable is $d=1$. For $d=2$, up to our knowledge, such results exist but with a time homogeneous diffusion matrix (\cite{stroockvar}, Exercise 7.3.4).
\end{remark}

\subsection{Pathwise uniqueness results and strong solutions of time inhomogeneous SDEs with discontinuous coefficients}

We have the following results.

\begin{theorem}[J.-F. Le Gall, \cite{legall}]
\label{thm:legall}
Let $\bar{\sigma}\in\Theta(\bar{m},\bar{M})$ and $\bar{b}\in\Xi(\bar{M})$ for some $0<\bar{m}<\bar{M}<\infty$. Assume further that there exists a strictly increasing
function $f:\R\to\R$ such that
\begin{equation}
\label{eq:f}
|\bar{\sigma}(t,x)-\bar{\sigma}(t,y)|^2\leq |f(x)-f(y)|,\quad\forall (t,x,y)\in [0,T]\times\R\times\R.
\end{equation}
Then the SDE
\begin{equation}
\label{eq:Y}
dY_t=\bar{\sigma}(t,Y_t)dW_t+\bar{b}(t,Y_t)dt,\quad t\in [0,T],\quad Y_0=y_0
\end{equation}
enjoys pathwise uniqueness.

\end{theorem}

As an immediate consequence we get the following corollary.

\begin{corollary}
\label{cor-legall}

Let $I\in \N^*$. For each $1\leq i\leq I$, let $y_i:[0,T] \to\R$ be a continuous function of bounded variation,
and assume that $y_i(t)<y_j(t)$ for all $t\in [0,T]$ and all $1\leq i<j\leq I$. 

Let $\bar{\sigma}\in\Theta(\bar{m},\bar{M})$ and $\bar{b}\in\Xi(\bar{M})$ for some $0<\bar{m}<\bar{M}<\infty$. 

The SDE \eqref{eq:Y} has a weak solution.

Assume further that $\bar{\sigma}$ satisfies the $\mathbf{H}^{(y_i)}$ and $\mathbf{AJ}^{(y_i)}$-hypothesis.

Then the SDE \eqref{eq:Y} enjoys pathwise uniqueness and 
has in fact a unique strong solution.

\end{corollary}

\begin{proof}
As already pointed out in Subsection \ref{ss:pbmart} equation \eqref{eq:Y} has a weak solution. 
We aim now at using Theorem~\ref{thm:legall}.
Then the well known results of Yamada and Watanabe (\cite{yamada}) will provide the desired conclusion.

First we notice that for all $(t,x,y)\in [0,T]\times \R\times \R$,
$$
|\bar{\sigma}(t,x)-\bar{\sigma}(t,y)|^2\leq \bar{\sigma}^2(t,x)+\bar{\sigma}^2(t,y)-2(\bar{\sigma}^2(t,x)\wedge\bar{\sigma}^2(t,y))
= |\bar{\sigma}^2(t,y)-\bar{\sigma}^2(t,x)|.
$$
Thus, to get the result by Theorem \ref{thm:legall} it suffices to find a stricly increasing
function $f:\R\to\R$ such that
\begin{equation}
\label{eq:fbis}
|\bar{\sigma}^2(t,x)-\bar{\sigma}^2(t,y)|\leq |f(x)-f(y)|,\quad\forall (t,x,y)\in [0,T]\times\R\times\R.
\end{equation}

Using the $\mathbf{H}^{(y_i)}$-hypothesis, we set
$$
K=\max\big\{ \,  \sup_{t\in [0,T]}\sup_{ x< y_1(t)}|(\bar{\sigma}^2)'_x(t,x)|,
\, \max_{1\leq i\leq I}\sup_{t\in [0,T]}\sup_{y_i(t)\leq x< y_{i+1}(t)}|(\bar{\sigma}^2)'_x(t,x)|,
\, \sup_{t\in [0,T]}\sup_{ x\geq y_I(t)}|(\bar{\sigma}^2)'_x(t,x)|   \, \big\}<\infty.$$

One can define a strictly increasing function $f:\R\to\R$ by
$$
f(x)=Kx+C\sum_{z\leq x}\int_0^T|\bar{\sigma}^2(s,z+)-\bar{\sigma}^2(s,z-)|ds,$$
 where $C$ is the constant involved in the $\mathbf{AJ}^{(y_i)}$-hypothesis (note that as $\sum_{z\leq x}|\bar{\sigma}^2(s,z+)-\bar{\sigma}^2(s,z-)|$ is finite and bounded -for any $s$-, Fubini's Theorem ensures that $f$ takes finite values). Then one can use the
  $\mathbf{H}^{(y_i)}$ and $\mathbf{AJ}^{(y_i)}$-hypotheses to check that for $x<y$,
 $$
 \begin{array}{lll}
 |\bar{\sigma}^2(t,x)-\bar{\sigma}^2(t,y)|&\leq& K(y-x)+\sum_{x\leq z\leq y}|\bar{\sigma}^2(t,z+)-\bar{\sigma}^2(t,z-)|\\
 \\
 &\leq & K(y-x)+ C\int_0^T\sum_{x\leq z\leq y}|\bar{\sigma}^2(s,z+)-\bar{\sigma}^2(s,z-)|ds\\
 \\
 &=&f(y)-f(x)=|f(y)-f(x)|.
 \end{array} 
 $$
 Thus $f$  satisfies \eqref{eq:fbis}.  
\end{proof}

\begin{remark}
\label{rem:surHhyp}
It would be tempting to set $f(x)=Kx+\sum_{z\leq x}\sup_{s\in[0,T]}|\bar{\sigma}^2(s,z+)-\bar{\sigma}^2(s,z-)|$ in order to try to check 
\eqref{eq:f}. But
as $\sup_{s\in[0,T]}|\bar{\sigma}^2(s,z+)-\bar{\sigma}^2(s,z-)|$ could be non zero for uncountably many values of~$z$ the function $f$ could be not well defined as a function from $\R$
to $\R$. This justifies our assumption $\mathbf{AJ}^{(y_i)}$.

\end{remark}

\subsection{The It\^o-Peskir formula}
\label{ssec:peskir}

Our fundamental tool is the following result due to G. Peskir (see \cite{peskir}).

\begin{theorem}[Time inhomogeneous symmetric Itô-Tanaka formula (\cite{peskir})]
\label{thm-urpeskir}
Let $Y$ be a continuous $\R$-valued semimartingale. Let $\gamma:[0,T]\to\R$ be a continuous function of bounded variation.

Denote 
$C=\{(t,x)\in [0,T]\times {\mathbb R}~:~x<\gamma(t)\}$ and $D=\{(t,x)\in [0,T]\times {\mathbb R}~:~x>\gamma(t)\}$. 

Let $r\in C(E)\cap C^{1,2}(\overline{C})\cap C^{1,2}(\overline{D})$.
Then, for any $0\leq t<T$,
\begin{equation}
\label{eq:peskir}
\begin{array}{lll}
r(t,Y_t)&=&\ds r(0,Y_0)+\int_0^t\frac 1 2(r'_t(s,Y_s+)+r'_t(s,Y_s-))ds
+\int_0^t\frac 1 2(r'_y(s,Y_s+)+r'_y(s,Y_s-) )dY_s\\
\\
&&\ds +\frac 1 2\int_0^tr''_{yy}(s,Y_s)\1_{Y_s\neq \gamma(s)}d\langle Y\rangle_s
+\frac 1 2 \int_0^t(r'_y(s,Y_s+)-r'_y(s,Y_s-))dL^\gamma_s(Y).\\
\end{array}
\end{equation}
\end{theorem}

Note that in the above Theorem, the assumption $r\in C^{1,2}(\overline{C})\cap C^{1,2}(\overline{D})$ means that $r$ restricted to~$C$ coincides with a function $r_0$ lying in the whole space $C^{1,2}(E)$, and $r$ restricted to $D$ coincides with a function~$r_1$ lying in the whole space $C^{1,2}(E)$. 

However, when dealing with PDE aspects (Sections \ref{sec:feynman}, \ref{sec:EDP} and \ref{sec:markov}), we will need to apply the Itô-Peskir formula to functions that have less smoothness: these functions will only possess continuous partial derivatives (of order one in time and at least two in the space variable) with limits all the way up to the boundary $\Delta_\gamma=\{(t,x)\in [0,T]\times\R:\,x=\gamma(t)\}$. The price to pay, in order to get the same formula 
\eqref{eq:peskir},
is then to require additional smoothness of the curve $\gamma(t)$: we require it to be of class $C^1$.

In Theorem \ref{thm-peskir} below, we give the adaptation of the Itô-Peskir formula that will be used in Sections \ref{sec:feynman} and \ref{sec:markov} (in fact the formula is the key the forthcoming Proposition \ref{prop:transfo}, that will be used repeatedly in the sequel).  Note that the assumptions on the function $r$ in Theorem \ref{thm-urpeskir} imply the ones in Theorem \ref{thm-peskir}. But of course, on the opposite, the fact that $\gamma$ is $C^1$ implies the fact that it is continuous of bounded variation.

 For the sake of completeness, we will give hints for a full proof of Theorem \ref{thm-peskir} in the Appendix along the same lines as \cite{peskir}.

\begin{theorem}
\label{thm-peskir}
Let $Y$ be a continuous $\R$-valued semimartingale. Let $\gamma:[0,T]\to\R$ be a function of class $C^1$, and 
consider $\Delta_\gamma =\{(t,x)\in [0,T]\times\R:\,x=\gamma(t)\}$.
Let $r\in C(E)\cap C^{1,2}(E^\circ\setminus \Delta_\gamma)$ such that the limits $r'_t(t, \gamma(t)\pm)$, $r'_y(t,\gamma(t)\pm)$, and
$r^{''}_{yy}(t,\gamma(t)\pm)$ exist and are continuous as functions of $t\in [0,T)$. Then, for any $0\leq t<T$, we have \eqref{eq:peskir}.
\end{theorem}

\vspace{0.5cm}

For our purpose we need a more general formula, valid for multiple curves and local times. Such an extension of the result of Theorem \ref{thm-peskir} was announced in \cite{peskir} (see the Remark 2.3 therein) without proof. 

\begin{corollary}
\label{cor-peskirmulti}
Let $Y$ be a continuous $\R$-valued semimartingale. 

Let $I\in \N^*$. For each $1\leq i\leq I$, let $y_i:[0,T]\to\R$ be a continuous function of bounded variation,
and assume that $y_i(t)<y_j(t)$ for all $t\in [0,T]$ and all $1\leq i<j\leq I$. 

Let $r\in C(E)\cap\Big(\cap_{i=0}^I C^{1,2}(\overline{D^y_i})\Big)$.
Then, for any $0\leq t< T$,
$$
\begin{array}{lll}
r(t,Y_t)&=&\ds r(0,Y_0)+\int_0^t\frac 1 2(r'_t(s,Y_s+)+r'_t(s,Y_s-))ds
+\int_0^t\frac 1 2(r'_y(s,Y_s+)+r'_y(s,Y_s-) )dY_s\\
\\
&&\ds +\frac 1 2\int_0^tr''_{yy}(s,Y_s)\1_{\{Y_s\neq y_i(s),\,\forall 1\leq i\leq I\}}d\langle Y\rangle_s
+\frac 1 2 \sum_{i=1}^I\int_0^t(r'_y(s,Y_s+)-r'_y(s,Y_s-))dL^{y_i}_s(Y).\\
\end{array}
$$
The result remains valid if the curves $y_i$'s are of class $C^1$ and if
 $r\in C(E)\cap C^{1,2}(E^\circ\setminus \Delta_{\bf y})$ is such that for all $1\leq i\leq I$, the limits $r'_t(t, y_i(t)\pm)$, $r'_y(t,y_i(t)\pm)$, and
$r^{''}_{yy}(t,y_i(t)\pm)$ exist and are continuous as functions of $t\in [0,T)$. 

\end{corollary}
\begin{proof}
The proof is postponed to the Appendix.
\end{proof}

\subsection{Time inhomogeneous Markov processes, infinitesimal generator of the associated space-time process}
\label{ss:markov-inho}

The presentation of Markov processes, especially when coming to the time inhomogeneous case, varies slightly from one book to the other. Here we precise some definitions and concepts. We follow mainly \cite{RY} but we are also inspired by other references (\cite{kara}, \cite{wentzell}; see also \cite{bottcher}).

\vspace{0.2cm}

 Let $(\Omega,\cF,\P)$ be a probability space, $(\cF_t)_{t\in[0,T]}$ a filtration ($\cF_t\subset\cF$ for any $t\in [0,T]$) and 
consider $Z=(Z_t)_{t\in[0,T]}$ an adapted process defined on this probability space, taking values in a measurable space $(U,\cU)$.

We will say that $Z$ is an $(\cF_t)$-{\it Markov process} if for any $0\leq s\leq t\leq T$, and any $f\in C_b(U)$ we have
$$
\E[f(Z_t)\,|\,\cF_s]=\E[f(Z_t)\,|\,Z_s].$$
Denoting $\E^{s,x}(\cdot)=\E(\cdot\,|Z_s=x)$ and defining the operator $P_{s,t}$ by $P_{s,t}f(x)=\E^{s,x}[f(Z_t)]$,  for any $f\in C_b(U)$, any $x\in U$, we clearly have 
$
\E[f(Z_t)\,|\,\cF_s]=P_{s,t}f(Z_s)$. The family $(P_{s,t})_{0\leq s\leq t\leq T}$ is called the transition function of $Z$. We will say that $Z$ is a {\it time homogeneous Markov process} if $P_{s,t}=P_{0,t-s}$. In the opposite case it is called {\it time inhomogeneous}.

Now to fix ideas suppose the Markov process $Z$ is $\R$-valued, and denote $(P_{s,t})$ its transition function. Consider the associated $E$-valued {\it space-time process} $\tilde{Z}=((t,Z_t))_{t\in[0,T]}$. It is an exercise 
(\cite{RY}, Exercise III.1.10)
to check that for any $\varphi\in C_b(E)$ and any $0\leq s\leq t\leq T$,
\begin{equation*}
\E[\varphi(\tilde{Z}_t)\,|\,\cF_s]=P_{t-s}\varphi(\tilde{Z}_s)
\end{equation*}
with 
\begin{equation}
\label{eq:defPt}
\forall (s,x)\in E, \;\;\forall \varphi\in C_b(E), \;\; \forall\,0\leq t\leq T-s, \quad P_t\varphi(s,x)=P_{s,t+s}\varphi(t+s,x)=\E^{s,x}[\varphi(s+t,Z_{s+t})]
\end{equation}
(the value of $P_t\varphi(s,x)$ for $t+s>T$ is arbitrarily set to zero; see the forthcoming Remark \ref{rem:domL}). Thus the space-time process $\tilde{Z}$ is always a time homogeneous Markov process ($Z$ being time homogeneous or not), with transition function given by \eqref{eq:defPt}.

Note that the family $(P_{s,t})$ satisfies $P_{t,t}=\mathrm{Id}$ and thanks to the Markov property of $Z$  the evolution property
\begin{equation}
\label{eq:evo}
P_{s,u}\circ P_{u,t}=P_{s,t},\quad\forall 0\leq s\leq u\leq t\leq T.
\end{equation}

The family $(P_t)$ satisfies $P_0=\mathrm{Id}$ and thanks to the time homogeneous Markov property of $\tilde{Z}$  the semigroup property
\begin{equation}
\label{eq:semigroup}
P_s\circ P_t=P_{t+s},\quad \forall 0\leq s\leq T,\;\;\forall\,0\leq t\leq T-s.
\end{equation}

If the family $(P_{s,t})$ satisfies, in addition to \eqref{eq:evo}, that for any $f\in C_0(\R)$ we have $P_{s,t}f\in C_0(\R)$,
 $||P_{s,t}f||_\infty\leq ||f||_\infty$, $P_{s,t}f\geq 0$ if $f\geq 0$, and
 \begin{equation}
 \label{eq:contevo}
 \lim_{\substack{
 (s,t)\to(v,w)\\
 s\leq t\\}}||P_{s,t}f-P_{v,w}f||_\infty=0
 \end{equation}
it is called a {\it Feller evolution system}.

If the family $(P_t)$ satisfies, in addition to \eqref{eq:semigroup}, that for any $\varphi\in C_0(E)$, we have
$P_t\varphi\in C_0(E)$, 
$||P_t\varphi||_\infty\leq ||\varphi||_\infty$, $P_t\varphi\geq 0$ if $\varphi\geq 0$,
and  $\lim_{t\downarrow 0}||P_t\varphi-\varphi||_\infty=0$, then it is called a {\it Feller semigroup}.

\vspace{0.2cm}

We have the following result.

\begin{theorem}[\cite{bottcher}]
\label{thm:bottcher}
Let $Z$ be a Markov process with corresponding transitions $(P_{s,t})$. Let $(P_t)$ be the semigroup associated to the space-time process $\tilde{Z}$ of $Z$.  
Then the following statements are equivalent:

i) $(P_{s,t})$ is a Feller evolution system.

ii) $(P_t)$ is a Feller semigroup.
\end{theorem}

\begin{proof}
Note that our definition of the space-time process, which follows \cite{RY}, is a bit different from the one in \cite{bottcher},\cite{wentzell}, which is more canonical. But in fact, the families of operators $(P_{s,t})$ and $(P_t)$ that we have defined above, are exactly the same as the ones in \cite{bottcher},\cite{wentzell}. Therefore is suffices to adapt the proof of \cite{bottcher}, which is carried out on a infinite time interval, to the finite time interval case.
\end{proof}

We will say that $Z$ is a {\it Feller time inhomogeneous Markov process} if its corresponding evolution system $(P_{s,t})$ is Feller, or equivalently if the semigroup $(P_t)$ of the corresponding space-time process $\tilde{Z}$ is Feller (note that $\tilde{Z}$ is therefore a Feller process in the sense of \cite{RY}). We will focus on this latter point of view, because we believe it provides a more representative setting in order to describe the operators associated to a Feller time inhomogeneous Markov process $Z$. More precisely we will work at identifying the parabolic operator that is the infinitesimal generator of the space-time process $\tilde{Z}$.

\vspace{0.2cm}

At this point we recall the following definition.

\begin{definition}
Let $\tilde{Z}$ be a $E$-valued Feller process, with associated Feller semigroup $(P_t)$. A function $\varphi$ in $C_0=C_0(E)$ is said the belong to the domain $D(\cL)$ of the infinitesimal generator of $\tilde{Z}$ if the limit
\begin{equation}
\label{eq:defi-gene}
\cL \varphi=\lim_{t\downarrow 0}\frac 1 t (P_t\varphi-\varphi)
\end{equation}
exists in $C_0$. The operator $\cL:D(\cL)\to C_0$ thus defined is called the infinitesimal generator of the process $\tilde{Z}$ or of the semigroup $(P_t)$.

\end{definition}

In order to identify such infinitesimal generators we will use the following proposition.

\begin{proposition}
\label{prop-genemart}
Let $Z=(Z_t)_{t\in[0,T]}$ be a $\R$-valued Feller time inhomogeneous $(\cF_t)$-Markov process and let
 $\tilde{Z}=((t,Z_t))_{t\in[0,T]}$ be the $E$-valued corresponding space-time process. Assume $\tilde{Z}$ has generator $(\cL,D(\cL))$.

 If $\varphi\in C_0$, and if there exists a function $g\in C_0$ such that $M^{\varphi,g}=(M^{\varphi,g}_t)_{t\in[s,T]}$ defined by
$$
\forall t\in[s,T],\quad M^{\varphi,g}_t=\varphi(\tilde{Z}_t)-\varphi(\tilde{Z}_s)-\int_s^t g(\tilde{Z}_u)du$$
is a $(\cF_t)$-martingale under $\P^{s,x}$ (for any $(s,x)\in E$), then $\varphi\in D(\cL)$ and $\cL \varphi=g$.

\end{proposition}

\begin{proof}Here we adapt the proof of Proposition 
 VII.1.7 in \cite{RY} to the inhomogeneous case. 
Recall that the semigroup $(P_t)$ associated to $\tilde{Z}$ is defined by
$$
\forall (s,x)\in E, \;\;\forall \varphi\in C_0(E), \;\; \forall\,0\leq t\leq T-s, \quad P_t\varphi(s,x)=\E^{s,x}[\varphi(s+t,Z_{s+t})].
$$
 Let $(s,x)\in E$. Thanks to the hypothesis the process $M=(M_t)_{t\in [0,T-s]}$ defined by
$$
\forall t\in [0,T-s],\quad M_t=\varphi(s+t,Z_{s+t})-\varphi(s,Z_s)-\int_0^tg(s+u,Z_{s+u})du$$
is an $(\cF_t)$-martingale under $\P^{s,x}$. Taking the expectation under $\P^{s,x}$ we get
$$
P_t\varphi(s,x)-\varphi(s,x)-\int_0^tP_ug(s,x)du=0.$$
Thus we get
$$\big|\big| \frac 1 t(P_t\varphi-\varphi)-g \big|\big|_\infty=\big|\big| \frac 1 t \int_0^t(P_ug-g)du \big|\big|_\infty\leq \frac 1 t\int_0^t||P_ug-g||_\infty du$$
which goes to zero as $t$ goes to zero.
\end{proof}

\begin{remark}
In the sequel, for any $\R$-valued Markov process $Z$ the family $(P_t)$ will denote the semigroup associated with its space-time process $\tilde{Z}$. This will be clear from the context, and there will be no risk to confuse this semigroup with the one associated to $Z$, should this process be time homogeneous Markov (as $P_t$ will act on functions from $E$ to $\R$).
\end{remark}

\begin{remark}
For a time inhomogeneous diffusion we can expect that $\cL\varphi(t,\cdot)=(\partial_t+L_t)\varphi(t,\cdot)$, with $L_t$ a second order elliptic operator in the space variable. But in our case, with discontinuous coefficients and singular terms, $D(\cL)$ will not contain $C^{1,2}$(E) functions (cf Section \ref{sec:markov}).
\end{remark}

%\subsection{Parabolic transmission problems in divergence form}
%
%We will call a classical solution of the transmission problem in divergence form
%$(\mathcal{P}^\lambda_\mathrm{div,\Delta}(1,{\bf A}, {\bf B}))$, a function  $v(t,x)$ that is of class 
%$C(E)\cap C^{1,2}(E\setminus\Delta)$, is such that $v'_t(t,x\pm)$, $v'_{x}(t,x\pm)$ and $v''_{xx}(t,x\pm)$ exist for any $(t,x)\in\Delta$, and that satisfies
%$$
%(\mathcal{P}^\lambda_\mathrm{div,{\Delta}}(1,{\bf A},{\bf B}))
%\left\{
%\begin{array}{rcll}
%\big[v'_t+\dfrac{1}{2}\big({\bf A}v'_x\big)'_x+{\bf B}\,v'_x-\lambda v\big](t,x)&=&g(t,x)&\forall (t,x)\in E\setminus\Delta_{\bf x}\\
%\\
%{\bf A}(t,i+)v'_x(t,i+) &=& {\bf A}(t,i-)v'_x(t,i-)& \forall 1\leq i\leq I,\,\forall t\in[0,T)\;\;(\star) \\
%\\
%v(T,x)&=&f(x)&\forall x\in \R.\\
%\\
%\lim_{|x|\to\infty}|v(t,x)|&=&0&\forall t\in[0,T].
%\end{array}
%\right.
%$$

\section{Getting solutions by the mean of a space transform }
\label{sec:SDE}

\subsection{Main results}

In the sequel $W=(W_t)_{t\in[0,T]}$ will always denote some $(\cF_t)$-Brownian motion defined on some filtered probability space
$(\Omega,\cF,(\cF_t)_{t\in[0,T]},\P)$.

\vspace{0.1cm}

Our main results will be the followings: the first one (Proposition \ref{prop:transfo}) is a change of variable formula for time inhomogeneous SDEs with local time (it is thus more general than the formula stated in Theorem 3.1 of \cite{peskir}, but our assumptions are more restrictive). Assuming a solution $Y$ exists to the time inhomogeneous SDE with local time \eqref{eq:EDS-Ytransfo} below, Proposition~\ref{prop:transfo} gives the form of some transformed process
$\phi(t,Y_t)$. This formula will be used extensively in the sequel. To start with, it allows to prove Theorem \ref{thm-EDSTL}, that gives existence and uniqueness results for the solution $X=(X_t)_{t\in[0,T]}$ to  equation \eqref{eq:X} under some conditions on the coefficients $\sigma(t,x)$, $b(t,x)$, $\beta_i(t)$, $1\leq i\leq I$, and the curves $x_i(t)$. 
But Proposition \ref{prop:transfo} will be again used in Sections \ref{sec:feynman} and \ref{sec:markov}.

\begin{proposition}
\label{prop:transfo}
Let $I\in \N^*$. For each $1\leq i\leq I$, let $y_i:[0,T]\to\R$ be a function of class $C^1$,
and assume that $y_i(t)<y_j(t)$ for all $t\in [0,T]$ and all $1\leq i<j\leq I$.

Let $Y=(Y_t)_{0\leq t\leq T}$ be a continuous $\R$-valued semimartingale satisfying
\begin{equation}
\label{eq:EDS-Ytransfo}
dY_t=\bar{\sigma}(t,Y_t)dW_t+\bar{b}(t,Y_t)dt+\sum_{i=1}^I\bar{\beta}_i(t)dL^{y_i}_t(Y)
\end{equation}
where $\bar{\sigma},\bar{b}:[0,T]\times\R\to\R$  are some bounded functions, and the functions 
$\bar{\beta}_i:[0,T]\to(-1,1)$, $1\leq i\leq I$, are of class $C^1$.

Let
$\phi\in C(E)\cap C^{1,2}(E^\circ\setminus \Delta_{\bf y})$ such that for all $1\leq i\leq I$, the limits $\phi'_t(t, y_i(t)\pm)$, $\phi'_y(t,y_i(t)\pm)$, and
$\phi^{''}_{yy}(t,y_i(t)\pm)$ exist and are continuous as functions of $t\in [0,T)$.

Set $X_t=\phi(t,Y_t)$ for any $t\in[0,T]$. Then
\begin{equation}
\label{eq:Xprems}
\begin{array}{lll}
dX_t&=&(\bar{\sigma}\phi'_{y,\pm})(t,Y_t)dW_t+[\phi'_{t,\pm}+\bar{b}\phi'_{y,\pm}](t,Y_t)dt+\frac 1 2 (\bar{\sigma}^2\phi''_{yy})(t,Y_t)\1_{\{Y_t\neq y_i(t),\,\forall 1\leq i\leq I\}}dt\\
\\
&&+\sum_{i=1}^I[\vartriangle\phi'_y(t,y_i(t))+\bar{\beta}_i(t)\phi'_{y,\pm}(t,y_i(t))]\,dL^{y_i}_t(Y).
\end{array}
\end{equation}
Assume further that $\phi\in C(E)\cap\Big(\cap_{i=0}^I C^{1,2}(\overline{D^y_i})\Big)$ and
\begin{equation}
\label{eq:phicrois}
\phi'_y(t,y)>0 \quad\forall (t,y)\in E\setminus\Delta_{{\bf y}}
\end{equation}
and denote, for any $t\in[0,T]$, $\Phi(t,\cdot)=[\phi(t,\cdot)]^{-1}$  and 
\begin{equation}
\label{eq:xi-phi}
x_i(t)=\phi(t,y_i(t))
\end{equation}
 for all $1\leq i\leq I$. Then
\begin{equation}
\label{eq:Xphi}
dX_t=\sigma(t,X_t)dW_t+b(t,X_t)dt+\sum_{i=1}^I\beta_i(t)dL^{x_i}_t(X)
\end{equation}
with
$$
\begin{array}{ccl}
\sigma(t,x)&=&(\bar{\sigma}\phi'_{y,\pm})(t,\Phi(t,x))\\
\\
b(t,x)&=&[\phi'_{t,\pm}+\bar{b}\phi'_{y,\pm}](t,\Phi(t,x))+\frac 1 2 (\bar{\sigma}^2\phi''_{yy})(t,\Phi(t,x))\1_{\{x\neq x_i(t),\,\forall 1\leq i\leq I\}}\\
\end{array}
$$
and
\begin{equation}
\label{eq:newbetai}
\beta_i(t)=\frac{\vartriangle\phi'_y(t,y_i(t))+\bar{\beta}_i(t)\phi'_{y,\pm}(t,y_i(t))}{\phi'_{y,\pm}(t,y_i(t))+\bar{\beta}_i(t)\vartriangle\phi'_y(t,y_i(t))}
\end{equation}
for all $t\in[0,T]$ and all $x\in\R$.
\end{proposition}

\begin{remark}
\label{rem:xiC1}
Note that the curves $x_i$ defined by \eqref{eq:xi-phi} are $C^1$-functions, so that the local times terms in
\eqref{eq:Xphi} are still well defined.

In order to see that, let us focus on $x_1(t)=\phi(t,y_1(t))$. As $\phi$ is in $C(E)\cap\Big(\cap_{i=0}^I C^{1,2}(\overline{D^y_i})\Big)$ one has that 
$\phi$ restricted to $D^y_0$ coincides with a function $\phi_0\in C^{1,2}(E)$, and that $\phi$ restricted to $D^y_1
$ coincides with a function $\phi_1\in C^{1,2}(E)$. Thus, as $\phi$ is continuous, one has
$$
\phi_0(t,y_1(t))=\phi(t,y_1(t))=\phi_1(t,y_1(t)),\quad\forall t\in[0,T].$$
 In particular $x_1(t)=\phi_0(t,y_1(t))$ is as a composition of $C^1$-functions itself of class $C^1$.

\end{remark}

\begin{remark}
\label{rem:variante-newbetai}
Note that 
\begin{equation}
\label{eq:varbetai}
\vartriangle\phi'_y(t,y_i(t))+\bar{\beta}_i(t)\phi'_{y,\pm}(t,y_i(t))=\phi'_y(t,y_i(t)+)(1+\bar{\beta_i}(t))-\phi'_y(t,y_i(t)-)(1-\bar{\beta_i}(t))
\end{equation}
and that
$\phi'_{y,\pm}(t,y_i(t))+\bar{\beta}_i(t)\vartriangle\phi'_y(t,y_i(t))=\phi'_y(t,y_i(t)+)(1+\bar{\beta_i}(t))+\phi'_y(t,y_i(t)-)(1-\bar{\beta_i}(t))$, so that
the new coefficients $\beta_i(t)$ in Proposition \ref{prop:transfo} may be rewritten
\begin{equation}
\label{eq:variante-newbetai}
\beta_i(t)
=\frac{\phi'_y(t,y_i(t)+)(1+\bar{\beta_i}(t))-\phi'_y(t,y_i(t)-)(1-\bar{\beta_i}(t))}{\phi'_y(t,y_i(t)+)(1+\bar{\beta_i}(t))+\phi'_y(t,y_i(t)-)(1-\bar{\beta_i}(t))}.
\end{equation}

\end{remark}

\begin{remark}
Note that the result of Proposition \ref{prop:transfo} is a time inhomogeneous version of Proposition 3.1 in \cite{etore05a} (or equivalently 
Proposition 2.2.1 in \cite{etore06a}).
\end{remark}

\begin{theorem}
\label{thm-EDSTL}

Let $I\in \N^*$. For each $1\leq i\leq I$, let $x_i:[0,T]\to\R$ be a function of class $C^1$,
and assume that $x_i(t)<x_j(t)$ for all $t\in [0,T]$ and all $1\leq i<j\leq I$. 

Let $\sigma\in\Theta(m,M)$ and $b\in\Xi(M)$ for some $0<m<M<\infty$.

Assume that for each $1\leq i\leq I$, the function $\beta_i:[0,T]\to[k,\kappa]$ ($-1<k\leq\kappa<1$) is of class $C^1$, and that 
$|\beta_i'(t)|\leq M$ for any $t\in[0,T]$. 

Then the time inhomogeneous SDE with local time
\begin{equation*}
dX_t=\sigma(t,X_t)dW_t+b(t,X_t)dt+\sum_{i=1}^I\beta_i(t)dL^{x_i}_t(X),\quad t\in [0,T],\quad X_0=x_0
\end{equation*}
(i.e. equation \eqref{eq:X}) has a weak solution.

Assume further that $\sigma$ satisfies the $\mathbf{H}^{(x_i)}$ and $\mathbf{AJ}^{(x_i)}$-hypotheses.

Then the SDE \eqref{eq:X} has a unique strong solution (as it enjoys pathwise uniqueness).
\end{theorem}

\begin{remark}
The conditions of Theorem \ref{thm-EDSTL} have to be compared to the conditions in \cite{legall}. In particular, as in \cite{legall},
it is required that the $\beta_i$'s stay in $(-1,1)$. In view of the results in \cite{harrison-shepp} p. 312 at the end of Section~3 stated for the plain standard homogeneous Skew Brownian motion, it should be clear that if for some $1\leq i\leq I$, and some 
$(a,b)\subset [0,T]$, $(a,b)\neq\emptyset$, we have $|\beta_i(t)|>1$ for 
$t\in(a,b)$,
  then there is no possibility to ensure the existence of solutions to \eqref{eq:X} ${\mathbb P}$-a.s (see also \cite{RY} Chap. VI, Exercise 2.24 p. 246 where the equation is written with the right-hand sided local time instead of the symmetric local time). However in this case, it should be also possible to show that there exists an event $\tilde{\Omega}^{x_0}$ with ${\mathbb P}(\tilde{\Omega}^{x_0})<1$ such that solutions to \eqref{eq:X} exist on this event, namely these solutions to \eqref{eq:X} constructed in such a way that they do not hit the curve $t\mapsto x_i(t)$ during the time subinterval where $|\beta_i|>1$.

\end{remark}

\subsection{Proofs}
\label{ss-preuvessec3}

{\bf Proof of Proposition \ref{prop:transfo}.} Applying Corollary \ref{cor-peskirmulti} we get
\begin{equation*}
%\label{eq:phiX}
\begin{array}{lll}
dX_t&=&\phi'_{t,\pm}(t,Y_t)dt+\phi'_{y,\pm}(t,Y_t)dY_t
+\frac 1 2\phi''_{yy}(t,Y_t)\bar{\sigma}^2(t,Y_t)\1_{\{Y_t\neq y_i(t),\;1\leq i\leq I \}}dt\\
\\
&&+\sum_{i=1}^I\vartriangle\phi'_y(t,Y_t)dL^{y_i}_t(Y)\\
&=&(\bar{\sigma}\phi'_{y,\pm})(t,Y_t)dW_t+[\phi'_{t,\pm}+\bar{b}\phi'_{y,\pm}](t,Y_t)dt+\frac 1 2 (\bar{\sigma}^2\phi''_{yy})(t,Y_t)\1_{\{Y_t\neq y_i(t),\,\forall 1\leq i\leq I\}}dt\\
\\
&&+\sum_{i=1}^I[\vartriangle\phi'_y(t,y_i(t))+\bar{\beta}_i(t)\phi'_{y,\pm}(t,y_i(t))]\,dL^{y_i}_t(Y),\\
\end{array}
\end{equation*}
where we have used the fact that $dL^{y_i}_t(Y)=\1_{Y_t=y_i(t)}dL^{y_i}_t(Y)$, for any $1\leq i\leq I$.

Thus, the first part of Proposition \ref{prop:transfo} is proved. To prove the second part it suffices to use the following lemma.

\begin{lemma}
\label{lem-tempslocXY}
In the above context and under \eqref{eq:phicrois} we have
$$ d L^{y_i}_t(Y)=\frac{ dL^{x_i}_t(X)}{\phi'_{y,\pm}(t,y_i(t))+\bar{\beta}_i(t)\vartriangle\phi'_y(t,y_i(t))},\quad\forall 1\leq i\leq I.$$
\end{lemma}

\begin{proof}
Let $1\leq i\leq I$. On one side we  apply the symmetric Tanaka formula to the process $X-x_i$. We get
\begin{equation}
\label{eq:tanaka}
\begin{array}{lll}
d|X_t-x_i(t)|&=&\sgn(X_t-x_i(t))d(X_t-x_i(t))+dL^0_t(X-x_i)\\
\\
&=&dL^{x_i}_t(X)-\sgn(Y_t-y_i(t))dx_i(t)\\
\\
&&+\sgn(Y_t-y_i(t))\sigma(t,\phi(t,Y_t))dW_t+\sgn(Y_t-y_i(t))b(t,\phi(t,Y_t))dt\\
\\
&&+\sum_{j\neq i}\sgn(Y_t-y_i(t))[\vartriangle\phi'_y(t,y_j(t))+\bar{\beta}_j(t)\phi'_{y,\pm}(t,y_j(t))]\,dL^{y_j}_t(Y).\\
\end{array}
\end{equation}
In the above expression we have first used the fact that $\sgn(X_t-x_i(t))=\sgn(Y_t-y_i(t))$ for any $t\in [0,T]$ (as $\phi(t,\cdot)$ is stricly increasing). Second we have used the fact that with the symmetric sign function we have
$$
\sgn(X_t-x_i(t))=\sgn(Y_t-y_i(t))=0 \text{ for any } t\in [0,T] \text{ s.t. } Y_t=y_i(t).$$
Third we have used $dL^{y_j}_t(Y)=\1_{Y_t=y_j(t)}dL^{y_j}_t(Y)$, for any $1\leq j\leq I$.

On the other side we may apply the first part of Proposition \ref{prop:transfo} (that is equation \eqref{eq:Xprems}; we stress that at this stage this part is already proved) with the semimartingale $Y$ and the function $\zeta:(t,y)\mapsto |\phi(t,y)-x_i(t)|$. We get
\begin{equation}
\label{eq:peskir2}
\begin{array}{lll}
d|X_t-x_i(t)|&=&d|\phi(t,Y_t)-x_i(t)|\\
\\
&=&(\bar{\sigma}\zeta'_{y,\pm})(t,Y_t)dW_t+[\zeta'_{t,\pm}+\bar{b}\zeta'_{y,\pm}](t,Y_t)dt+\frac 1 2 (\bar{\sigma}^2\zeta''_{yy})(t,Y_t)\1_{\{Y_t\neq y_j(t),\,\forall 1\leq j\leq I\}}dt\\
\\
&&+\sum_{j=1}^I[\vartriangle\zeta'_y(t,y_j(t))+\bar{\beta}_j(t)\zeta'_{y,\pm}(t,y_j(t))]\,dL^{y_j}_t(Y)\\
\\
&=&-\sgn(Y_t-y_i(t))dx_i(t)+\sgn(Y_t-y_i(t))\sigma(t,\phi(t,Y_t))dW_t+\sgn(Y_t-y_i(t))b(t,\phi(t,Y_t))dt\\
\\
&&+\sum_{j\neq i}\sgn(Y_t-y_i(t))[\vartriangle\phi'_y(t,y_j(t))+\bar{\beta}_j(t)\phi'_{y,\pm}(t,y_j(t))]\,dL^{y_j}_t(Y)\\
\\
&&+[\phi'_{y,\pm}(t,y_i(t))+\bar{\beta}_i(t)\vartriangle\phi'_y(t,y_i(t))]dL^{y_i}_t(Y).\\
\end{array}
\end{equation}
In \eqref{eq:peskir2} we have used several facts. 

First, as $x_i(t)$ is of class $C^1$ (Remark \ref{rem:xiC1}), we have
$$\zeta'_{t,\pm}(t,y)=-\sgn(y-y_i(t))x_i'(t)+\frac 1 2(\sgn_+(y-y_i(t))\phi'_t(t,y+)+\sgn_-(y-y_i(t))\phi'_t(t,y-)),$$
where $\sgn_\pm$ denote the right and left sign functions.

In the same manner we have that
$$
\zeta'_{y,\pm}(t,y)=\frac 1 2\big(\sgn_+(y-y_i(t))\phi'_y(t,y+)+\sgn_-(y-y_i(t))\phi'_y(t,y-)\big)$$
and
$$
\zeta''_{yy}(t,y)\1_{\{y\neq y_j(t),\,\forall 1\leq j\leq I\}}=\sgn(y-y_i(t))\phi''_{yy}(t,y)\1_{\{y\neq y_j(t),\,\forall 1\leq j\leq I\}}
$$
Second, focusing for a while on $(\bar{b}\zeta'_{y,\pm})(t,Y_t)dt$, we claim that this is equal to
$\sgn(Y_t-y_i(t))(\bar{b}\phi'_{y,\pm})(t,Y_t)dt$ (we recall that $\sgn$ denotes the symmetric sign function). 
Indeed, using Exercise VI.1.15 in \cite{RY} (some extension of the occupation times formula), to the semimartingale $Y-y_i$, one can show that
\begin{equation}
\label{eq:supp1}
\int_0^T\1_{Y_t=y_i(t)}dt=0\quad\P-\text{a.s.}
\end{equation}
So that (a.s.)
$$
\begin{array}{lll}
(\bar{b}\zeta'_{y,\pm})(t,Y_t)dt&=&\frac 1 2\bar{b}(t,Y_t)\big(\sgn_+(Y_t-y_i(t))\phi'_y(t,Y_t+)+\sgn_-(Y_t-y_i(t))\phi'_y(t,Y_t-)\big)\1_{Y_t\neq y_i(t)}dt\\
\\
&=&
\sgn(Y_t-y_i(t))(\bar{b}\phi'_{y,\pm})(t,Y_t)\1_{Y_t\neq y_i(t)}dt=\sgn(Y_t-y_i(t))(\bar{b}\phi'_{y,\pm})(t,Y_t)dt.
\end{array}
$$
In the same manner one can see that 
$$\frac 1 2 (\bar{\sigma}^2\zeta''_{yy})(t,Y_t)\1_{\{Y_t\neq y_j(t),\,\forall 1\leq j\leq I\}}dt
=\frac 1 2 \sgn(Y_t-y_i(t))(\bar{\sigma}^2\phi''_{yy})(t,Y_t)\1_{\{Y_t\neq y_j(t),\,\forall 1\leq j\leq I\}}dt,$$
and, using $dx_i(t)=x_i'(t)dt$, that
$$\zeta'_{t,\pm}(t,Y_t)dt=\sgn(Y_t-y_i(t))\phi'_{t,\pm}(t,Y_t)dt-\sgn(Y_t-y_i(t))dx_i(t).$$
Third, using Itô isometry in order to use the above arguments one may also show that
$(\bar{\sigma}\zeta'_{y,\pm})(t,Y_t)dW_t=\sgn(Y_t-y_i(t))(\bar{\sigma}\phi'_{y,\pm})(t,Y_t)dW_t$.

To sum up, using the definition of $b(t,x)$, $\sigma(t,x)$ we have that
$$
\begin{array}{l}
(\bar{\sigma}\zeta'_{y,\pm})(t,Y_t)dW_t+[\zeta'_{t,\pm}+\bar{b}\zeta'_{y,\pm}](t,Y_t)dt+\frac 1 2 (\bar{\sigma}^2\zeta''_{yy})(t,Y_t)\1_{\{Y_t\neq y_j(t),\,\forall 1\leq j\leq I\}}dt\\
\\
=-\sgn(Y_t-y_i(t))dx_i(t)+\sgn(Y_t-y_i(t))\sigma(t,\phi(t,Y_t))dW_t+\sgn(Y_t-y_i(t))b(t,\phi(t,Y_t))dt.
\end{array}
$$
Fourth (we are now turning to the local time terms) for $j<i$, we have $y_j(t)<y_i(t)$ and thus $\phi(t,y_j(t))<x_i(t)$ for any $t\in[0,T]$, which
leads to 
$$[\vartriangle\zeta'_y(t,y_j(t))+\bar{\beta}_j(t)\zeta'_{y,\pm}(t,y_j(t))]=-[\vartriangle\phi'_y(t,y_j(t))+\bar{\beta}_j(t)\phi'_{y,\pm}(t,y_j(t))].$$
Using $dL^{y_j}_t(Y)=\1_{Y_t=y_j(t)}dL^{y_j}_t(Y)$ we then get that
$$
[\vartriangle\zeta'_y(t,y_j(t))+\bar{\beta}_j(t)\zeta'_{y,\pm}(t,y_j(t))]\,dL^{y_j}_t(Y)
=\sgn(Y_t-y_i(t))[\vartriangle\phi'_y(t,y_j(t))+\bar{\beta}_j(t)\phi'_{y,\pm}(t,y_j(t))]\,dL^{y_j}_t(Y)$$
We have the same result for $j>i$ (plus sign replaces minus sign).

Fifth, we now examine what happens for $j=i$.  The crucial fact is that because of the different sign of $\phi(t,y_i(t)\pm)-x_i(t)$ we have
$$
[\vartriangle\zeta'_y(t,y_i(t))+\bar{\beta}_i(t)\zeta'_{y,\pm}(t,y_i(t))]=[\phi'_{y,\pm}(t,y_i(t))+\bar{\beta}_i(t)\vartriangle\phi'_y(t,y_i(t))].$$
Therefore \eqref{eq:peskir2}.

Comparing \eqref{eq:tanaka} and \eqref{eq:peskir2} we get the desired result.
\end{proof}
\vspace{0.5cm}

{\bf Proof of Theorem \ref{thm-EDSTL}.}
Inspired by \cite{legall}, we will use the following bijection in space $r(t,\cdot)$ (for any $t\in [0,T]$),
that we now define.

For any $t\in [0,T]$ we define
\begin{equation}
\label{eq:def-mu}
\mu(t,x)=\prod_{x_i(t)\leq x}\frac{1-\beta_i(t)}{1+\beta_i(t)}
\end{equation}

(with the convention that $\mu(t,x)=1$ for any $x<x_1(t)$).

Let then 
\begin{equation}
\label{eq:defi-R}
R(t,x)=\int_{x_1(t)}^x\mu(t,z)dz.
\end{equation}

As $\mu(t,z)$ is strictly positive for any $z\in\R$ the function $R(t,\cdot)$ is strictly increasing. Thus we can define
\begin{equation}
\label{eq:defi-r}
r(t,y)=\big[ R(t,\cdot) \big]^{-1}(y).
\end{equation}
For any $1\leq i\leq I$ we define
\begin{equation}
\label{eq:def-yi}
y_i(t)=R(t,x_i(t))
\end{equation}
(note that $y_1\equiv 0$). It is easy to check that 
\begin{equation}
\label{eq:defi-r2}
r(t,y)=\int_0^y\alpha(t,z)dz+x_1(t)
\end{equation}
 with
\begin{equation}
\label{eq:def-alpha}
\alpha(t,y)=\prod_{y_i(t)\leq y}\frac{1+\beta_i(t)}{1-\beta_i(t)}
\end{equation}
(with  $\alpha(t,y)=1$ for any $y<y_1(t)$).
Note that  the function $r(t,\cdot)$ is strictly increasing too.

\vspace{0.2cm}
Let us check that $R$ is in $C(E)\cap\Big(\cap_{i=0}^I C^{1,2}(\overline{D^x_i})\Big)$ and that
$r$ is in $C(E)\cap\Big(\cap_{i=0}^I C^{1,2}(\overline{D^y_i})\Big)$. We focus on $R(t,x)$, as the computations are similar for $r(t,y)$.

Using \eqref{eq:def-mu}\eqref{eq:defi-R} it is easy to check that $R(t,x)$ coincides on $D^x_0$ with the function
$R_0(t,x)=x-x_1(t)$. On $D^x_i$, $1\leq i\leq I$, it coincides with the function
$$
R_i(t,x)=\sum_{j=1}^{i-1}\Big\{\prod_{k\leq j}\frac{1-\beta_k(t)}{1+\beta_k(t)}\Big\}(x_{j+1}(t)-x_j(t))
+\Big\{\prod_{k\leq i}\frac{1-\beta_k(t)}{1+\beta_k(t)}\Big\}(x-x_i(t)).$$
Obviously, all the functions $R_i(t,x)$, $0\leq i\leq I$ are in $C^{1,2}(E)$, and thus we see that $R(t,x)$ is in $\cap_{i=0}^I C^{1,2}(\overline{D^x_i})$.

To see that $R(t,x)$ is in $C(E)$ it remains to prove that it is continuous at any point $(t_0,x_0)\in\Delta_{\bf x}$. For such a point we have
$(t_0,x_0)=(t_0,x_i(t_0))$, for some $t_0\in[0,T]$ and some $1\leq i\leq I$. But, together with the relationship
$$
R(t_0,x_i(t_0))=R_{i-1}(t_0,x_i(t_0))=R_i(t_0,x_i(t_0))$$
the continuity of $R_{i-1}$ and $R_i$ then yields the desired result. 
Thus, $R$ is indeed in $C(E)\cap\Big(\cap_{i=0}^I C^{1,2}(\overline{D^x_i})\Big)$.
Note that this implies that the $y_i$'s defined by \eqref{eq:def-yi} are of class $C^1$ (by the same arguments as in Remark~\ref{rem:xiC1}).

\vspace{0.3cm}

We then set
\begin{equation}
\label{eq:defi-sigma}
\bar{\sigma}(t,y)=\frac{\sigma(t,r(t,y))}{r'_{y,\pm}(t,y)}\quad\text{and}\quad
\bar{b}(t,y)=\frac{b(t,r(t,y))}{r'_{y,\pm}(t,y)}-\frac{r'_{t,\pm}(t,y)}{r'_{y,\pm}(t,y)}
\end{equation}

It is easy to check that $\bar{\sigma}\in\Theta(\bar{m},\bar{M})$ and $\bar{b}\in\Xi(\bar{M})$ for some 
$0<\bar{m}<\bar{M}<\infty$.

From now on the starting point  $x_0\in\R$ is fixed.
By Corollary \ref{cor-legall} we have the existence of a weak solution~$Y$ to
\begin{equation}
\label{eq:EDSY}
dY_t=\bar{\sigma}(t,Y_t)dW_t+\bar{b}(t,Y_t)dt,\quad Y_0=R(0,x_0).
\end{equation}
We wish now to use the second part of Proposition \ref{prop:transfo}, with the function $r(t,y)$ and the process~$Y$ (and the curves $y_i$).
Note that by construction we have
$$
(\bar{\sigma}r'_{y,\pm})(t,R(t,x))=\sigma(t,x),
$$
$$
[r'_{t,\pm}+\bar{b}r'_{y,\pm}](t,R(t,x))+\frac 1 2 (\bar{\sigma}^2r''_{yy})(t,R(t,x))\1_{\{x\neq x_i(t),\,\forall 1\leq i\leq I\}}
=b(t,x),
$$
(we have used in particular $r''_{yy}\equiv 0$ in the above expression) and
$$
\frac{\vartriangle r'_y(t,y_i(t))}{r'_{y,\pm}(t,y_i(t))}=
\frac{\prod_{j<i}\frac{1+\beta_j(t)}{1-\beta_j(t)} \frac 1 2\big( \frac{1+\beta_i(t)}{1-\beta_i(t)} -1  \big) }
{\prod_{j<i}\frac{1+\beta_j(t)}{1-\beta_j(t)} \frac 1 2\big( \frac{1+\beta_i(t)}{1-\beta_i(t)} +1  \big)  }
=\frac{ \frac{2\beta_i(t)}{1-\beta_i(t)} }{ \frac{2}{1-\beta_i(t)}}=\beta_i(t)
$$
(here we have computed \eqref{eq:newbetai} using the fact that there is no local time term in \eqref{eq:EDSY}).

So that by setting
\begin{equation}
\label{eq:XrY}
X_t=r(t,Y_t),\quad\forall t\in[0,T]
\end{equation}
we immediately see by Proposition \ref{prop:transfo} that $X$ is a weak solution to \eqref{eq:X}.

In order to prove the last part of the theorem, we first notice that $\bar{\sigma}$ satisfies the $\mathbf{H}^{(y_i)}$ and $\mathbf{AJ}^{(x_i)}$-hypotheses.
Thus \eqref{eq:EDSY} enjoys pathwise uniqueness (Corollary \ref{cor-legall}).
Assume $X'$ is a second solution to \eqref{eq:X}, then we could show that $Y'_t=R(t,X'_t)$ is a solution to
\eqref{eq:EDSY}. Thus, using the pathwise uniqueness property of \eqref{eq:EDSY}, we would show that pathwise uniqueness holds for
\eqref{eq:X}. Therefore Theorem \ref{thm-EDSTL} is proved.

\section{Feynman-Kac formula: link with a parabolic transmission problem}
\label{sec:feynman}

Assume the curves $x_i$, $1\leq i\leq I$ and the coefficients $\beta_i$, $1\leq i\leq I$ are as in Theorem \ref{thm-EDSTL}, $b$ is in 
$\Xi(M)\cap C(E\setminus\Delta_{\bf x})$, and $\sigma$ is in $\Theta(m,M)\cap C(E\setminus\Delta_{\bf x})$.

For $\lambda\geq 0$, a source term $g\in C^{}_c(E)$ and a terminal condition $f\in C_0(\R)\cap L^2(\R)$, we will call a classical solution of the parabolic transmission problem $(\cP^\lambda_{\Delta_{\bf x}}(\sigma,b,\beta))$ a function $u(t,x)$ that is of class 
$C(E)\cap C^{1,2}(E^\circ\setminus\Delta_{\bf x})$, is such that for all $1\leq i\leq I$ the limits $u'_t(t,x_i(t)\pm)$, $u'_{x}(t,x_i(t)\pm)$ 
and $u''_{xx}(t,x_i(t)\pm)$ exist and are continuous as functions of $t\in [0,T)$,
 and that satisfies
$$
(\mathcal{P}^\lambda_{\Delta_{\bf x}}(\sigma,b,\beta))
\left\{
\begin{array}{rcll}
\big[u'_t+\frac 1 2\sigma^2u''_{xx}+b\,u'_x-\lambda u\big](t,x)&=&g(t,x)&\forall (t,x)\in E^\circ\setminus\Delta_{\bf x} \\
\\
(1+\beta_i(t))u'_x(t,x_i(t)+) &=& (1-\beta_i(t))u'_x(t,x_i(t)-)& \forall 1\leq i\leq I,\,\forall t\in[0,T)\;\;(\star) \\
\\
u(T,x)&=&f(x)&\forall x\in \R.\\
\\
\lim_{|x|\to\infty}|u(t,x)|&=&0&\forall t\in[0,T].
\end{array}
\right.
$$

In particular we stress that the first and second line of this system of equations are required to hold in the classical sense, i.e. pointwise.

The question whether a classical solution $u(t,x)$ exists to $(\mathcal{P}^\lambda_{\Delta_{\bf x}}(\sigma,b,\beta))$ will be discussed in Section
\ref{sec:EDP} (see Theorem \ref{thm:sol-classique}), with the help of an equivalent formulation of this parabolic transmission problem, in a more divergence-like form 
(Subsection \ref{ssec:eq-form}).  The condition $(\star)$ will be called the {\it transmission condition} in the sequel.

For the moment, assuming in this section the existence of such a solution $u(t,x)$, we draw some consequences on the solution $X$ of \eqref{eq:X}: we have a Feynman-Kac formula linking $X$ and $u(t,x)$. We will see in Section \ref{sec:markov} that the properties of $u(t,x)$ allow to say more on $X$: we can prove
that $X$ is a Feller time inhomogeneous Markov process and identify the infinitesimal generator of the space-time process~$\tilde{X}$.

We have the following result.

\begin{theorem}
\label{thm:feynman}
Any classical solution $u(t,x)$ of  $(\mathcal{P}^\lambda_{\Delta_{\bf x}}(\sigma,b,\beta))$ admits the stochastic representation
\begin{equation*}
u(t,x)=\E^{t,x}\Big[ f(X_T)e^{-\lambda(T-t)}-\int_t^Tg(s,X_s)e^{-\lambda(s-t)}ds \Big]
\end{equation*}
where $X$ is the solution to \eqref{eq:X}; in particular such a classical solution $u(t,x)$ is unique.
\end{theorem}

\begin{remark}
The uniqueness of $u(t,x)$ in Theorem \ref{thm:feynman} comes from the uniqueness in law of the weak solution~$X$ (see Subsection \ref{ss:pbmart}).
\end{remark}

\begin{proof}
We will follow the lines of the proof of Theorem~5.7.6 in \cite{kara}, and use our Proposition \ref{prop:transfo} in the computations.
Let $t\in[0,T)$. Applying Proposition \ref{prop:transfo} and equation \eqref{eq:varbetai} we get for any $s\in[t,T)$,
\begin{equation}
\label{eq:feyn1}
\begin{array}{lll}
u(s,X_s)e^{-\lambda(s-t)}-u(t,X_t)
&=&\int_t^su'_{x,\pm}(v,X_v)e^{-\lambda(v-t)}\sigma(v,X_v)dW_v\\
\\
&&+\int_t^se^{-\lambda(v-t)}\big[ u'_{t,\pm}+bu'_{x,\pm} -\lambda u \big](v,X_v)dv\\
\\
&&+\frac 1 2 \int_s^te^{-\lambda(v-s)}u''_{xx}(v,X_v)\sigma^2(v,X_v)\1_{\{X_v\neq x_i(v),1\leq i\leq I\}}dv\\
\\
&&+\frac 1 2\sum_{i=1}^I\int_t^s\big[ (1+\beta_i(v))u'_x(v,X_v+) - (1-\beta_i(v))u'_x(v,X_v-) \big]dL^{x_i}_v(X)\\
\\
&=&\int_t^su'_{x,\pm}(v,X_v)e^{-\lambda(v-t)}\sigma(v,X_v)dW_v+\int_t^se^{-\lambda(v-t)}g(v,X_v)dv\\
\end{array}
\end{equation}
where we have first used the transmission condition $(\star)$ satisfied by $u(t,x)$.

 Second we have used the fact that 
\begin{equation}
\label{eq:int-temps-lebeg-zero}
\forall 1\leq i\leq I,\quad \quad\int_0^T\1_{X_t=x_i(t)}dt=0\quad\P-\text{a.s.}
\end{equation}
so that for example ($\P$-a.s.)
\begin{equation}
\label{eq:phit}
e^{-\lambda(v-t)}u'_{t,\pm}(v,X_v)dv=\1_{\{X_v\neq x_i(v),\,\forall 1\leq i\leq I\}}e^{-\lambda(v-t)}u'_t(v,X_v)dv.
\end{equation}
To see that \eqref{eq:int-temps-lebeg-zero} holds for some $1\leq i\leq I$, one uses the same arguments as in the proof of  Proposition~\ref{prop:transfo}. Then to get 
\eqref{eq:phit} it suffices to notice that $\1_{\{X_t\neq x_i(t),\,\forall 1\leq i\leq I\}}=\prod_{i=1}^I\1_{X_t\neq x_i(t)}$ and that 
\eqref{eq:int-temps-lebeg-zero} implies $H(t)\1_{X_t\neq x_i(t)}dt=H(t)dt$ (a.s.) for any integrable process $H$ and any $1\leq i\leq I$.
By the same arguments one can see that 
$$
\begin{array}{l}
e^{-\lambda(v-t)}\big[ u'_{t,\pm}+bu'_{x,\pm}  -\lambda u \big](v,X_v)dv+e^{-\lambda(v-t)} (\frac{\sigma^2}{2}u''_{xx})(v,X_v)
\1_{\{X_v\neq x_i(v),1\leq i\leq I\}}dv\\
\\
=e^{-\lambda(v-t)}g(v,X_v)\1_{\{X_v\neq x_i(v),1\leq i\leq I\}}dv\\
\\
=e^{-\lambda(v-t)}g(v,X_v)dv.\\
\end{array}
$$
Therefore \eqref{eq:feyn1} holds.

We introduce the sequence of stopping times $(\tau_n)$ defined by
$\tau_n=\inf\{s\geq t: |X_s|\geq n\}$ for any $n\in\N$. Taking the expectation $\E^{t,x}(\cdot)$ of \eqref{eq:feyn1} with $s=(T-\delta)\wedge\tau_n$ 
($\delta>0$ is sufficiently small)
we get
\begin{equation}
\label{eq:feyn2}
\begin{array}{lll}
u(t,x)
&=&\E^{t,x}\Big[ u(T-\delta,X_{T-\delta})e^{-\lambda(T-\delta-t)}\,\1_{\tau_n>T-\delta}\Big]+\E^{t,x}\Big[ u(\tau_n,X_{\tau_n})e^{-\lambda(\tau_n-t)}\,\1_{\tau_n\leq T-\delta}\Big]\\
\\
&&-\E^{t,x}\Big[ \int_t^{(T-\delta)\wedge\tau_n}g(s,X_s)e^{-\lambda(s-t)}\,ds\,\Big].\\
\end{array}
\end{equation}

To conclude the proof we may show by dominated convergence that, as $n\to\infty$ and
$\delta\downarrow 0$ the quantity
 $\E^{t,x}\Big[ u(T-\delta,X_{T-\delta})e^{-\lambda(T-\delta-t)}\,\1_{\tau_n>T-\delta}\Big]$ converges to
$\E^{t,x}\Big[ f(X_T)e^{-\lambda(T-t)}\Big]$,
$\E^{t,x}\Big[ \int_t^{(T-\delta)\wedge\tau_n}g(s,X_s)e^{-\lambda(s-t)}\,ds\,\Big]$
converges to
$\E^{t,x}\Big[ \int_t^{T}g(s,X_s)e^{-\lambda(s-t)}\,ds\,\Big]$,
and finally
$\E^{t,x}\Big[ u(\tau_n,X_{\tau_n})e^{-\lambda(\tau_n-t)}\,\1_{\tau_n\leq T-\delta}\Big]$
converges to zero (we stress the fact that here, as $u$ is in $C_0(E)$ it is bounded; this is because we have chosen to deal in the parabolic problem with a terminal condition vanishing at infinity; this lightens some technical aspects of the proof of Theorem~5.7.6 in \cite{kara}).
\end{proof}

\section{Parabolic transmission problem with time-dependent coefficients}
\label{sec:EDP}

\subsection{Equivalent formulation in divergence like form and getting cylindrical subdomains by the mean of a space transform}
\label{ssec:eq-form}

Assume that we have curves $x_i$, $1\leq i\leq I$ satisfying the same assumptions as in Theorem \ref{thm-EDSTL}.
Let us consider coefficients $\rho\in\Theta(m',M')\cap C(E\setminus \Delta_{\bf x})$,
$a\in\Theta(m',M')\cap C^{0,1}(E\setminus \Delta_{\bf x})$,
 and a coefficient $B\in\Xi(M')\cap C(E\setminus \Delta_{\bf x})$ (for some $0<m'<M'<\infty$).

For $\lambda\geq 0$, a source  term  $g\in C^{}_c(E)$ and a terminal condition $f\in C_0(\R)\cap L^2(\R)$, we will call a classical solution 
of the transmission problem in divergence form
$(\mathcal{P}^\lambda_{\mathrm{div},\Delta_{\bf x}}(\rho,a,B))$, a function  $u(t,x)$ that is of class 
$C(E)\cap C^{1,2}(E^\circ\setminus\Delta_{{\bf x}})$, 
is such that for all $1\leq i\leq I$ the limits $u'_t(t,x_i(t)\pm)$, $u'_{x}(t,x_i(t)\pm)$ 
and $u''_{xx}(t,x_i(t)\pm)$ exist and are continuous as functions of $t\in [0,T)$,  and that satisfies
$$
(\mathcal{P}^\lambda_\mathrm{div,\Delta_{\bf x}}(\rho,a,B))
\left\{
\begin{array}{rcll}
\big[u'_t+\dfrac \rho 2\big(au'_x\big)'_x+B\,u'_x-\lambda u\big](t,x)&=&g(t,x)&\forall (t,x)\in E^\circ\setminus\Delta_{\bf x} \\
\\
a(t,x_i(t)+)u'_x(t,x_i(t)+) &=& a(t,x_i(t)-)u'_x(t,x_i(t)-)& \forall 1\leq i\leq I,\,\forall t\in[0,T)\;\;(\star) \\
\\
u(T,x)&=&f(x)&\forall x\in \R.\\
\\
\lim_{|x|\to\infty}|u(t,x)|&=&0&\forall t\in[0,T].
\end{array}
\right.
$$

For any $\rho,a,B$ with 
\begin{equation}
\label{eq:triplet}
\rho a =\sigma^2, \quad a(t,x_i(t)\pm)=p_i(t)(1\pm\beta_i(t)),\;\; \forall 1\leq i\leq I, \;\;\forall t\in[0,T)\quad \text{and }
B=b-\rho\,a'_{x,\pm}/2,
\end{equation}
 it is clear that a classical solution to $(\mathcal{P}^\lambda_\mathrm{div,\Delta_{\bf x}}(\rho,a,B))$ is a classical solution to 
$(\cP^\lambda_{\Delta_{\bf x}}(\sigma,b,\beta))$ (here  $p_i(t)$ is a non zero multiplicative factor that depends on 
$1\leq i\leq I$, $t\in[0,T)$). One may for example choose
for any~$(t,x)\in E$
\begin{equation}
\label{eq:defi-arhoB}
a(t,x)=\prod_{x_i(t)\leq x}\frac{1+\beta_i(t)}{1-\beta_i(t)},\quad
\rho(t,x)=\dfrac{\sigma^2(t,x)}{a(t,x)},\quad
B(t,x)=b(t,x)
\end{equation}
(Note that here $p_i(t)=\dfrac{1}{1-\beta_i(t)}\,\prod_{j<i}\dfrac{1+\beta_j(t)}{1-\beta_j(t)}$).
Note that the presence of the variable coefficient $\rho(t,x)$ is due to the fact that the coefficient $\sigma(t,x)$ has been chosen independently from the $\beta_i(t)$'s.
Note also that a convenient triple $(\rho,a,B)$ is not unique - indeed if $(\rho,a,B)$ satisfies \eqref{eq:triplet}, for any $c>0$ the triplet
$(c\rho,a/c,B)$ will also do, with of course different multiplicative factors $p_i(t)$.

Conversely, it is always possible to pass from a transmission problem in the form $(\mathcal{P}^\lambda_\mathrm{div,\Delta_{\bf x}}(\rho,a,B))$  to another one in the form 
$(\cP^\lambda_{\Delta_{\bf x}}(\sigma,b,\beta))$, by setting in particular 
\begin{equation}
\label{eq:betai-poura}
\beta_i(t)=\frac{a(t,x_i(t)+)-a(t,x_i(t)-)}{a(t,x_i(t)+)+a(t,x_i(t)-)}.
\end{equation}

In fact, in the PDE litterature, parabolic transmission problems are classically studied in the purely divergence-like form  of
$(\mathcal{P}^\lambda_\mathrm{div,\Delta_{\bf x}}(\rho\equiv 1,a,B))$. Up to our knowledge fewer studies exist in the non divergence form $(\cP^\lambda_{\Delta_{\bf x}}(\sigma,b,\beta))$. 
The aim of this section
is to present some known results on the problem\\
 $(\mathcal{P}^\lambda_\mathrm{div,\Delta}(\rho\equiv 1,a,B))$, 
and to derive new ones for the general case ($\rho\neq 1$). So that we will finally get results for the problem $(\cP^\lambda_{\Delta_{\bf x}}(\sigma,b,\beta))$ (see Theorem \ref{thm:sol-classique} in the conclusion of this section).
\vspace{0.2cm}

In the case $\rho\equiv 1$, the transmission problem in divergence form $(\mathcal{P}^\lambda_\mathrm{div,\Delta_{\bf x}}(\rho,a,B))$  is well studied in the PDE litterature, concerning the existence and uniqueness of weak solutions (see the forthcoming Subsection~\ref{ss:sol-faible} for a definition of weak solution). We can refer for instance to \cite{lady}, \cite{lions-magenes}, \cite{lieberman}, for the study of weak solutions under the general assumption of uniform ellipticity and boundedness of the coefficient $a(t,x)$, boundedness of $B(t,x)$ and non-negativity of $\lambda$. 

%In the case of a smooth coefficient $a(t,x)$, and then with no transmission condition $(\star)$, the smoothness of this weak solution (turning out to be a classical solution) is well studied (\cite{lions-magenes},\cite{lieberman} again).
Concerning classical solutions in the presence of a discontinuous coefficient $a(t,x)$ like in our case, it seems that less references are available. In  the fundamental paper \cite{lady1} it is shown that, still with $\rho\equiv 1$, and in the case of cylindrical space-time subdomains  (that is to say 
$x_i(t)=x_i$ for all $1\leq i\leq I$, $0\leq t\leq T$) every weak solution to $(\mathcal{P}^\lambda_\mathrm{div,\Delta_{\bf x}}(\rho\equiv 1,a,B))$ is in fact classical. As a consequence there exists a classical solution to $(\mathcal{P}^\lambda_\mathrm{div,\Delta_{\bf x}}(\rho\equiv 1,a,B))$. 

In the case $\rho\neq 1$ and in the presence of non-cylindrical subdomains some results are announced in \cite{lady1} and \cite{lady}. However they are stated without any
complete proof (with the notable exception of the proof of the existence of a unique weak solution in the case of cylindrical subdomains, but with $\rho\neq 1$, pp 229-232 of \cite{lady}; see Subsection~\ref{ss:sol-faible} for further comments).

\vspace{0.2cm}

We continue this subsection by noticing that in fact we can get rid of the difficulty of having non-cylindrical subdomains, by applying a space transform trick, available only because the space dimension is one. We choose to present things on the problem in its non-divergence form 
$(\cP^\lambda_{\Delta_{\bf x}}(\sigma,b,\beta))$ again.

From now on we assume 
$I\geq 3$ and set 
$$
\forall (t,\hat{x})\in E,\quad\psi(t,\hat{x})=\left\{
\begin{array}{lll}
x_1(t)+(x_2(t)-x_1(t))(\hat{x}-1) &\text{if} &\hat{x}<1\\
\\
 x_j(t)+(x_{j+1}(t)-x_j(t))(\hat{x}-j) &\text{if} &j\leq \hat{x}<j+1,\;j=1,\ldots,I-2\\
 \\
 x_{I-1}(t)+(x_{I}(t)-x_{I-1}(t))(\hat{x}-I+1) &\text{if} &\hat{x}\geq I-1\\
\end{array}
\right.
$$
For any $t\in[0,T]$ we note $\Psi(t,\cdot)= [\psi(t,\cdot)]^{-1}(\cdot)$. Notice that 
$$\Delta=\Psi(\Delta_{\bf x})$$
and that $E\setminus\Delta$ appears as the union of some open cylindrical space-time domains.

 We have the following result.

\begin{proposition} 
\label{prop:redress}
A function $u(t,x)$ is a classical solution to $(\cP^\lambda_{\Delta_{\bf x}})(\sigma,b,\beta)$ if and only if $\hat{u}(t,\hat{x}):= u(t,\psi(t,\hat{x}))$ is a classical solution 
to
$$
((\hat{\cP}^\lambda_{\Delta})(\hat{\sigma},\hat{b},\hat{\beta}))
\left\{
\begin{array}{rcll}
\big[\hat{u}'_t+\frac 1 2\hat{\sigma}^2\hat{u}''_{\hat{x}\hat{x}}+\hat{b}\,\hat{u}'_{\hat{x}}-\lambda \hat{u}\big](t,\hat{x})&=&\hat{g}(t,\hat{x})&\forall (t,\hat{x})\in E\setminus\Delta \\
\\
(1+\hat{\beta}_i(t))\hat{u}'_{\hat{x}}(t,i+) &=& (1-\hat{\beta}_i(t))\hat{u}'_{\hat{x}}(t,i-)& \forall 1\leq i\leq I,\,\forall t\in[0,T)\;\;(\hat{\star}) \\
\\
\hat{u}(T,\hat{x})&=&\hat{f}(\hat{x})&\forall \hat{x}\in \R.\\
\\
\lim_{|\hat{x}|\to\infty}|\hat{u}(t,\hat{x})|&=&0&\forall t\in[0,T],
\end{array}
\right.
$$
where 
\begin{equation}
\label{eq:defisigchapeau}
\hat{\sigma}(t,\hat{x})=\sigma(t,\psi(t,\hat{x}))\times\Psi'_{x,\pm}(t,\psi(t,\hat{x})),\quad\hat{b}(t,\hat{x})=b(t,\psi(t,\hat{x}))\times\Psi'_{x,\pm}(t,\psi(t,\hat{x}))+\Psi'_{t,\pm}(t,\psi(t,\hat{x})),
\end{equation}

\noindent
$\hat{g}(t,\hat{x})=g(t,\psi(t,\hat{x}))$, $\hat{f}(\hat{x})=f(\psi(T,\hat{x}))$ and
\begin{equation}
\label{eq:defibetachapeau}
\hat{\beta}_i(t)=\frac{ (1+\beta_i(t))\Psi'_x(t,x_i(t)+) - (1-\beta_i(t))\Psi'_x(t,x_i(t)-) }  { (1+\beta_i(t))\Psi'_x(t,x_i(t)+) + (1-\beta_i(t))\Psi'_x(t,x_i(t)-)   }.
 \end{equation}

\end{proposition}

\begin{remark}
\label{rem:Psi}
Note that
\begin{equation}
\label{eq:defi-Psi}
\forall (t,x)\in E,\quad \Psi(t,x)=\left\{
\begin{array}{lll}
(x-x_1(t))/(x_2(t)-x_1(t))+1 &\text{if} &x<x_1(t)\\
\\
 (x-x_j(t))/(x_{j+1}(t)-x_j(t))+j &\text{if} &x_{j}(t)\leq x<x_{j+1}(t),\;j=1,\ldots,I-2\\
 \\
(x-x_{I-1}(t))/(x_{I}(t)-x_{I-1}(t))+I-1 &\text{if} &x\geq x_{I-1}(t)\\
\end{array}
\right.
\end{equation}
and that this function is of class $C(E)\cap C^{1,2}(E\setminus\Delta_{\bf x})$. Besides,  choosing $\varepsilon<\inf_{1\leq j\leq I-1}\inf_{s\in[0,T]}(x_{j+1}(s)-x_j(s))$ and
using the fact that $\varepsilon<x_{j+1}(t)-x_j(t)\leq\sup_{s\in[0,T]}(x_{j+1}(s)-x_j(s))$ we can see that there exist constants $0<\hat{m}<\hat{M}<\infty$ such that 
$\Psi'_{x,\pm}\in\Theta(\hat{m},\hat{M})$. In addition $\Psi'_{t,\pm}$ remains bounded (thanks in particular to the fact
that the $x_i:[0,T]\to\R$, $1\leq i\leq I$ are of class $C^1$). Thus the coefficients $\hat{\sigma}(t,\hat{x})$,
$\hat{b}(t,\hat{x})$ and $\hat{\beta_i}(t)$, $1\leq i\leq I$, still satisfy the hypotheses of Section \ref{sec:feynman}.

\end{remark}

\begin{proof}[Proof of Proposition \ref{prop:redress}]
We only prove the sufficient condition, the converse being proved in the same manner.

First for any $(t,x)\in E\setminus\Delta_{\bf x}$ we have 
\begin{equation}
\label{eq:uprime-Psi}
u'_x(t,x)=\hat{u}'_{\hat{x}}(t,\Psi(t,x))\times\Psi'_x(t,x), 
\end{equation}
and, as $\Psi''_{xx}(t,x)=0$,
\begin{equation}
\label{eq:usec-Psi}
u''_{xx}(t,x)=\hat{u}''_{\hat{x}\hat{x}}(t,\Psi(t,x))\times[\Psi'_x(t,x)]^2.
\end{equation}
We also have
\begin{equation}
\label{eq:ut}
u'_t(t,x)=\hat{u}'_t(t,\Psi(t,x))+\hat{u}'_{\hat{x}}(t,\Psi(t,x))\times\Psi'_t(t,x).
\end{equation}

So that for any $(t,\hat{x})\in E\setminus\Delta$ we may use this with $(t,x)=(t,\psi(t,\hat{x}))$ in the first line of 
$(\cP^\lambda_{\Delta_{\bf x}}(\sigma,b,\beta))$ and thus we get the first line of $(\hat{\cP}^\lambda_{\Delta}(\hat{\sigma},\hat{b},\hat{\beta}))$, with the newly defined coefficients $\hat{\sigma}$, $\hat{b}$ and 
$\hat{g}$.

Concerning the transmission condition $(\hat{\star})$, we notice that we have from $(\star)$ in $(\cP^\lambda_{\Delta_{\bf x}}(\sigma,b,\beta))$
$$
\forall t\in[0,T],\quad (1+\beta_i(t))\Psi'_x(t,x_i(t)+)\hat{u}'_{\hat{x}}(t,\Psi(t,x_i(t))+)
=(1-\beta_i(t))\Psi'_x(t,x_i(t)-)\hat{u}'_{\hat{x}}(t,\Psi(t,x_i(t))-)$$
for any $1\leq i\leq I$. As $\Psi(t,x_i(t))=i$ for any $1\leq i\leq I$, an easy computation shows that this is equivalent to $(\hat{\star})$,
with the newly defined $\hat{\beta}_i(t)$, $1\leq i\leq I$. 

The third and fourth lines of $(\hat{\cP}^\lambda_{\Delta}(\hat{\sigma},\hat{b},\hat{\beta}))$ are straightforward.
\end{proof}

\vspace{0.2cm}

We can sum up the preceding discussions in the following proposition.

\begin{proposition}
\label{prop:sum-up}
Assume the curves $x_i$,  and the coefficients $\beta_i$, $1\leq i\leq I$, are as in Theorem~\ref{thm-EDSTL}, and that~$b$ is in 
$\Xi(M)\cap C(E\setminus\Delta_{\bf x})$, and $\sigma$ is in $\Theta(m,M)\cap C(E\setminus\Delta_{\bf x})$.

 Let
$\hat{\sigma}$, $\hat{b}$, $\hat{\beta}_i$, $1\leq i\leq I$, defined by \eqref{eq:defisigchapeau} \eqref{eq:defibetachapeau}.  Let $\hat{\rho},\hat{a}$, $\hat{B}$ be defined by
\eqref{eq:defi-arhoB}, but with $\hat{\sigma}$, $\hat{b}$, $\hat{\beta}_i$, $1\leq i\leq I$ instead of
$\sigma$, $b$, $\beta_i$, $1\leq i\leq I$.

Then $(\cP^\lambda_{\Delta_{\bf x}}(\sigma,b,\beta))$ has a classical solution if and only if $(\mathcal{P}^\lambda_\mathrm{div,\Delta}(\hat{\rho},\hat{a},\hat{B}))$ 
has a classical solution $\hat{u}(t,\hat{x})$. This classical solution of $(\cP^\lambda_{\Delta_{\bf x}}(\sigma,b,\beta))$ is given by $u(t,x)=\hat{u}(t,\Psi(t,x))$ with $\Psi(t,x)$
defined by \eqref{eq:defi-Psi}.
\end{proposition}

Without loss of generality we shall investigate the problem $(\cP^\lambda_{\mathrm{div},\Delta}(\rho,a,B))$ (i.e. with $x_i\equiv i$, $1\leq i\leq I$): in Subsection
\ref{ss:sol-faible} we deal with weak solutions, and in Subsection \ref{ssec:solclassique} with classical solutions but in the case $\rho\equiv 1$ (we sum up the results of
\cite{lady1}).  In Subsection~\ref{ssec:solclassique-transfo} we present a way to get classical solutions in the case $\rho\neq 1$, using the results of Subsection \ref{ssec:solclassique}, and again (different) space transform techniques.

\subsection{Weak solutions}
\label{ss:sol-faible}

In this subsection it is assumed $\rho,a\in\Theta(m',M')$ and $B\in\Xi(M')$ for some $0<m'<M'<\infty$, and that the coefficient  $\rho$ satisfies the $\mathbf{H}^{(t)}$-hypothesis.

\vspace{0.1cm}

We will call a weak solution of the parabolic problem $(\cP^\lambda_{\mathrm{div},\Delta}(\rho,a,B))$ a function $u(t,x)$ in the space
$L^2(0,T;H^1(\R))\cap C([0,T];L^2(\R))$, with $u(T,\cdot)=f$ a.e., and satisfying for any test function $\varphi\in H^{1,1}_{0}(E)$ the relation
\begin{equation}
\label{eq:sol-faible}
\begin{array}{l}
\ds\int_0^T\int_\R u \frac{\mathrm{d}\varphi}{\mathrm{dt}}\rho^{-1}\,dxdt \\
\\
\ds+ \frac 1 2 \int_0^T\int_\R a \frac{ \mathrm{d}u}{\mathrm{dx}}  \frac{ \mathrm{d}\varphi}{\mathrm{dx}}  \,dxdt
-\int_0^T\int_\R B \frac{ \mathrm{d}u}{\mathrm{dx}}\varphi\rho^{-1}dxdt
+\int_0^T\int_\R u(\lambda-\frac{\rho'_t}{\rho})\varphi\rho^{-1}dxdt=-\int_0^T\int_\R g\varphi\rho^{-1}\,dxdt.\\
\end{array}
\end{equation}
Indeed, imagine for a while that we have a classical solution $u(t,x)$ of $(\cP^\lambda_{\mathrm{div},\Delta_{}}(\rho,a,B))$. If we formally multiply the first line of $(\cP^\lambda_{\mathrm{div},\Delta}(\rho,a,B))$ by a test function $\varphi$ vanishing at infinity and with 
$\varphi(0,\cdot)=\varphi(T,\cdot)=0$, and integrate the resulting equation against $\rho^{-1}dxdt$ on $[0,T]\times\R$ we recover 
\eqref{eq:sol-faible}, using in particular~$(\star)$ in the integration by parts formula.
 %(see in the Appendix the proof of Proposition \ref{prop:sol-classique-faible} for details).

\vspace{0.2cm}

We first aim at proving the following result.

\begin{proposition}
\label{prop:sol-faible}
The parabolic problem $(\cP^\lambda_{\mathrm{div},\Delta}(\rho,a,B))$ has a unique weak solution.
\end{proposition}

In fact this result is in essence contained in the discussion p 229-232 of \cite{lady}, but we want here to give our own, new and different proof, using the tools proposed in 
\cite{lions-magenes}. They  differ from the ones used in \cite{lady}\cite{lieberman} but provide an elegant framework to handle the problem, and could be the starting point for the use of Generalized Dirichlet forms in these questions (on this point see Remark
\ref{rem:dir} below).
We believe that studying directly the weak solutions of $(\cP^\lambda_{\mathrm{div},\Delta}(\rho,a,B))$ with these tools has an interest per se, and paves the way for future research
in the presence of coefficients having even less smoothness. 
 \vspace{0.4cm}

In order to use the tools in \cite{lions-magenes} we denote $\cH=L^2(0,T;L^2(\R);\rho^{-1})$ the set of measurable functions $f(t,x)$ such that 
$$\int_0^T\int_\R|f(t,x)|^2\rho^{-1}(t,x)dxdt<\infty,$$
equipped with the scalar product
$$
\forall u,v\in \cH,\quad \langle u,v\rangle_{\cH}=\int_0^T\int_\R u(t,x)v(t,x)\rho^{-1}(t,x)dxdt.$$

We denote $\cV=L^2(0,T;H^1(\R);\rho^{-1})$ the set of mesurable functions $f(t,x)$ such that for any $t\in[0,T]$ the function $f(t,\cdot)$ is in $H^1(\R)$ and
$$\int_0^T\int_\R|f(t,x)|^2\rho^{-1}(t,x)dxdt+\int_0^T\int_\R|\frac{\mathrm{d} f}{\mathrm{d x}}(t,x)|^2\rho^{-1}(t,x)dxdt<\infty,$$
equipped with the scalar product
$$
\forall u,v\in \cV,\quad \langle u,v\rangle_{\cV}=\langle u,v\rangle_\cH+\langle \frac{\mathrm{d} u}{\mathrm{d x}},\frac{\mathrm{d} v}{\mathrm{d x}} \rangle_\cH.$$

We will denote by $||\cdot||_\cH$ and $||\cdot||_\cV$ the norms corresponding to the above defined scalar products. We denote by $\cV'$ the dual of $\cV$. Note that we have
$
\cV\subset\cH\subset\cV'$
with dense inclusions.

\begin{remark}
\label{rem:rho}
Note that as $\rho\in\Theta(m',M')$, of course $\cH$ (resp. $\cV$) is, as a set, just equal to
$L^2(0,T;L^2(\R))$ (resp. $L^2(0,T;H^1(\R))$). Besides, as a set, $\cV'$ is equal to
$L^2(0,T;H^{-1}(\R))$. 
\end{remark}

\vspace{0.3cm}
\noindent
We define a semigroup $(U_t)_{t\in[0,T]}$ of contraction on $\cV'$ by
$$
U_tf(s,\cdot)=\left\{
\begin{array}{ll}
f(s+t,\cdot)&\text{if }  0<s<T-t\\
0&\text{otherwise}.
\end{array}
\right.
$$
We denote $(\Lambda,D(\Lambda;\cV'))$ the infinitesimal generator of $(U_t)$. We have the following elementary fact.

\begin{lemma}
\label{lem:geneU}
We have 
$$D(\Lambda,\cV')=\Big\{ u\,|\;u\in\cV',\; \frac{\mathrm{d} u}{\mathrm{d t}}\in\cV', \;u(T,\cdot)=0 \Big\}$$
and $\Lambda u=\dfrac{\mathrm{d} u}{\mathrm{d t}}$ for any $u\in D(\Lambda,\cV')$.

\end{lemma}

\begin{remark}
\label{rem:der-t-distru}
 In Lemma \ref{lem:geneU},  the time derivative 
$\dfrac{\mathrm{d} u}{\mathrm{d t}}$ is understood in the distribution sense. For example, in the case 
$u\in\cV\cap D(\Lambda,\cV')$, we have  $\langle u,v\rangle_{\cV',\cV}=\langle u,v\rangle_\cH$ for any $v\in\cV$,
and for any $\varphi\in C^{\infty,\infty}_{c,c}(E)$
\begin{equation*}
\langle \dfrac{\mathrm{d} u}{\mathrm{d t}},\varphi\rangle_{\cV',\cV}
=-\int_0^T\int_\R u(\varphi\rho^{-1})'_t dxdt
=-\langle u,\varphi'_t-\varphi\frac{\rho'_t}{\rho}\rangle_\cH.
\end{equation*}

Besides, for $u\in \cV\cap D(\Lambda,\cV')$ and $\varphi\in H^{1,1}_0(E)$ we have 
\begin{equation}
\label{eq:ipp1}
\langle \dfrac{\mathrm{d} u}{\mathrm{d t}},\varphi\rangle_{\cV',\cV}=
-\langle u,\dfrac{\mathrm{d} \varphi}{\mathrm{d t}}-\varphi\frac{\rho'_t}{\rho}\rangle_\cH
\end{equation}
(using the fact that $C^{\infty,\infty}_{c,c}(E)$ is dense in $H^{1,1}_0(E)$). Note that $\rho'_t$ exists in the classical sense, even if it is not continuous, thanks to the fact that the subdomains are cylindrical. Besides, $\rho'_t$ is bounded thanks to the $\mathbf{H}^{(t)}$-hypothesis.

\end{remark}

\begin{proof}
See \cite{lions-magenes}, Section 3.4.3.
\end{proof}

As $\rho\neq 1$ we cannot use directly Theorem 3.4.1 in \cite{lions-magenes}. We will use a natural generalization of this result, that we now state  (besides note that we deal here with backward problems with terminal condition). The proof is provided in the Appendix for the sake of completeness. 

\begin{theorem}
\label{thm:gene-lions}
Assume $\cA$ is a bilinear form on $\cV$ satisfying
\vspace{0.2cm}

i) $|\cA(u,v)|\leq C||u||_\cV||v||_\cV$ for all $u,v\in\cV$, where $0<C<\infty$.
\vspace{0.2cm}

ii) $\cA(v,v)+\lambda_0||v||_\cH^2\geq \alpha_0||v||_\cV^2$ for all $v\in\cV$ (for some $\lambda_0,\alpha_0>0$).

\vspace{0.2cm}

Then for any $G\in\cV'$ and any $f\in\cH$ there exists a unique $u\in L^2(0,T;H^1(\R))\cap C([0;T];L^2(\R))$ 
(in particular $u$ is in $\cV$) such that $u(T,\cdot)=f$, and with
$\dfrac{\mathrm{d} u}{\mathrm{d t}}\in L^2(0,T;H^{-1}(\R))$ 
 and
\begin{equation}
\label{eq:sol-faible2}
\big\langle -\dfrac{\mathrm{d} u}{\mathrm{d t}},v\big\rangle_{\cV',\cV}+\cA(u,v)=
\big\langle G,v\big\rangle_{\cV',\cV}\quad \forall v\in\cV.
\end{equation}

\end{theorem}
\begin{proof}
See the Appendix.
\end{proof}
\vspace{0.3cm}

\noindent
In order to apply Theorem \ref{thm:gene-lions} we now define for any $u,v\in \cV$ 
\begin{equation}
\label{eq:defA}
\cA(u,v)=\frac 1 2\int_0^T\int_\R a(t,x)\frac{\mathrm{d} u}{\mathrm{d x}}(t,x)\frac{\mathrm{d} v}{\mathrm{d x}}(t,x)dxdt
-\int_0^T\int_\R B(t,x)\frac{\mathrm{d} u}{\mathrm{d x}}(t,x)v(t,x)\rho^{-1}(t,x)dxdt+\lambda\langle u,v\rangle_\cH
\end{equation}
and for any $\lambda_0>0$
\begin{equation}
\label{eq:defA0}
\cA_{\lambda_0}(u,v)=\cA(u,v)+\lambda_0\langle u,v\rangle_\cH.
\end{equation}

Not surprisingly, using the strict ellipticity and boundedness of $\rho,a$, and the boundednes of $B$ we get the following result (the proof is postponed to the Appendix).

\begin{lemma}
\label{lem:Acontcoer}

The bilinear form $\cA(\cdot,\cdot)$  defined by \eqref{eq:defA} is continuous, i.e.
\begin{equation}
\label{eq:Adefcont}
\forall u,v\in\cV,\quad |\cA(u,v)|\leq C||u||_\cV||v||_\cV,
\end{equation}
where $C=C(m',M',\lambda)$.

It is always possible to choose $\lambda_0>0$ large enough such that $\cA_{\lambda_0}(\cdot,\cdot)$ defined by 
\eqref{eq:defA}\eqref{eq:defA0} is coercive, i.e.
\begin{equation}
\label{eq:Acoer}
\forall v\in\cV,\quad \cA_{\lambda_0}(v,v)\geq \alpha_0||v||^2_\cV.
\end{equation}
where $\alpha_0=\alpha_0(m',M')$.
\end{lemma}
\begin{proof}
See the Appendix.
\end{proof}

\vspace{0.5cm}

We are now in position to prove Proposition \ref{prop:sol-faible}. Indeed, thanks to Lemma \ref{lem:Acontcoer} we may apply Theorem \ref{thm:gene-lions} with $\cA(\cdot,\cdot)$ defined by \eqref{eq:defA} and with  $G\in\cV'$
defined by $\langle G,v\rangle_{\cV',\cV}=-\langle g,v\rangle_\cH$ for any $v\in\cV$.  For any $\varphi\in H^{1,1}_{0}(E)\subset\cV$, using \eqref{eq:ipp1} in the computation of the term
$\big\langle -\dfrac{\mathrm{d} u}{\mathrm{d t}},\varphi\big\rangle_{\cV',\cV}$ appearing in \eqref{eq:sol-faible2} ($\varphi$ replaces $v$), we get \eqref{eq:sol-faible}.

\vspace{0.5cm}
It is possible to go a bit further in the analysis of the weak solution and to prove the following lemma, that asserts that the weak solution of $(\cP^\lambda_{\mathrm{div},\Delta}(\rho,a,B))$ is of class $H^1$ in the time variable.

\begin{lemma}
\label{lem:ut-L2}
The weak $u$ solution of $(\cP^\lambda_{\mathrm{div},\Delta}(\rho,a,B))$ satisfies
$\dfrac{\mathrm{d} u}{\mathrm{d t}}\in L^2(0,T;L^2(\R))$.
\end{lemma}

\begin{proof}
See the Appendix.
\end{proof}

The above result is one of the crucial steps in the study of the case $\rho\equiv 1$ in \cite{lady1}.
However, it seems challenging to  adapt all the other steps of \cite{lady1} and \cite{lady} to our case $\rho\neq 1$, see Remark \ref{rem:rho-diff}. 

\begin{remark}
\label{rem:dir}
If we have a look at the operator $(\Lambda,D(\Lambda,\cV'))$ and the form $\cA(\cdot,\cdot)$ we have used just above, we can notice that those objects are very similar to the ones used to define a generalized Dirichlet form (note that the formalism in \cite{stannat} concerning the abstract operators seems inspired by 
\cite{lions-magenes}). 

This could be the starting point of the use of generalized Dirichlet forms to handle the problem of a fairly broad class of time inhomogeneous SDEs with local time
(see the already mentionned papers \cite{russo-trutnau}, \cite{trutnau1} for some results in this direction). This issue could be addressed in a future work.
\end{remark}

\subsection{Classical solutions in the case $\rho\equiv 1$}
\label{ssec:solclassique}

Here we want to summarize the results of the seminal paper \cite{lady1} for the problem 
$(\mathcal{P}^\lambda_{\mathrm{div},\Delta_{\rm z}}(1,\mathbf{A},\mathbf{B}))$ 
that we will use in Subsection \ref{ssec:solclassique-transfo}. In fact, for our coming purpose, we consider a slightly more general problem, that we denote by 
$(\mathcal{P}^\lambda_{\mathrm{div},\Delta_{\rm z},(l,r)}(1,\mathbf{A},\mathbf{B}))$
(with $-\infty\leq l<r\leq \infty$).
 It is defined by the following system of equations:
 $$
\left\{
\begin{array}{rcll}
\big[v'_t+\dfrac 1 2\big(\mathbf{A}v'_z\big)'_z+\mathbf{B}\,v'_z-\lambda v\big](t,z)&=&g(t,z)&\forall (t,z)\in [0,T)\times(l,r)
\setminus\Delta_{\bf z} \\
\\
\mathbf{A}(t,z_i+)v'_z(t,z_i+) &=& \mathbf{A}(t,z_i-)v'_z(t,z_i-)& \forall 1\leq i\leq I,\,\forall t\in[0,T)\;\;(\star) \\
\\
v(T,z)&=&f(z)&\forall z\in (l,r).\\
\\
v(t,l)&=&f_l(t)&\forall t\in[0,T)\\
\\
v(t,r)&=&f_r(t)&\forall t\in[0,T).\\
\end{array}
\right.
$$

Here we have $l<z_1<\ldots<z_I<r$ and we have denoted 
$\Delta_{\rm z}=\{(t,z_i):0\leq t\leq T\}_{i=1}^I$. The functions
$f_l,f_r$ giving the Dirichlet conditions are in $L^2(0,T)$. Note that the problem
$(\mathcal{P}^\lambda_{\mathrm{div},\Delta_{\rm z}}(1,\mathbf{A},\mathbf{B}))$ corresponds simply to $l=-\infty$, $r=\infty$ and
$f_l=f_r=0$.

We should precise what we mean by a classical solution $v(t,z)$ of 
$(\mathcal{P}^\lambda_{\mathrm{div},\Delta_{\rm z},(l,r)}(1,\mathbf{A},\mathbf{B}))$. 
For any compact $K\subset (0,T)\times (l,r)$ this is a function of class $C(K)\cap C^{1,2}(K\setminus \Delta_{\rm z})$
 such that for all $1\leq i\leq I$ the limits $v'_t(t,z_i\pm)$, $v'_z(t,z_i\pm)$ and $v''_{zz}(t,z_i\pm)$ exist and are continuous as functions of $t\in[0,T)$ (we assume for simplicity that $K$ contains all the $z_i$'s). Then $v(t,z)$ satisfies in particular the first and second line of 
 $(\mathcal{P}^\lambda_{\mathrm{div},\Delta_{\rm z},(l,r)}(1,\mathbf{A},\mathbf{B}))$
 in the classical sense.

\begin{theorem}[O.A. Ladyzhenskaya et al., \cite{lady1}] 
\label{thm:lady}
For any $\mathbf{A}\in\Theta(m',M')$ satisfying the $\mathbf{H}^{(x_i)}$ and $\mathbf{H}^{(t)}$-hypotheses, any 
$\mathbf{B}\in\Xi(M')$ satisfying the $\mathbf{H}^{(t)}$-hypothesis, and provided 
that $g$ satisfies the $\mathbf{H}^{(t)}$-hypothesis,
the parabolic problem $(\mathcal{P}^\lambda_{\mathrm{div},\Delta_{\rm z},(l,r)}(1,\mathbf{A},\mathbf{B}))$ has a classical solution $v(t,z)$,
that is Hölder continuous (see Remark \ref{rem:holder}).  Besides the time derivative $v'_t$ is itself Hölder continuous.
\end{theorem}

\begin{remark}
\label{rem:holder}
Here the Hölder continuity means more precisely that for any compact $K\subset (0,T)\times(l,r)$ we have
\begin{equation}
\label{eq:holder}
\forall (t,x),(s,y)\in K,\quad |v(t,x)-v(s,y)|\leq C|(t,x)-(s,y)|^\nu
\end{equation}
with $C,\nu$ positive constants depending on $K,m',M'$.
\end{remark}

{\bf Sketch of the proof of Theorem \ref{thm:lady}.}
We will give elements for the case $l=-\infty$, $r=\infty$, $f_l=f_r=0$, the cases with bounded domains and non-homogeneous Dirichlet boundary conditions being treated in a similar manner.

In \cite{lady1} things are studied in the forward form
$w'_t-\dfrac 1 2\big(\tilde{\mathbf{A}}w'_z\big)'_z-\tilde{\mathbf{B}}\,w'_z-\lambda w=-\tilde{g},
$
but it suffices to set $\tilde{\mathbf{A}}(t,x)=\mathbf{A}(T-t,x)$, $\tilde{\mathbf{B}}(t,x)=\mathbf{B}(T-t,x)$ 
and $\tilde{g}(t,x)=g(T-t,x)$, and to define
 $v(t,x)=w(T-t,x)$, in order to recover results on $v(t,x)$ as a solution to $(\mathcal{P}^\lambda_\mathrm{div,\Delta_z}(1,\mathbf{A},\mathbf{B}))$. Therefore we will explain things directly in the backward form of interest.
 
 \vspace{0.2cm}
 {\bf STEP1.} There exists a weak solution $v(t,z)$ to $(\mathcal{P}^\lambda_\mathrm{div,\Delta_z}(1,\mathbf{A},\mathbf{B}))$. The proof of this fact can be found in the books \cite{lady},\cite{lieberman}. The method of \cite{lions-magenes}, that we have adapted in Subsection \ref{ss:sol-faible} to the case $\rho\neq 1$, provides an alternative method.  Note that $v(t,z)$ lives in 
 $L^2(0,T;\,H^1(\R))$, which provides the boundary condition at infinity, as $H^1(\R)=H^1_0(\R)$.
 
 \vspace{0.1cm}
 {\bf STEP2.} This weak solution $v(t,z)$ is Hölder continuous (the proof of this point is particularly involved;  in \cite{lieberman} it requires the use of a parabolic Harnack inequality, available only in the case $\rho\equiv 1$; see also \cite{lady}).

\vspace{0.1cm}
{\bf STEP3.}
 One of the crucial steps in \cite{lady1} is to show that 
\begin{equation}
\label{eq:vt-L2}
\frac{\mathrm{d} v}{\mathrm{d}t}\in L^2(0,T;L^2(\R))
\end{equation}
 (see also Theorem 6.6 in \cite{lieberman}; in fact these authors work in a bounded space domain $D$ and show that 
$\frac{\mathrm{d} v}{\mathrm{d}t}\in L^2(0,T;L^2(D'))$ for any $D'\subsetneq D$; but we claim that their computations can be easily adapted to the case of unbounded domains. Note that \eqref{eq:vt-L2} is provided by the more general result (possibly $\rho\neq 1$) stated in Lemma \ref{lem:ut-L2}).

\vspace{0.1cm}
{\bf STEP4.} In fact $\frac{\mathrm{d} v}{\mathrm{d}t}$ has even more smoothness: it is itself Hölder continuous. In order to see that,
the authors of \cite{lady1} differentiate with respect to time the initial equation, to see $v'_t=\frac{\mathrm{d} v}{\mathrm{d}t}$ as a weak solution of
\begin{equation}
\label{eq:vt-faible}
(v'_{t})_t+\frac 1 2 (\mathbf{A}(v'_{t})'_z)'_z+\mathbf{B}(v'_{t})'_z-\lambda v'_{t}=g'_{t}-\frac 1 2 (a'_{t}v'_z)'_z-B'_tv'_z.
\end{equation}
Note that, as $a'_{t}v'_z$ is discontinuous, the source term in \eqref{eq:vt-faible} is a distribution, which is not a problem for obtaining the Hölder continuity of the weak solution $v'_t$ (see p144-145 of \cite{lady1}; one can then use the same general result that has been used in Step 2).

\vspace{0.1cm}
{\bf STEP5.} For fixed $t$, one can then see $v(t,\cdot)$ as a solution of the elliptic problem
$$
\dfrac 1 2\big(\mathbf{A}v'_z\big)'_z+\mathbf{B}\,v'_z-\lambda v=g-v'_t$$
with a smooth source term $g-v'_t$. Using results on the smoothness of elliptic problems one can then see that for all $t\in[0,T)$ the transmission condition $(\star)$ is satisfied in the classical sense. See the forthcoming remark.

\begin{remark}
As the space dimension is one, one can easily see that the transmission condition is satisfied in the classical sense for a.e. $t\in[0,T)$ in the following manner.
Noticing that
$$
\int_0^T\int_\R v \frac{\mathrm{d}\varphi}{\mathrm{dt}}\,dzdt 
=-\int_0^T\int_\R v'_t\varphi \,dzdt$$
for any $\varphi \in H^{1,1}_0(E)$ (the right hand side is a convergent integral thanks to $v'_t=\frac{\mathrm{d} v}{\mathrm{d}t}\in L^2(0,T;L^2(\R))$)
and using \eqref{eq:sol-faible} with 
$\rho\equiv 1$, $a=\mathbf{A}$ and $b=\mathbf{B}$, we get 
\begin{equation*}
 \frac 1 2 \int_0^T\int_\R \mathbf{A} \frac{\mathrm{d} v}{\mathrm{dz}} \frac{\mathrm{d} \varphi}{\mathrm{dz}}  \,dzdt=
\int_0^T\int_\R \big(\mathbf{B}  \frac{\mathrm{d} v}{\mathrm{dz}} -\lambda v -g + v'_t\big)\varphi\, dzdt.
\end{equation*}
Then for a.e. $t\in[0,T)$ we have for any $\phi\in H^1_0(\R)$
\begin{equation}
\label{eq:forvart}
\int_\R \mathbf{A}(t,\cdot) \frac{\mathrm{d} v}{\mathrm{dz}}(t,\cdot)  \frac{\mathrm{d} \phi}{\mathrm{dz}}  \,dz
=2\int_\R \big(\mathbf{B} \frac{\mathrm{d} v}{\mathrm{dz}} -\lambda v -g + v'_t\big)(t,\cdot)\phi\,dz.
\end{equation}
As $\big[\big(\mathbf{B}\frac{\mathrm{d} \phi}{\mathrm{dz}}  -\lambda v -g + v'_t\big)\big](t,\cdot)$ is in $L^2(\R)$ 
we can infer that $\mathbf{A}(t,\cdot)\frac{\mathrm{d} v}{\mathrm{dz}}(t,\cdot)$ is in $H^1(\R)$.

Let us draw some intermediate conclusions. As $v(t,\cdot)$ is in $H^1(\R)$ we know that $v(t,\cdot)\in C(\R)$, and more precisely that 
\begin{equation}
\label{eq:repres-cont}
v(t,z)-v(t,y)=\int_y^z\frac{\mathrm{d}v}{\mathrm{dz}}(t,\cdot)d\xi,\quad \forall z,y\in\R
\end{equation}
(\cite{brezis}, Theorem VIII.2). So that $v'_z(t,\cdot)$ exists in the classical sense and is equal a.e. to $\frac{\mathrm{d}v}{\mathrm{dz}}(t,\cdot)$. Using the same argument we see that
$\mathbf{A}(t,\cdot)v'_z(t,\cdot)$ is in $C(\R)$. As $\mathbf{A}(t,\cdot)$ is smooth on the intervals $(-\infty,z_1)$, $[z_i,z_{i+1})$, $i=1,\ldots,I-1$, $[z_I,\infty)$ we see that $v'_z(t,\cdot)$ is continuous on each of these intervals. So that $v(t,\cdot)\in C(\R)\cap C^1(\R\setminus\{z_1,\ldots,z_I\})$. Note that as for any $i=1,\ldots,I$ the limits $\mathbf{A}(t,z_i\pm)$ exist, the limits $v'_z(t,z_i\pm)$ exist too (if not 
$\mathbf{A}(t,\cdot)v'_z(t,\cdot)$ would not be continuous). Besides, the continuity of $\mathbf{A}(t,\cdot)v'_z(t,\cdot)$ on the whole real line $\R$ implies the transmission condition $(\star)$.

To show that the transmission condition is satisfied for every time $t\in[0,T)$, one may then use the smoothness of $v(t,z)$ outside the interfaces (forthcoming Step 6), 
together with uniform convergence arguments.
\end{remark}

{\bf STEP6.}  Using the additional smoothness of the coefficients outside the interfaces, one is able to assert that $v(t,z)$ satisfies the first line of
$(\mathcal{P}^\lambda_\mathrm{div,\Delta_z}(1,\mathbf{A},\mathbf{B}))$ in the classical sense.

\begin{remark}
\label{rem:rho-diff}
In \cite{lady}\cite{lady1}  the authors claim that this is feasible to mimic all the steps of the above summarized proof in the case $\rho\neq 1$ (but without writing down the proofs, except for the existence of the weak solution as already mentionned). However, in our opinion, to prove directly that the weak solution~$u(t,x)$ is Hölder presents difficulties in the case $\rho\neq 1$. 
\end{remark}

\subsection{Classical solutions in the case $\rho\neq 1$ by  means of space transforms}
\label{ssec:solclassique-transfo}

 We now aim at proving the following result.

\begin{proposition}
\label{prop:sol-classique}
Let $\lambda\geq 0$, a source term $g\in C^{}_c(E)$ and a terminal condition $f\in C_0(\R)\cap L^2(\R)$.

 Let $\rho,a\in\Theta(m',M')$ and $B\in\Xi(M')$ for some $0<m'<M'<\infty$. 
We assume that~$\rho,a$ satisfy the $\mathbf{H}^{(i)}$~and  $\mathbf{H}^{(t)}$-hypotheses, and that $B$ and $g$ satisfy the $\mathbf{H}^{(t)}$-hypothesis.

The
 problem 
 $(\cP^\lambda_{\mathrm{div},\Delta_{}}(\rho,a,B))$
  has a classical solution $u(t,x)$.
 %, which satisfies in addition $u'_t\in C_0(E)$.

\end{proposition}

\vspace{0.2cm}

\noindent
{\bf Proof of Proposition \ref{prop:sol-classique}.}

\vspace{0.2cm}
{\bf STEP1.} The problem $(\cP^\lambda_{\mathrm{div},\Delta_{}}(\rho,a,B))$ has a weak solution $u(t,x)$ 
(see Subsection \ref{ss:sol-faible}). We shall aim at proving that $u(t,x)$ is in fact a classical solution.

\vspace{0.3cm}
In the sequel we (arbitrarily) set $\delta=1/4$. We denote $\sigma=\sqrt{\rho a}$.

\vspace{0.2cm}
{\bf STEP2.} We treat in details what happens around the interface $\{(t,1):0\leq t\leq T\}$. We set
$$
\phi_1(t,x)=\int_1^x\frac{dy}{\rho(t,y)},$$

\begin{equation}
\label{eq:defi-bA}
\mathbf{A}_1(t,z)=\frac{a}{\rho}(t,\Phi_1(t,z))
\end{equation}
and
\begin{equation}
\label{eq:defi-bB}
\mathbf{B}_1(t,z)=[(\phi_1)'_{t,\pm}+B(\phi_1)'_{x,\pm}](t,\Phi_1(t,z)),
\end{equation}
where $\Phi_1(t,\cdot)=[\phi_1(t,\cdot)]^{-1}$.

We set $z_1=\inf_{t\in [0,T]}\phi_1(t,2-\delta)$. We will show that $u(t,x)$ satisfies $(\cP^\lambda_{\mathrm{div},\Delta_{}}(\rho,a,B))$ 
in the classical sense in the subregion $\{(t,x)\in E:\,x\leq \Phi_1(t,z_1)\}$. 

Note that for any $t\in[0,T]$ we have 
$\phi_1(t,1)=0$ and $\Phi_1(t,0)=1$, and that for any $z\leq z_1$, any $t\in[0,T]$ we have $\Phi_1(t,z)\leq \Phi_1(t,z_1)\leq 2-\delta$.
So that the sole singularity of the coefficients $\mathbf{A}_1(t,z)$ and $\mathbf{B}_1(t,z)$ in the region 
$\{(t,z)\in E:\,0\leq t\leq T,\,z\leq z_1\}$ is for $z=0$.

We consider the function $v_1(t,z)=u(t,\Phi_1(t,z))$, $0\leq t\leq T$, $z\leq z_1$. We claim that this is a weak solution to the problem
$(\mathcal{P}^\lambda_{\mathrm{div},\Delta_0,(-\infty,z_1)}(1,\mathbf{A}_1,\mathbf{B}_1))$ defined by the system of equations
$$
\left\{
\begin{array}{rcll}
\big[(v_1)'_t+\dfrac 1 2\big(\mathbf{A}_1(v_1)'_z\big)'_z+\mathbf{B}_1\,(v_1)'_z-\lambda v_1\big](t,z)&=&g(t,\Phi_1(t,z))&\forall (t,z)\in [0,T)\times(-\infty,z_1)
\setminus\Delta_{\bf 0} \\
\\
\mathbf{A}_1(t,0+)(v_1)'_z(t,0+) &=& \mathbf{A}_1(t,0-)(v_1)'_z(t,0-)& \forall t\in[0,T)\;\;(\star) \\
\\
v_1(T,z)&=&f(\Phi_1(T,z))&\forall z\in (-\infty,z_1)\\
\\
\lim_{z\to-\infty}v_1(t,z)&=&0&\forall t\in[0,T)\\
\\
v_1(t,z_1)&=&u(t,\Phi_1(t,z_1))&\forall t\in[0,T)\\
\end{array}
\right.
$$
(here, note that as $u(t,x)$ lives in particular in $L^2(0,T;L^2(\R))$, the function $t\mapsto u(t,\Phi_1(t,z_1))$ is in $L^2(0,T)$, as required for the Dirichlet boundary condition).

Indeed the restriction of $u(t,x)$ to the region $\{(t,x)\in E:\,x\leq \Phi_1(t,z_1)\}$ is in particular such that for any $t\in[0,T]$, 
$u(t,\cdot)\in H^1((-\infty,\Phi_1(t,z_1))$, with
$
\int_0^T\int_{-\infty}^{\Phi_1(t,z_1)}u \,dxdt+\int_0^T\int_{-\infty}^{\Phi_1(t,z_1)}\frac{\mathrm{d}u}{\mathrm{dx}} \,dxdt<\infty$,
and satisfies
\begin{equation*}
\begin{array}{l}
\ds\int_0^T\int_{-\infty}^{\Phi_1(t,z_1)} u \frac{\mathrm{d}\varphi}{\mathrm{dt}}\rho^{-1}\,dxdt 
+ \frac 1 2 \int_0^T\int_{-\infty}^{\Phi_1(t,z_1)} a \frac{ \mathrm{d}u}{\mathrm{dx}}  \frac{ \mathrm{d}\varphi}{\mathrm{dx}}  \,dxdt\\
\\
\ds
-\int_0^T\int_{-\infty}^{\Phi_1(t,z_1)} B \frac{ \mathrm{d}u}{\mathrm{dx}}\varphi\rho^{-1}dxdt
+\int_0^T\int_{-\infty}^{\Phi_1(t,z_1)} u(\lambda-\frac{\rho'_t}{\rho})\varphi\rho^{-1}dxdt=-\int_0^T\int_{-\infty}^{\Phi_1(t,z_1)} g\varphi\rho^{-1}\,dxdt,\\
\end{array}
\end{equation*}
for any test function $\varphi$ living in  $H^{1,1}_0(E)$ and satisfying in addition $\varphi(t,x)=0$, for any $t\in[0,T]$, $x\geq \Phi_1(t,z_1)$.
In fact, using Lemma \ref{lem:ut-L2}, we can rewrite the above equation as
\begin{equation}
\label{eq:form-faible-subdom}
\begin{array}{l}
\ds-\int_0^T\int_{-\infty}^{\Phi_1(t,z_1)}  \frac{\mathrm{d}u}{\mathrm{dt}}\varphi\rho^{-1}\,dxdt 
+ \frac 1 2 \int_0^T\int_{-\infty}^{\Phi_1(t,z_1)} a \frac{ \mathrm{d}u}{\mathrm{dx}}  \frac{ \mathrm{d}\varphi}{\mathrm{dx}}  \,dxdt\\
\\
\ds
-\int_0^T\int_{-\infty}^{\Phi_1(t,z_1)} B \frac{ \mathrm{d}u}{\mathrm{dx}}\varphi\rho^{-1}dxdt
+\lambda\int_0^T\int_{-\infty}^{\Phi_1(t,z_1)} u\varphi\rho^{-1}dxdt=-\int_0^T\int_{-\infty}^{\Phi_1(t,z_1)} g\varphi\rho^{-1}\,dxdt,\\
\end{array}
\end{equation}

Note that $u(t,x)=v_1(t,\phi_1(t,x))$ for $(t,x)$ with $x\leq \Phi_1(t,z_1)$, and 
$$
\frac{ \mathrm{d}u}{\mathrm{dx}}(t,x)=\frac{ \mathrm{d}v_1}{\mathrm{dz}}(t,\phi_1(t,x))\frac{1}{\rho(t,x)},\quad t\in[0,T],\;
x\leq \Phi_1(t,z_1)$$
(see Corollary VIII.10 in \cite{brezis}). For any test function $\varphi$ as above we set $\bar{\varphi}(t,z)=\varphi(t,\Phi_1(t,z))$,
$t\in[0,T]$, $z\leq z_1$. 
Note that $\frac{ \mathrm{d}\bar{\varphi}}{\mathrm{dz}}(t,z)=\frac{ \mathrm{d}\varphi}{\mathrm{dx}}(t,\Phi_1(t,z))\rho(t,\Phi_1(t,z))$.
Then, performing the change of variable $x=\Phi_1(t,z)$ in \eqref{eq:form-faible-subdom}, we get,
using in particular $dx=\rho(t,\Phi_1(t,z))dz$, $t\in[0,T]$,
$$
\begin{array}{l}
\ds-\int_0^T\int_{-\infty}^{z_1}  \frac{\mathrm{d}u}{\mathrm{dt}}(t,\Phi_1(t,z))\bar{\varphi}(t,z)\,dzdt 
+ \frac 1 2 \int_0^T\int_{-\infty}^{z_1} a(t,\Phi_1(t,z)) \frac{ \mathrm{d}v_1}{\mathrm{dz}}(t,z)\rho^{-1}(t,\Phi_1(t,z))  \frac{ \mathrm{d}\bar{\varphi}}{\mathrm{dz}}(t,z)  \,dzdt\\
\\
\ds
-\int_0^T\int_{-\infty}^{z_1} B(t,\Phi_1(t,z)) \frac{ \mathrm{d}v_1}{\mathrm{dz}}(t,z)\rho^{-1}(t,\Phi_1(t,z))\bar{\varphi}(t,z)\,dzdt
+\lambda\int_0^T\int_{-\infty}^{z_1} v_1\bar{\varphi}dzdt\\
\\
\ds =-\int_0^T\int_{-\infty}^{z_1} g(t,\Phi_1(t,z))\bar{\varphi}(t,z)\,dzdt.\\
\end{array}
$$
Using now 
$$
\frac{\mathrm{d}u}{\mathrm{dt}}(t,x)=\frac{\mathrm{d}v_1}{\mathrm{dt}}(t,\phi_1(t,x))
+\frac{ \mathrm{d}v_1}{\mathrm{dz}}(t,\phi_1(t,x))(\phi_1)'_{t,\pm}(t,x)$$
(see Proposition IX.6 in \cite{brezis}) and \eqref{eq:defi-bA} \eqref{eq:defi-bB} we can claim that we have
$$
\begin{array}{l}
\ds-\int_0^T\int_{-\infty}^{z_1}  \frac{\mathrm{d}v_1}{\mathrm{dt}}\bar{\varphi}\,dzdt 
+ \frac 1 2 \int_0^T\int_{-\infty}^{z_1} \mathbf{A}_1 \frac{ \mathrm{d}v_1}{\mathrm{dz}} \frac{ \mathrm{d}\bar{\varphi}}{\mathrm{dz}}\,dzdt\\
\\
\ds
-\int_0^T\int_{-\infty}^{z_1} \mathbf{B}_1 \frac{ \mathrm{d}v_1}{\mathrm{dz}}\bar{\varphi}\,dzdt
+\lambda\int_0^T\int_{-\infty}^{z_1} v_1\bar{\varphi}dzdt
 =-\int_0^T\int_{-\infty}^{z_1} g(t,\Phi_1(t,z))\bar{\varphi}(t,z)\,dzdt,\\
\end{array}
$$
for any $\bar{\varphi}\in H^{1,1}_0((0,T)\times(-\infty,z_1))$. As 
$\ds -\int_0^T\int_{-\infty}^{z_1}  \frac{\mathrm{d}v_1}{\mathrm{dt}}\bar{\varphi}\,dzdt 
=\int_0^T\int_{-\infty}^{z_1} v_1 \frac{\mathrm{d} \bar{\varphi}}{\mathrm{dt}}\,dzdt $, this means that 
$v_1$ is indeed a weak solution of $(\mathcal{P}^\lambda_{\mathrm{div},\Delta_0,(-\infty,z_1)}(1,\mathbf{A}_1,\mathbf{B}_1))$
(one could easily check that $v_1\in L^2(0,T; H^1((-\infty,z_1)))\cap C([0,T];L^2((-\infty,z_1)))$).

But according to the proof of Theorem \ref{thm:lady}, the function $v_1(t,z)$ is in fact also a classical solution of
$(\mathcal{P}^\lambda_{\mathrm{div},\Delta_0,(-\infty,z_1)}(1,\mathbf{A}_1,\mathbf{B}_1))$. We draw the consequences on the PDE problem solved by $u(t,x)$ in the classical sense, using again
$u(t,x)=v_1(t,\phi_1(t,x))$ and the expression of the classical derivatives (for $t\in[0,T], \;x\leq \Phi_1(t,z_1), \;x\neq 1$)
\begin{equation}
\label{eq:der-vu}
u'_x(t,x)=(v_1)'_z(t,\phi_1(t,x))(\phi_1)'_x(t,x)
\end{equation}
\begin{equation}
\label{eq:dert-vu}
u'_t(t,x)=(v_1)'_t(t,\phi_1(t,x))+(v_1)'_z(t,\phi_1(t,x))(\phi_1)'_t(t,x)
\end{equation}
\begin{equation}
\label{eq:der2-vu}
u''_{xx}(t,x)=(v_1)''_{zz}(t,\phi_1(t,x))((\phi_1)'_x)^2(t,x)+(v_1)'_z(t,\phi_1(t,x))(\phi_1)''_{xx}(t,x).
\end{equation}
We first identify the transmission condition at the interface 
$\{(t,1):0\leq t\leq T\}$. We have, using in particular
$(\phi_1)'_x(t,x)=\dfrac{1}{\rho(t,x)}$ and \eqref{eq:der-vu},
\begin{equation}
\label{eq:CT-class}
\begin{array}{lll}
a(t,1+)u'_x(t,1+)&=&\mathbf{A}_1(t,0+)\rho(t,1+)u'_x(t,1+)=\mathbf{A}_1(t,0+)(v_1)'_z(t,0+)\\
\\
&=& \mathbf{A}_1(t,0-)(v_1)'_z(t,0-)=\mathbf{A}_1(t,0-)\rho(t,1-)u'_x(t,1-)  = a(t,1-)u'_x(t,1-).\\
\end{array}
\end{equation}
Second, for $t\in[0,T)$, $x\leq \Phi_1(t,z_1)$, $x\neq 1$, we have
\begin{equation}
\label{eq:eq-class}
\begin{array}{l}
\big[u'_t+\dfrac \rho 2\big(au'_x\big)'_x+B\,u'_x-\lambda u\big](t,x)\\
\\
=
\big[u'_t+\dfrac{\sigma^2}{2}u''_{xx}+\big(B+\dfrac{\rho a'_x}{2} \big)\,u'_x-\lambda u\big](t,x)\\
\\
=
(v_1)'_t(t,\phi_1(t,x))+(v_1)'_z(t,\phi_1(t,x))(\phi_1)'_t(t,x)+\big(B+\dfrac{\rho a'_x}{2} \big)(t,x)(v_1)'_z(t,\phi_1(t,x))(\phi_1)'_x(t,x)\\
\\
\hspace{1cm}+\frac{\sigma^2(t,x)}{2}\Big( (v_1)''_{zz}(t,\phi_1(t,x))((\phi_1)'_x)^2(t,x)+v'_z(t,\phi_1(t,x))(\phi_1)''_{xx}(t,x) \Big)-\lambda v_1(t,\phi_1(t,x))\\
\\
=(v_1)'_t(t,\phi_1(t,x))+\frac{\sigma^2(t,x)}{2\rho^2(t,x)}(v_1)''_{zz}(t,\phi_1(t,x))-\lambda v_1(t,\phi_1(t,x))\\
\\
\hspace{1cm}+(v_1)'_z(t,\phi_1(t,x))\Big[  (\phi_1)'_t(t,x) + \big(B+\dfrac{\rho a'_x}{2} \big)(t,x)(\phi_1)'_x(t,x)+ \frac{\sigma^2(t,x)}{2}(\phi_1)''_{xx}(t,x) \Big]\\
\\
=\big[(v_1)'_t+\frac 1 2(\mathbf{A}_1(v_1)'_{z})'_z-\lambda v_1\big](t,\phi_1(t,x))\\
\\
\hspace{1cm}+\Big[\,(v_1)'_z\Big( (\phi_1)'_t\circ\Phi_1+\big(B(\phi_1)'_x+\dfrac{ a'_x}{2} \big)\circ\Phi_1  +\dfrac{\sigma^2(\phi_1)''_{xx}}{2}\circ\Phi_1-\dfrac{(\mathbf{A}_1)'_z}{2} \Big)\Big](t,\phi_1(t,x))\\
\\
=\big[(v_1)'_t+\frac 1 2(\mathbf{A}_1(v_1)'_{z})'_z+\mathbf{B}_1\,(v_1)'_z-\lambda v_1\big](t,\phi_1(t,x))=g(t,\Phi_1(t,\phi_1(t,x))=g(t,x).\\
\end{array}
\end{equation}
Here we have used 
$$(\mathbf{A}_1)'_z(t,z)=\big(\frac{a'_x}{\rho}-\frac{a\rho'_x}{\rho^2} \big)(t,\Phi_1(t,z))\rho(t,\Phi_1(t,z))
=a'_x(t,\Phi_1(t,z))+\sigma^2(t,\Phi_1(t,z))(\phi_1)''_{xx}(t,\Phi_1(t,z)).$$
In view of \eqref{eq:CT-class} and \eqref{eq:eq-class} we have proved that 
that $u(t,x)$ satisfies $(\cP^\lambda_{\mathrm{div},\Delta_{}}(\rho,a,B))$ 
in the classical sense in the subregion $\{(t,x)\in E:\,x\leq \Phi_1(t,z_1)\}$ (we can easily that $u(t,x)$ has the required smoothness and satisfies the terminal condition).

\vspace{0.3cm}
{\bf STEP3.} We repeat Step 2 around each interface $\{(t,i):0\leq t\leq T\}$, $2\leq i\leq I$. More precisely we define for any $2\leq i\leq I$
\begin{equation}
\label{eq:defi-phii}
\phi_i(t,x)=\int_i^x\frac{dy}{\rho(t,y)},
\end{equation}
and 
$$
z_{i,\rm d}=\sup_{t\in[0,T]}\phi_i(t,i-1+\delta).$$
For $2\leq i\leq I-1$ we define
$$
z_i=\inf_{t\in[0,T]}\phi_i(t,i+1-\delta).$$
By computations similar to Step 2 we will then prove that $u(t,x)$ satisfies $(\cP^\lambda_{\mathrm{div},\Delta_{}}(\rho,a,B))$ 
in the classical sense in each of the subregions
$
\{(t,x)\in E:\,\Phi_i(t,z_{i,\rm d})\leq x\leq \Phi_i(t,z_i)\}$, $2\leq i\leq I-1$, and in the region
$\{(t,x)\in E:\,\Phi_I(t,z_{I,\rm d})\leq x\}$.

In particular, at this stage, $u(t,x)$ satisfies the transmission condition $(\star)$ in $(\cP^\lambda_{\mathrm{div},\Delta_{}}(\rho,a,B))$ 
in the classical sense, at each interface (for $1\leq i\leq I)$.

\vspace{0.3cm}
{\bf STEP4.} The trouble is that we cannot say for the moment that the first line of $(\cP^\lambda_{\mathrm{div},\Delta_{}}(\rho,a,B))$
holds true in the whole domain $E^\circ\setminus\Delta$. Indeed let us examine what happens in the subregion
$\{(t,x)\in E:\,1<x<2\}$. It could happen that we do not have $\Phi_2(t,z_{2,\rm d})\leq \Phi_1(t,z_1)$ for any $t\in[0,T)$
(we  recall that $1\leq\Phi_1(t,z_1)\leq 2-\delta$ and note that $2\geq\Phi_2(t,z_{2,\rm d})\geq 1+\delta$). Indeed it depends on the variations of the coefficient $\rho$. So that the results of Steps 2 and 3 do not allow to say that the first line of 
$(\cP^\lambda_{\mathrm{div},\Delta_{}}(\rho,a,B))$ is satisfied in the whole region $\{(t,x)\in E:\,1<x<2\}$. 

Thus, we are led to use Theorem \ref{thm:lady} again, but in a different manner. We consider the restriction of $u(t,x)$ on the region
$\{(t,x)\in E:\,1< x< 2\}$. We claim that this is a weak solution of the problem
$(\mathcal{P}^\lambda_{\mathrm{div},\Delta_{},(1,2)}(1,\sigma^2,B-\frac{a\rho'_x}{2}))$
defined by the system of equations
$$
\left\{
\begin{array}{rcll}
\big[w'_t+\dfrac 1 2\big(\sigma^2 w'_x\big)'_x+(B-\frac{a\rho'_x}{2})\,w'_x-\lambda w\big](t,x)&=&g(t,x)&\forall (t,x)\in [0,T)\times(1,2) \\
\\
w(T,x)&=&f(x)&\forall x\in (1,2)\\
\\
w(t,1)&=&u(t,1)&\forall t\in[0,T)\\
\\
w(t,2)&=&u(t,2)&\forall t\in[0,T).\\
\end{array}
\right.
$$
Note that there is no transmission condition in $(\mathcal{P}^\lambda_{\mathrm{div},\Delta_{},(1,2)}(1,\sigma^2,B-\frac{a\rho'_x}{2}))$, as there is no interface in the considered domain.

To see that the restriction of $u(t,x)$ solves 
$(\mathcal{P}^\lambda_{\mathrm{div},\Delta_{},(1,2)}(1,\sigma^2,B-\frac{a\rho'_x}{2}))$ it suffices to start from
the weak formulation
\begin{equation*}
\begin{array}{l}
\ds\int_0^T\int_{1}^{2} u \frac{\mathrm{d}\varphi}{\mathrm{dt}}\rho^{-1}\,dxdt 
+ \frac 1 2 \int_0^T\int_{1}^{2} a \frac{ \mathrm{d}u}{\mathrm{dx}}  \frac{ \mathrm{d}\varphi}{\mathrm{dx}}  \,dxdt\\
\\
\ds -\int_0^T\int_{1}^{2} B \frac{ \mathrm{d}u}{\mathrm{dx}}\varphi\rho^{-1}dxdt
+\int_0^T\int_{1}^{2} u(\lambda-\frac{\rho'_t}{\rho})\varphi\rho^{-1}dxdt=-\int_0^T\int_{1}^{2} g\varphi\rho^{-1}\,dxdt,\\
\end{array}
\end{equation*}
stated for any $\varphi\in H^{1,1}_0((0,T)\times(1,2))$. Then, setting $\bar{\varphi}=\varphi\rho^{-1}$, using 
$\dfrac{\mathrm{d}\varphi}{\mathrm{dt}}\rho^{-1}-\dfrac{\varphi\rho'_t}{\rho^2}=\dfrac{\mathrm{d}\bar{\varphi}}{\mathrm{dt}}$,
$\dfrac{ \mathrm{d}\varphi}{\mathrm{dx}}=\rho\dfrac{ \mathrm{d}\bar{\varphi}}{\mathrm{dx}}+\bar{\varphi}\rho'_x$ (note that 
$\rho$ is differentiable w.r.t. $x$ in the classical sense in the considered subregion), and easy computations, we get
\begin{equation*}
\begin{array}{l}
\ds\int_0^T\int_{1}^{2} u \frac{\mathrm{d}\bar{\varphi}}{\mathrm{dt}}\,dxdt 
+ \frac 1 2 \int_0^T\int_{1}^{2} \sigma^2 \frac{ \mathrm{d}u}{\mathrm{dx}}  \frac{ \mathrm{d}\bar{\varphi}}{\mathrm{dx}}  \,dxdt\\
\\
\ds -\int_0^T\int_{1}^{2} (B-\frac{a\rho'_x}{2}) \frac{ \mathrm{d}u}{\mathrm{dx}}\bar{\varphi}\, dxdt
+\lambda\int_0^T\int_{1}^{2} u\bar{\varphi}\,dxdt=-\int_0^T\int_{1}^{2} g\bar{\varphi} \,dxdt,\\
\end{array}
\end{equation*}
for any $\bar{\varphi}\in H^{1,1}_0((0,T)\times(1,2))$. Thus the restriction of $u(t,x)$ is also a classical solution to
$(\mathcal{P}^\lambda_{\mathrm{div},\Delta_{},(1,2)}(1,\sigma^2,B-\frac{a\rho'_x}{2}))$ and we have for any
$(t,x)\in [0,T)\times(1,2)$,
$$
\begin{array}{lll}
g(t,x)&=&\big[u'_t+\dfrac 1 2\big(\sigma^2 u'_x\big)'_x+(B-\frac{a\rho'_x}{2})\,u'_x-\lambda u\big](t,x)\\
\\
&=&\big[u'_t+\dfrac 1 2\sigma^2 u''_{xx}+(B+\frac{(\sigma^2)'_x}{2}-\frac{a\rho'_x}{2})\,u'_x-\lambda u\big](t,x)\\
\\
&=&\big[u'_t+\dfrac 1 2\sigma^2 u''_{xx}+(B+\frac{\rho a'_x}{2})\,u'_x-\lambda u\big](t,x)\\
\\
&=&\big[u'_t+\dfrac \rho 2\big(a u'_x\big)'_x+B\,u'_x-\lambda u\big](t,x).\\
\end{array}
$$
Proceeding in the same way for the other subregions, and taking into account Steps 2 and 3 we can say that the first line of 
$(\cP^\lambda_{\mathrm{div},\Delta_{}}(\rho,a,B))$
is verified by $u(t,x)$ in the classical sense on $E^\circ\setminus\Delta$. Note that we clearly have $u\in C(E)$, as  for any $(t_0,x_0)\in E$ it is clear that $u$ is continuous at $(t_0,x_0)$ (even if $(t_0,x_0)\in\Delta$, using the continuity of $v_i(t,z)$ and $\phi_i(t,x)$).

\vspace{0.2cm}
Therefore Proposition \ref{prop:sol-classique} is proved.

\vspace{1cm}

%\begin{remark}
%Note  that of course the uniqueness of the classical solution $u(t,x)$, obtained through the probabilistic representation in
%Theorem \ref{thm:feynman}, can be obtained as a consequence of the uniqueness of the weak solution. Indeed
%if $u(t,x)$ and $\bar{u}(t,x)$ are two classical solutions of $(\cP^\lambda_{\mathrm{div},T}(\rho,a,B))$ (or 
%$(\cP^\lambda_\mathrm{T})$), they are weak solutions (Proposition \ref{prop:sol-classique-faible}), that are equal by the uniqueness part of Proposition \ref{prop:sol-faible}. Therefore they are equal for a.e. $(t,x)\in E$. But as they are continuous they are equal for any $(t,x)\in E$.
%\end{remark}

We now give further properties of the solution $u(t,x)$ considered in Proposition \ref{prop:sol-classique}.

\begin{lemma}
\label{lem:ut-cont}
In the context of Proposition \ref{prop:sol-classique},  the classical time derivative $u'_t$ of the classical solution of
$(\cP^\lambda_{\mathrm{div},\Delta_{}}(\rho,a,B))$
 is continuous.
\end{lemma}

\begin{proof}
That $u'_t$ is continuous at any point $(t,x)\notin\Delta$ is clear, by definition of a classical solution. Let $(t,x)\in \Delta$, i.e.
we have $(t,x)=(t,i)$ for some $1\leq i\leq I$. Considering \eqref{eq:dert-vu} we have
$$
u'_t(t\pm,i\pm)=(v_i)'_t(t\pm,0\pm)+(v_i)'_z(t\pm,0\pm)(\phi_i)'_t(t\pm,0\pm).
$$
But by taking the time derivative of \eqref{eq:defi-phii}, and inverting this derivative and the integral sign, we see that we simply have
$(\phi_i)'_t(t\pm,0\pm)=0$. And thus
$$
u'_t(t\pm,i\pm)=(v_i)'_t(t\pm,0\pm).
$$
But as $(v_i)'_t$ is continuous (Theorem \ref{thm:lady}) we see that $u'_t(t\pm,i\pm)=(v_i)'_t(t,0)=u'_t(t,i)$.
\end{proof}

\begin{remark}
\label{rem:ut-cont}
Note that the result of Lemma \ref{lem:ut-cont} is true because the interfaces are not moving. In the case of moving interfaces $u'_t$ will not be continuous in general, because there is no reason the second RHS term in \eqref{eq:ut} vanishes at the interface (contrary to what happens in \eqref{eq:dert-vu}).
\end{remark}

{\bf Conclusion of Section \ref{sec:EDP}.} In view of Theorem~\ref{thm:feynman} and Propositions \ref{prop:sum-up} and \ref{prop:sol-classique}, 
we have proved the following theorem.

\begin{theorem}
\label{thm:sol-classique}
Assume the $x_i$'s and $\beta_i$'s are as in Theorem \ref{thm-EDSTL}. Assume that $\sigma\in\Theta(m,M)$ satisfies the 
$\mathbf{H}^{(x_i)}$ and $\mathbf{H}^{(t)}$-hypotheses and that 
 $b\in\Xi(M)$ satisfies the $\mathbf{H}^{(t)}$-hypothesis. Assume that $\lambda\geq 0$, that $g\in C_c(E)$ satisfies the 
 $\mathbf{H}^{(t)}$-hypothesis and that $f\in C_0(\R)\cap L^2(\R)$.

Then the problem $(\cP^\lambda_{\Delta_{\bf x}}(\sigma,b,\beta))$ defined in Section \ref{sec:feynman} has a unique classical solution.
\end{theorem}

\section{Markov property, Feller semigroup and generator in the strong sense}
\label{sec:markov}

We first have the following result.

\begin{proposition}
\label{prop:Xmarkov}
In the context of Theorem \ref{thm-EDSTL},
assume that $\sigma$ satisfies the $\mathbf{H}^{(x_i)}$ and $\mathbf{H}^{(t)}$-hypotheses and that $b$ satisfies the $\mathbf{H}^{(t)}$-hypothesis.

Let $(X,W),(\Omega,\cF,(\cF_t)_{t\in[0,T]},\P)$ be a weak solution of \eqref{eq:X}.

 Then $X$
is a Feller time inhomogeneous $(\cF_t)$-Markov process.
\end{proposition}

\begin{proof}
Remember that for any $t\in[0,T]$, $X_t=r(t,Y_t)$ where $Y$ is the solution of \eqref{eq:Y} with  the coefficients defined by \eqref{eq:defi-sigma}. As these coefficients satisfy the hypotheses of Theorem \ref{thm:pbmart}
  we can see from Theorem 6.2.2 in \cite{stroockvar} that $Y$ is Markov, as already pointed in Subsection \ref{ss:pbmart}. 

Therefore we can easily see that $X$ is Markov and that the associated family $(P_{s,t}^X)$ satisfies \eqref{eq:evo}. Thus the family $(P^X_t)$ (associated to the space time process $\tilde{X}$) satisfies \eqref{eq:semigroup}. The only point that requires special attention is to show that $(P^X_t)$ is a Feller semigroup. Indeed, as the coefficients $\bar{\sigma},\bar{b}$ in \eqref{eq:Y}
are not smooth, we cannot apply directly Corollary 3.1.2 in \cite{stroockvar}, to get the Feller property for the family $(P^Y_{s,t})$ associated to~$Y$, and deduce the Feller property
for $(P_{s,t}^X)$. 

Thus we will focus on $(P^Y_t)$, and prove by our means that this is a Feller semigroup.
We recall that
\begin{equation}
\label{eq:defPYt}
\forall (s,y)\in E, \;\;\forall \varphi\in C_0(E), \;\; \forall\,0\leq t\leq T-s, \quad P^Y_t\varphi(s,y)=P^Y_{s,t+s}\varphi(t+s,y)=\E^{s,y}[\varphi(s+t,Y_{s+t})].
\end{equation}
Then, one may show that $(P_t^X)$
inherits the Feller property of $(P_t^Y)$. To that aim, one may denote now $\tilde{r}(t,y)=(t,r(t,y))$, $\tilde{R}(t,x)=(t,R(t,x))$, use the relationship
$$
\forall (s,x)\in E,\;\;\forall \varphi\in C_0(E),\;\forall t\in[0,T-s],\quad P^X_t\varphi(s,x)=P^Y_t(\varphi\circ\tilde{r})(\tilde{R}(s,x)),
$$
the continuity of $r(t,z)$, $R(t,x)$, and $\lim_{y\to\pm\infty}r(t,y)=\pm\infty$, $\lim_{x\to\pm\infty}R(t,x)=\pm\infty$,
for any~$t\in[0,T]$. 

\vspace{0.3cm}
That being said, we now prove that $(P^Y_t)$ is Feller. We denote $\Delta_{\bf y}=\tilde{R}(\Delta_{\bf x})$. Note that, thanks to the assumptions on the coefficients, and Proposition \ref{prop:sol-classique}, we have that $(\mathcal{P}^\lambda_{\Delta_{\bf y}}(\bar{\sigma},\bar{b},0))$ has a classical solution for any finite time horizon, terminal condition $f\in C_0(\R)\cap L^2(\R)$, and $g\equiv 0$. Note that $(\mathcal{P}^\lambda_{\Delta_{\bf y}}(\bar{\sigma},\bar{b},0))$
is a parabolic transmission problem with discontinuous coefficients, but with no transmission condition (more precisely the transmission condition is simply of type
$u'_y(t,y_i(t)+)=u'_y(t,y_i(t)-)$ for any $t\in[0,T)$).

\vspace{0.3cm}
{\bf STEP1.}  Pick $\varphi\in C^{\infty,\infty}_c(E)$. We will show that $P^Y_t\varphi$ is in $C_0=C_0(E)$.

\vspace{0.1cm}
{\bf a)} Let $(s,y)\in E$ be fixed. We first show that $P^Y_t\varphi$ is continuous at point $(s,y)$. Let $\delta >0$. For any 
$(r,z)\in E$ (we suppose that $t+s,t+r<T$) we have
\begin{equation}
\label{eq:PYtcont}
\begin{array}{lll}
|P^Y_t\varphi(s,y)-P^Y_t\varphi(r,z)|&\leq&\big|\E^{s,y}[\varphi(t+s,Y_{t+s})]-\E^{r,z}[\varphi(t+s,Y_{t+s})]\,\big|
\\
\\
&&\hspace{-2.5cm}+\big|\E^{r,z}[\varphi(t+s,Y_{t+s})]-\E^{r,z}[\varphi(t+s,Y_{t+r})]\,\big|
+|P^Y_{r,t+r}\varphi(t+s,z)-P^Y_{r,t+r}\varphi(t+r,z)|.\\
\end{array}
\end{equation}
Note that  by virtue of Theorem \ref{thm:feynman}, for any $(r,z)$ we may regard $\E^{r,z}[\varphi(t+s,Y_{t+s})]$ as $u_{t+s}(r,z)$, 
where $u_{t+s}$ is the classical solution of the parabolic problem $(\cP^0_{\Delta_{\bf y}}(\bar{\sigma},\bar{b},0))$ (with time horizon $t+s\leq T$), with terminal condition $\varphi(t+s,\cdot)\in C^\infty_c(\R)\subset C_0(\R)\cap L^2(\R)$ 
and source term $g\equiv 0$.

As the function $u_{t+s}$ is continuous on $E$ we may find $\eta_1$ such that for any
$(r,z)$ with $|(s,y)-(r,z)|<\eta_1$ we have
$$
\big|\E^{s,y}[\varphi(t+s,Y_{t+s})]-\E^{r,z}[\varphi(t+s,Y_{t+s})]\,\big|<\frac \delta 3.$$
We now turn to the second RHS term in \eqref{eq:PYtcont}. We have, 
$$
\big|\E^{r,z}[\varphi(t+s,Y_{t+s})]-\E^{r,z}[\varphi(t+s,Y_{t+r})]\,\big|\leq ||\varphi'_x||_\infty  \E^{r,z}|Y_{t+s}-Y_{t+r}|.$$
Further, we have
$$
\E^{r,z}|Y_{t+s}-Y_{t+r}|^2\leq 4\big(  \E^{r,z}\big| \int_{t+s}^{t+r}\bar{\sigma}(u,Y_u)dW_u \big|^2 
+  \E^{r,z}\big| \int_{t+s}^{t+r}\bar{b}(u,Y_u)du \big|^2   \big)
\leq 4\bar{M}^2(|r-s|+|r-s|^2),$$
where we have used $|a+b|^2\leq 4(|a|^2+|b|^2)$ and the fact that $\bar{\sigma},\bar{b}\in\theta(\bar{m},\bar{M})$. Thus by Jensen inequality we see that
$$
\E^{r,z}|Y_{t+s}-Y_{t+r}|\leq C(T)|r-s|^{1/2}.$$
 To sum up we may find $\eta_2>0$ such that for any $|(s,y)-(r,z)|<\eta_1\wedge\eta_2$ we have
$$
  \big|\E^{r,z}[\varphi(t+s,Y_{t+s})]-\E^{r,z}[\varphi(t+s,Y_{t+r})]\,\big|< \frac \delta 3.$$
  To finish with, we turn to the third RHS term in \eqref{eq:PYtcont}. It is clear that we have
  $$
  |P^Y_{r,t+r}\varphi(t+s,z)-P^Y_{r,t+r}\varphi(t+r,z)|\leq ||\varphi(t+s,\cdot)-\varphi(t+r,\cdot)||_\infty
  \leq ||\varphi'_t||_\infty \,|r-s|,$$
so that we may find $\eta_3>0$ such that for any $|(r,z)-(s,x)|<\eta_3$ we have
$$
|P^Y_{r,t+r}\varphi(t+s,z)-P^Y_{r,t+r}\varphi(t+r,z)|<\frac \delta 3.$$
Thus, setting $\eta=\eta_1\wedge\eta_2\wedge\eta_3$, we have
$$
|P^Y_t\varphi(s,x)-P^Y_t\varphi(r,z)|<\delta$$
for any $|(r,z)-(s,x)|<\eta$. Therefore the continuity of $P^Y_t\varphi$ is established.

\vspace{0.2cm}
{\bf b)}  We now show that $\lim_{|y|\to\infty}P^Y_t\varphi(s,y)\to 0$ (for any $s\in[0,T]$). Again we may see $P^Y_t\varphi(s,\cdot)$ as the solution $u_{t+s}(s,\cdot)$ (at time $s\in[0,t+s]$) of $(\cP^0_{\Delta_{\bf y}}(\bar{\sigma},\bar{b},0))$ with terminal condition $\varphi(t+s,\cdot)$ (again time horizon is $t+s$ and the source term is zero). The result then follows from the boundary condition in problem $(\cP^0_{\Delta_{\bf y}}(\bar{\sigma},\bar{b},0))$.

\vspace{0.3cm}
{\bf STEP2.} Pick $\varphi\in C_0$. We may construct a sequence $(\varphi_n)$ in $C^{\infty,\infty}_c(E)$ such that 
$||\varphi_n-\varphi||_\infty\to 0$ as $n\to\infty$. As $||P^Y_tf||_\infty\leq ||f||_\infty$ for any $f\in C_b(E)$, we get
$||P^Y_{t}\varphi-P^Y_{t}\varphi_n||_\infty\leq ||\varphi-\varphi_n||_\infty$, and we see that the sequence $(P^Y_{t}\varphi_n)$ in 
$C_b(E)$ converges uniformly to $P^Y_{t}\varphi$. Therefore $P^Y_{t}\varphi$ is in $C_0$, as each $P^Y_{t}\varphi_n$ is in $C_0$ by Step 1. This shows that for any $t\in[0,T]$, $P^Y_tC_0\subset C_0$.

\vspace{0.2cm}
{\bf STEP3.} Let $(s,y)\in E$ and $\varphi\in C_0$. From \eqref{eq:defPYt} and the continuity of $Y$, we easily see by dominated convergence that 
$P^Y_t\varphi(s,y)\to\varphi (s,y)$ as $t\downarrow 0$. Using this and the conclusion of Step 2, we deduce from 
Proposition III.2.4 in \cite{RY} that $(P^Y_t)$ is a Feller semigroup.
\end{proof}

\vspace{0.5cm}

Therefore the corresponding space-time process $\tilde{X}=((t,X_t))_{t\in [0,T]}$ is an $E$-valued Feller homogeneous $(\cF_t)$-Markov process (cf Subsection \ref{ss:markov-inho}). We wish to identify the infinitesimal generator of $\tilde{X}$. 
For technical reasons we only treat the case $\Delta_{\bf x}=\Delta$ (see Remark \ref{rem:gene-cyl}).
To that aim we have to introduce further notations.

\vspace{0.2cm}
With the same assumptions on the coefficients $\beta_i$ as in Theorem \ref{thm-EDSTL},  we define
$$
\begin{array}{l}
\mathcal{S}^X=\Big\{\,\varphi\in C(E)\cap C^{1,2}(E\setminus\Delta):\;\;  \text{with}\;\; \varphi(T,\cdot)=0,\;\;\;
t\mapsto \varphi'_t(t,i) \text{ is continuous on }[0,T),\\
\\
\hspace{1cm}\forall (t,i)\in\Delta,\;\;\varphi'_t(t,i\pm) = \varphi'_t(t,i)
\text{ and } \varphi'_x(t,i\pm)\text{ and }\varphi''_{xx}(t,i\pm)\text{ exist with } \\
\\
\hspace{0.2cm}\frac{\sigma^2}{2}(t,i+)\varphi''_{xx}(t,i+)+b(t,x+)\varphi'_x(t,i+)=\frac{\sigma^2}{2}(t,i-)\varphi''_{xx}(t,i-)+b(t,i-)\varphi'_x(t,i-).\\
\\
\hspace{2.5cm}\text{Besides}\;\;
(1+\beta_i(t))\varphi'_x(t,i+)=(1-\beta_i(t))\varphi'_x(t,i-)\;\;\forall 1\leq i\leq I,\,\forall t\in[0,T)\;\;(\star)
\\
\\
\hspace{1.5cm}\forall 1\leq i\leq I,\;\;\;\varphi'_x(t,i\pm) \text{ and }\varphi''_{xx}(t,i\pm)\text{ are continuous functions of } t\in[0,T)\\
\\
\hspace{1.5cm}\text{and }\lim_{|x|\to\infty}\big(\varphi'_t(t,x)+\frac 1 2\sigma^2(t,x)\varphi''_{xx}(t,x)+b(t,x)\varphi'_x(t,x)\big)=0
\quad\forall t\in[0,T]
  \Big\}.\\
  \end{array}
  $$
  For any $\varphi\in\cS^X$ we define $L^X\varphi$ by
  $$
\begin{array}{lll}
\forall (t,x)\in E\setminus\Delta,&
L^X\varphi(t,x)&=\varphi'_t(t,x)+\frac 1 2\sigma^2(t,x)\varphi''_{xx}(t,x)+b(t,x)\varphi'_x(t,x)\\
\\
\forall (t,i)\in\Delta,& L^X\varphi(t,i)&=\varphi'_t(t,i)+\frac{\sigma^2}{2}(t,i+)\varphi''_{xx}(t,i+)+b(t,i+)\varphi'_x(t,i+)\\
\\
&&=\varphi'_t(t,i)+\frac{\sigma^2}{2}(t,i-)\varphi''_{xx}(t,i-)+b(t,i-)\varphi'_x(t,i-).\\
\end{array}
$$

We will have the following result.

\begin{theorem}
\label{thm:gene}

Assume $\Delta_{\bf x}=\Delta$. In the context of Proposition \ref{prop:Xmarkov} let $X=(X_t)_{t\in [0,T]}$ be the solution of~\eqref{eq:X}.

 We then denote by $(\cL^X,D(\cL^X))$ the infinitesimal generator of the Feller space-time process $\tilde{X}$. 

Then the operator $(\cL^X,D(\cL^X))$ is the closure of $(L^X,\cS^{X})$.

\end{theorem}

\begin{remark}
\label{rem:domL}
Note that the condition $\varphi(T,\cdot)=0$ in the definition of $\cS^X$ is here because we already know that the functions $\varphi$ in $D(\cL^X)$ have to satisfy
$\varphi(T,\cdot)=0$. Indeed, as we have set $P^X_t\varphi(s,x)=0$ for $t+s>T$, this is needed in order to have the existence of the limit in \eqref{eq:defi-gene} for $s=T$. This is somehow the same issue
as in the definition 
of the domain $D(\Lambda,\cV')$ in Lemma \ref{lem:geneU}.
\end{remark}

 \begin{proof}[Proof of Theorem \ref{thm:gene}] 
Take $\varphi\in \cS^X\subset C_0$ and notice that $L^X\varphi$ is in $C_0$. Then, using Proposition \ref{prop:transfo},
equation \eqref{eq:varbetai} and condition $(\star)$,
 we have
for any $0\leq s\leq t\leq T$,
$$
\varphi(\tilde{X}_t)-\varphi(\tilde{X}_s)-\int_s^tL^X\varphi(\tilde{X}_u)du=\int_s^t\varphi'_{x,\pm}(u,X_u)\sigma(u,X_u)dW_u.
$$

The above $t$-indexed process being a martingale we see by Proposition \ref{prop-genemart} that 
$\cS^X\subset D(\cL^X)$ and that $\cL^X$ coincides with $L^X$ on $\cS^X$.

We shall now prove that the closure of $(L^X,\cS^{X})$ is the generator of a Feller semigroup on $C_0$.
Indeed the result will then follow from Exercise VII.1.18 in \cite{RY} (note that in the language of \cite{ethier} we have $(L^X,\cS^{X})\subset (\cL^X,D(\cL^X))$, and
that $(\cL^X,D(\cL^X))$ is closed, see Proposition VII.1.3 in \cite{RY}).

The idea is to apply Theorem 1.2.12 in \cite{ethier}, which is an Hille-Yosida type theorem, in the Banach space $C_0$ (see also their Theorem 4.2.2).

\vspace{0.2cm}

{\bf STEP1.}
Let $g\in C^{1,0}_c(E)\subset C_0$, and $\lambda>0$. The equation
\begin{equation}
\label{eq:res-forte}
\lambda u-L^Xu=-g
\end{equation}
with terminal condition $u(T,\cdot)=0$ and with $\lim_{|x|\to\infty}u(t,x)=0$, has a classical solution $u(t,x)$ satisfying $(\star)$, 
 living in $C_0(E)\cap C^{1,2}(E\setminus\Delta)$, and satisfying all the other requirements for being in $\cS^X$, thanks to the results of Subsection
 \ref{ssec:solclassique-transfo} (see in particular Lemma \ref{lem:ut-cont} and Theorem \ref{thm:sol-classique}). Note in particular that as 
 $(L^Xu)(t,x)=(g+\lambda u)(t,x)$, and as $g\in C_c(E)$ and $u\in C_0(E)$, we clearly have
 that $(L^Xu)(t,x)\to 0$ as $x\to\infty$ (for any $t\in [0,T]$).

Remember that $C_c^{1,0}(E)$ is dense in $C_0$. Thus, denoting by $\cR(\lambda I-L^X)$ the image of $\cS^{X}$ by the operator 
$\lambda I-L^X$, we have
$$
C_c^{1,0}(E)\subset \cR(\lambda I-L^X)\subset C_0,$$
and taking closures we see that $\cR(\lambda I-L^X)$ is dense in $C_0$.

{\bf STEP2.} The domain $\cS^{X}$ is obviously dense in $C_0$.

{\bf STEP3.} We show now that $(L^X,\cS^{X})$ is dissipative. Let $\lambda>0$ and pick $\varphi\in\cS^{X}$. 

\vspace{0.1cm}

{\bf a)} Assume $\varphi$ reaches a positive maximum at a point 
$(t_0,x_0)\in [0,T)\times\R$.

\vspace{0.1cm}

 If $(t_0,x_0)\notin\Delta$ it is clear that $\varphi'_t(t_0,x_0)\leq 0$, $\varphi'_x(t_0,x_0)=0$ and $\varphi''_{xx}(t_0,x_0)\leq 0$, thus
$L^X\varphi(t_0,x_0)\leq 0$. 

\vspace{0.1cm}

If $(t_0,x_0)\in\Delta$ (i.e. $x_0=i$ for some $1\leq i\leq I$) things are not so clear because of the lack of smoothness of $\varphi$ on $\Delta$. But because 
$(1+\beta_{i_0}(t_0)),(1-\beta_{i_0}(t_0))>0$, $\varphi'_x(t_0,x_0+)$ and $\varphi'_x(t_0,x_0-)$ share the same sign and this implies 
$\varphi'_x(t_0,x_0\pm)=0$.

Let us now prove that $\varphi'_t(t_0,x_0)\leq 0$. Indeed, since $t\mapsto \varphi(t,x_0)$ is a $C^1$ function, we may apply the mean value theorem ensuring that for $h>0$ there exists $\theta\in (0,1)$ such that $\frac{1}{h}(\varphi(t_0 + h, x_0) - \varphi(t_0, x_0)) = \varphi'_t(t_0,x_0) + (\varphi'_t(t_0 + \theta h,x_0) - \varphi'_t(t_0,x_0))$. Now, since  $\varphi$ reaches a positive maximum at a point 
$(t_0,x_0)\in [0,T)\times\R$, the left hand side of the equality is negative. Then, letting $h$ tend to zero in the right hand side ensures that necessarily $\varphi'_t(t_0,x_0)\leq 0$.

 Again, since $\varphi$ reaches a positive maximum at $(t_0,x_0)$ we have
$\varphi''_{xx}(t_0,x_0\pm)\leq 0$, and consequently $L^X\varphi(t_0,x_0)\leq 0$.

Thus we have
$$
||\lambda\varphi -L^X\varphi||_\infty\geq \lambda\varphi(t_0,x_0)-L^X\varphi(t_0,x_0)\geq\lambda\varphi(t_0,x_0)=\lambda||\varphi||_\infty.$$

{\bf b)} Assume now  $\varphi$ reaches a positive maximum at a point 
$(T,x_0)$, $x_0\in\R$, therefore this positive maximum is in fact zero. Thus, either $\varphi$ is the null function and we have automatically 
$\lambda||\varphi||_\infty\leq ||\lambda\varphi -L^X\varphi||_\infty$. Either this is not the case and $\varphi$ reaches a strictly negative minimum
on $[0,T)\times\R$. Thus considering~$-\varphi$ and applying Subset a) we get the desired inequality.

\vspace{0.1cm}
{\bf c)} If it is $-\varphi$ that reaches a positive maximum, we may repeat Substeps a)-b) to get $\lambda||\varphi||_\infty\leq ||\lambda\varphi -L^X\varphi||_\infty$.

\vspace{0.1cm}

{\bf STEP4.} We apply Theorem 1.2.12 in \cite{ethier} to see that the closure of $(L^X,\cS^{X})$ generates a strongly continuous, contraction semigroup $(T_t)$ on $C_0$.

{\bf STEP5.} It remains to see that $(T_t)$ is positive, but this can be accomplished in the same manner as in the proof of Theorem 4.2.2 in \cite{ethier} (note that $(T_t)$ is conservative,
thanks to Proposition III.2.2 in \cite{RY}).
\end{proof}

\begin{remark}
\label{rem:gene-cyl}
In fact, if we do not have $\Delta_{\bf x}=\Delta$, to prove that $\varphi'_t(t_0,x_0)\leq 0$ in Step 3-b) (case $(t_0,x_0)\in\Delta$) seems more difficult. Besides, note that we would have to define the domain $\cS^X$ in a different manner, as we would no more have
the continuity of $u'_t$ for $u$ solving the resolvent equation~\eqref{eq:res-forte} (see Remark \ref{rem:ut-cont}).
\end{remark}

\appendix
%\section{Appendix}

\section{The Itô-Peskir formula}

The assumption of the It\^o-Peskir formula in \cite{peskir} is difficult to check in general and does not seem to be valid for the solution $u(t,x)$
 of a problem of type $(\mathcal{P}^\lambda_{\Delta_{\bf x}}(\sigma,b,\beta))$, which is our main purpose.

The first object of this section is to prove the slight modifications (stated in our Subsection \ref{ssec:peskir} in Theorem \ref{thm-peskir}) of the result stated in \cite{peskir}. We recall that we use a stronger assumption on the curve $\gamma$ but with a somewhat weakened assumption on the function $r$. The method of proof is similar to that of \cite{peskir} (see the second proof in \cite{peskir} p. 17) and uses the famous trick of T. Kurtz. Such a trick has already been used in other works in order to relax the assumptions of the It\^o-Peskir formula in the case where $\gamma(t)\equiv 0$ and applied for a particular semimartingale in \cite{talay-martinez}.

For notational convenience, a function $r$ satisfying the assumptions of Theorem \ref{thm-peskir}  will be denoted to belong to the class $C^{1,2}_{-}(C)\cup C^{1,2}_{+}(D)$.
 Note that although this set of assumptions is quite strong, it does in general not guarantee $r$ to be in $C^{1,2}(\overline{C})\cap C^{1,2}(\overline{D})$ in the sense of \cite{peskir}.

\begin{proof}[Proof of Theorem \ref{thm-peskir}]
We begin first to reduce the study to the case where the frontier is the straight line $x=0$. To this end, let us 
set for $(t,x)\in E$ 
$$
G(t,x) = r(t, x+\gamma(t))
$$
and 
$$
Y_t = X_t - \gamma(t).
$$
We have that 
$$
r(t,X_t) = r(t, X-\gamma(t) + \gamma(t)) = G(t, Y_t)
$$
Moreover, we see that 
\begin{equation}
\label{Eq:relations-r-G}
G'_x(t,x) = r'_x(t, x + \gamma(t))~;~G''_{xx}(t,x) = r''_{xx}(t, x + \gamma(t))~;~G'_{t}(t,x) = r'_{t}(t, x + \gamma(t)) + \gamma'(t)r'_x(t, x + \gamma(t))
\end{equation}
where we have used the crucial fact that $\gamma \in C^1$ for the partial derivative w.r.t the time variable. Note also that $Y$ is a semimartingale.

We see that $r\in C^{1,2}_{-}(C)\cup C^{1,2}_{+}(D)$ transfers to $G\in C^{1,2}_{-}(\R^\ast_{-})\cup C^{1,2}_{+}(\R_+^\ast)$.

We will now prove the It\^o-Peskir formula applied to $G$ and $Y$ with $t\mapsto \tilde{\gamma}(t)\equiv 0$ as the frontier.

Let us now introduce two functions $G_1$ and $G_2$ that will play a similar role as $r_1$ and $r_2$ in the original assumptions of \cite{peskir}. 
We define $G_1$ as the symmetrization of $G$ restricted to $\R_-$, namely
\begin{align}
G_1(t,x)=\left \{
\begin{array}{ll}
G(t,x) &\text{if }x<0\\
2G(t,0) - G(t,-x) &\text{if }x\geq 0
\end{array}
\right .
\end{align} 
and
$G_2$ as the symmetrization of $G$ restricted to $\R_+$, namely
\begin{align}
G_2(t,x)=\left \{
\begin{array}{ll}
G(t,x) &\text{if }x>0\\
2G(t,0)-G(t,-x) &\text{if }x\leq 0
\end{array}
\right .
\end{align} 
Note that these functions are continuous and that since $G\in C^{1,2}_{-}(\R^\ast_{-})\cup C^{1,2}_{+}(\R_+^\ast)$, 
and because of the symmetry in the definition, we see that $G_1$ and $G_2$ belong to $C^{1,1}(E)$. For the second space derivatives, the partial functions $x\mapsto G_1(t,x)$ and  $x\mapsto G_2(t,x)$ are shown to lay in $C^2(\R\setminus \{0\})$. In particular $G_1$ and $G_2$ belong to $C^{1,2}([0,T]\times \R\setminus \{0\})$ with the partial derivatives having limits as $x$ tends to $0$.

We now claim that it is possible to apply a (almost) classical It\^o formula to $G_1$ and $G_2$.
In order to prove this fact, one may use a regularization technique, the dominated convergence theorems for classical and stochastic integrals in order to handle the first order partial derivatives, and finally that
$$
\int_0^t G''_{{i}_{xx}}(s, Y_s){\bf 1}_{Y_s = 0}d\langle Y\rangle_s = 0,\quad i=1,2,
$$
as a consequence of the generalized occupation-time formula (see again Exercise VI.1.15 in \cite{RY}).
Since the proof would be long but without difficulties, we decide to omit it. 

We are now in position to follow the second proof in \cite{peskir}~-~Section 3. {\it Another proof and extensions}.

Set $Z_t^1 = Y_t \wedge 0 = \frac{1}{2}(Y_t - |Y_t|)$ and $Z_t^2 = Y_t \vee 0 = \frac{1}{2}(Y_t + |Y_t|)$. We use the trick due to T. Kurtz~:~
\begin{equation}
G(t,Y_t) = G_1\pare{t, Z_t^1} + G_2(t, Z_t^2) - G(t,0).
\end{equation}
The rest of the proof now may follow exactly the same lines as \cite{peskir}~-~Section 3. {\it Another proof and extensions}. Namely, we differentiate $Z^1$ and $Z^2$ with the use of the It\^o-Tanaka formula and apply the classical It\^o formula to $G_1$ and $G_2$ and semimartingales $Z^1$ and $Z^2$. The remaining difficulty in the proof is to identify the terms.

Hence, we prove that
\begin{align*}
G(t,Y_t) =&G(0,Y_0) + \int_0^t \frac{1}{2}\pare{G_t(s,Y_s+) + G_t(s, Y_s-)}ds\nonumber\\
&+ \int_0^t \frac{1}{2}\pare{G_x(s, Y_s+) + G_x(s, Y_s-)}dY_s + \frac{1}{2}\int_0^t G''_{xx}(s, Y_s){\bf 1}_{Y_s\neq 0}d\langle Y\rangle_s\nonumber\\
& + \frac{1}{2}\int_0^t \pare{G_x(s,Y_s+) - G_x(s, Y_s-)}{\bf 1}_{Y_s= 0}dL_s^{0}(Y).
\end{align*}
Now recalling that $G(t,Y_t) = r(t,X_t)$, $Y_t=X_t - \gamma(t)$ and the relations \eqref{Eq:relations-r-G}, we get
\begin{align*}
r(t,X_t) =&r(0,X_0) + \int_0^t \frac{1}{2}\pare{r_t(s,X_s+) + r_t(s, X_s-)}ds + \int_0^t \frac{1}{2}\pare{r_x(s,X_s+) + r_x(s, X_s-)}\gamma'(s)ds\nonumber\\
&+ \int_0^t \frac{1}{2}\pare{r_x(s, X_s+) + r_x(s, X_s-)}dX_s - \int_0^t \frac{1}{2}\pare{r_x(s, X_s+) + r_x(s, X_s-)}d\gamma(s)\nonumber\\
&+\frac{1}{2}\int_0^t r''_{xx}(s, X_s){\bf 1}_{X_s\neq \gamma(s)}d\langle Y\rangle_s + \frac{1}{2}\int_0^t \pare{r_x(s,X_s+) - r_x(s, X_s-)}{\bf 1}_{X_s= \gamma(s)}dL_s^{\gamma}(X),
\end{align*}
and since $d\gamma(s) = \gamma'(s)ds$, we get the formula.
\end{proof}

We end this section by proving Corollary \ref{cor-peskirmulti}.

\begin{proof}[Proof of Corollary \ref{cor-peskirmulti}]
We denote $\varepsilon_y=\inf_{1\leq i\leq I-1}\inf_{t\in[0,T]}(y_{i+1}(t)-y_i(t))$.
We can construct continuous functions $r_i:[0,T]\times \R \to \R$, $1\leq i\leq I$, in the following way:

For any $t\in [0,T]$, we require that $r_1(t,y)=r(t,y)$ for all $y<y_1(t)+\varepsilon_y/4$ and $r_1(t,y)=0$ for $y\geq y_2(t)-\varepsilon_y/4$
 and choose arbitrarily the restriction of $r_1$ on
$\{(t,z)\in [0,T]\times\R:y_1(t)+\varepsilon/4\leq z<y_2(t)-\varepsilon_y/4\}$
in order to have
$r_1\in C^{1,2}( \overline{D^y_{0}})\cap C^{1,2}( \overline{D^y_1})$. 

Then
for $1<i<I$, for any $t\in [0,T]$, we set $r_i(t,y)=0$ 
for $y< y_{i-1}(t) +\varepsilon_y/4$ and $y\geq y_{i+1}(t)-\varepsilon_y/4$, 
$r_i(t,y)=r(t,y)-r_{i-1}(t,y)$ for $y_{i-1}(t)+\varepsilon/4\leq z<y_i(t)-\varepsilon_y/4$,
$r_i(t,y)=r(t,y)$ for all $y_i(t)-\varepsilon_y/4\leq y<y_{i}(t)+\varepsilon_y/4$. We choose 
arbitrarily the restriction of $r_i$ on
$\{(t,z)\in [0,T]\times\R:y_i(t)+\varepsilon_y/4\leq z<y_{i+1}(t)-\varepsilon_y/4\}$ in order to have 
$r_i\in C^{1,2}( \overline{D^y_{i-1}})\cap C^{1,2}( \overline{D^y_i})$.

Finally, for any $t\in [0,T]$, 
$r_I(t,y)=0$ for all $y<y_{I-1}(t)+\varepsilon_y/4$, $r_I(t,y)=r(t,y)-r_{I-1}(t,y)$ for $y_{I-1}(t)+\varepsilon_y/4\leq y<y_I(t)-\varepsilon_y/4$ and $r_I(t,y)=r(t,y)$ for all $y\geq y_I(t)-\varepsilon_y/4$.

\vspace{0.2cm}
Notice that this construction ensures that 
 $r_i\in C^{1,2}( \overline{D^y_{i-1}})\cap C^{1,2}( \overline{D^y_i})$ for any $1\leq i\leq I$ and
$$
r=\sum_{i=1}^Ir_i.$$
  
 Therefore the result, by summation of formula \eqref{eq:peskir} in Theorem \ref{thm-peskir}, and linearity of the derivatives. The second part of the corollary is
 proved in a similar manner.
\end{proof}

\section{Partial Differential Equations aspects}

\begin{proof}[Proof of Theorem \ref{thm:gene-lions}]{\bf STEP1.} We first treat the case $f\equiv 0$, and deal for the moment with a source term $G_*\in\cV'$.

{\bf a)} With the constant $\lambda_0>0$ of Condition ii) we denote $\cA_{\lambda_0}(u,v)=\cA(u,v)+\lambda_0\langle u,v\rangle_\cH$, for any $u,v\in\cV$. Using the triangular inequality, the Cauchy-Schwarz inequality, and 
$||\cdot||_\cH\leq||\cdot||_\cV$ we get from i)
\begin{equation}
\label{eq:Acont}
\forall u,v\in\cV,\quad |\cA_{\lambda_0}(u,v)|\leq C'||u||_\cV||v||_\cV,
\end{equation}
with $C'=\max(C,\lambda_0)$.  For any $w\in\cV$ the map $v\mapsto\cA_{\lambda_0}(w,v)$, $v\in\cV$, is a continuous linear form (thanks to \eqref{eq:Acont}), which we denote by $-A_{\lambda_0} w$. In other words $-A_{\lambda_0}:\cV\to\cV'$ is defined by
$$
\big\langle -A_{\lambda_0} w,v\big\rangle_{\cV',\cV}=\cA_{\lambda_0}(w,v),\quad\forall w,v\in\cV.$$
Again thanks to \eqref{eq:Acont} it can be seen that the linear application $-A_{\lambda_0}$ is continuous. Further, thanks to~ii), it satisfies
$$
\langle-A_{\lambda_0} v,v\rangle_{\cV',\cV}\geq \alpha_0||v||_\cV^2,\quad\forall v\in\cV.$$

Theorem 3.1.1 in \cite{lions-magenes} asserts then that $-\Lambda-A_{\lambda_0}$ is an isomorphism from $\cV\cap D(\Lambda,\cV')$
to $\cV'$, so that  for any $G_*^{\lambda_0}\in\cV'$ there exists a unique $w\in\cV\cap  D(\Lambda,\cV')$ (in particular 
$w(T,\cdot)=0$) such that,
\begin{equation}
\label{eq:sol-faible3}
 \big\langle -\dfrac{\mathrm{d} w}{\mathrm{d t}},v\big\rangle_{\cV',\cV}+\cA_{\lambda_0}(w,v)=
\big\langle G_*^{\lambda_0},v\big\rangle_{\cV',\cV}\quad\forall v\in\cV.
\end{equation}
\noindent
{\bf b)} We denote now $e^{\lambda_0\cdot}$ the function $t\mapsto e^{\lambda_0 t}$, and define $G_*^{\lambda_0}\in\cV'$ by 
$$
\big\langle G_*^{\lambda_0},v\big\rangle_{\cV',\cV}=\big\langle G_*,e^{\lambda_0\cdot}v\big\rangle_{\cV',\cV},\quad \forall v\in\cV.$$

We set $u_*(t,x)=e^{-\lambda_0t}w(t,x)$ with the function $w\in\cV$ satisfying \eqref{eq:sol-faible3}. Note that 
$u_*\in\cV\cap D(\Lambda,\cV')$
(in particular $\dfrac{\mathrm{d} u_*}{\mathrm{d t}}\in\cV'$ and $u_*(T,\cdot)= 0$). We have
$$
-\frac{\mathrm{d}w}{\mathrm{d t}}=-\lambda_0 w-e^{\lambda_0\cdot}\frac{\mathrm{d}u_*}{\mathrm{d t}}.$$
So that we deduce from \eqref{eq:sol-faible3} that
\begin{equation*}
\big\langle -\dfrac{\mathrm{d} u_*}{\mathrm{d t}},e^{\lambda_0\cdot}v\big\rangle_{\cV',\cV}+\cA(u_*,e^{\lambda_0\cdot}v)=
\big\langle G_*,e^{\lambda_0\cdot}v\big\rangle_{\cV',\cV},\quad \forall v\in\cV,
\end{equation*}
and thus
\begin{equation}
\label{eq:sol-faible4}
\big\langle -\dfrac{\mathrm{d} u_*}{\mathrm{d t}},v\big\rangle_{\cV',\cV}+\cA(u_*,v)=
\big\langle G_*,v\big\rangle_{\cV',\cV},\quad \forall v\in\cV.
\end{equation}
 
\vspace{0.1cm}

{\bf STEP2.}   We go back to the general case $f\in L^2(\R)$. Applying a trace theorem, we get the existence of $u_r\in L^2(0,T;H^1(\R))$,  with 
$\dfrac{\mathrm{d} u_r}{\mathrm{d t}}\in L^2(0,T;H^{-1}(\R))$, s.t. $u_r(T,\cdot)=f$ (cf \cite{lions-magenes} Chap. 1, N° 3). 
We define $G_*\in\cV'$ by
$$
\langle G_*,v\rangle_{\cV',\cV}=\langle G,v\rangle_{\cV',\cV}-\cA(u_r,v)
+\langle \dfrac{\mathrm{d} u_r}{\mathrm{d t}},v\rangle_{\cV',\cV},\quad\forall v\in\cV.$$
Considering the function $u_*$ of Step 1 we set $u=u_*+u_r$. Note that $u\in L^2(0,T;H^1(\R))$ (because $u\in\cV$), that 
$\dfrac{\mathrm{d} u}{\mathrm{d t}}\in L^2(0,T;H^{-1}(\R))$ (because $\frac{\mathrm{d}u_*}{\mathrm{d t}}\in\cV'$)
 and that $u(T,\cdot)=f$.

From \eqref{eq:sol-faible4} we get for any $v\in\cV$,
$$
\big\langle -\dfrac{\mathrm{d} u}{\mathrm{d t}},v\big\rangle_{\cV',\cV}=\langle G_*,v\rangle_{\cV',\cV}-\cA(u_*,v)-\big\langle \dfrac{\mathrm{d} u_r}{\mathrm{d t}},v\big\rangle_{\cV',\cV}=\langle G,v\rangle_{\cV',\cV}-\cA(u_*+u_r,v),
$$
and therefore \eqref{eq:sol-faible2}.  Besides, as $\dfrac{\mathrm{d} u}{\mathrm{d t}}\in L^2(0,T;H^{-1}(\R))$,
we can see from Theorem 1.3.1 and Proposition 1.2.1 in \cite{lions-magenes} that $u\in C([0,T];L^2(\R))$.

\vspace{0.2cm}
{\bf STEP3: uniqueness.} Suppose that $\bar{u}$ is another element of $L^2(0,T;H^1(\R))\cap C([0;T];L^2(\R))$ with $\bar{u}(T,\cdot)=f$ and 
$\dfrac{\mathrm{d} \bar{u}}{\mathrm{d t}}\in L^2(0,T;H^{-1}(\R))$
that satisfies \eqref{eq:sol-faible2}. Then $u-\bar{u}$ is an element
of $\cV\cap  D(\Lambda,\cV')$ (note that in particular $(u-\bar{u})(T,\cdot)=0$), that satisfies \eqref{eq:sol-faible4} with $G_*=0$. If we set
$w=e^{\lambda_0\cdot}(u-\bar{u})$ we will see that $w\in\cV\cap D(\Lambda,\cV')$ solves \eqref{eq:sol-faible3} with $G_*^{\lambda_0}=0$. But $-\Lambda-A_{\lambda_0}$ is an isomorphism, as stated in Step 1-a). Thus $w=0$ and therefore $u-\bar{u}=0$.
\end{proof}

\begin{proof}[Proof of Lemma \ref{lem:Acontcoer}]

First, recall that $\rho,a\in\Theta(m',M')$, $m'<M'$, and note that we have for any $v\in\cH$,
\begin{equation}
\label{eq:normeV}
\frac{1}{M'} ||v|| \leq ||v||_\cH\leq 
\frac{1}{m'} ||v||.
\end{equation}

Taking $u,v\in\cV$, we have, using Schwarz's inequality and \eqref{eq:normeV}
$$
\begin{array}{lll}
|\cA(u,v)|&\leq& M'\big|\big|\frac{\mathrm{d}u}{\mathrm{dx}}\big|\big|\times\big|\big|\frac{\mathrm{d}v}{\mathrm{dx}}\big|\big|
+\frac{M'}{m'}\big|\big|\frac{\mathrm{d}u}{\mathrm{dx}}\big|\big|\times||v||+\frac{\lambda}{m'}||u||\times||v||\\
\\
&\leq&C_3||u||_\cV||v||_\cV,\\
\end{array}
$$
where $C_3$ depends an $m',M',\lambda$. Therefore \eqref{eq:Adefcont} is proven.

Taking now $v\in\cV$ we have
$$
\cA_{\lambda_0}(v,v)\geq m'\big|\big|\frac{\mathrm{d}v}{\mathrm{dx}}\big|\big|^2+(\lambda+\lambda_0)||v||^2_\cH
-\int_0^T\int_\R\big( B \frac{\mathrm{d}v}{\mathrm{dx}} v \rho^{-1} \big)(t,x)dxdt.$$
But, using Young's inequality we get for $\delta>0$,
$$
\Big| \int_0^T\int_\R\big( B \frac{\mathrm{d}v}{\mathrm{dx}} v \rho^{-1} \big)(t,x)dxdt \Big|
\leq \frac{M'}{m'}\int_0^T\int_\R\big( \frac{1}{2\delta}\big|\frac{\mathrm{d}v}{\mathrm{dx}}\big|^2+\frac{\delta}{2}|v|^2  \big)(t,x)dxdt.$$
Choosing $\delta=M'/m'^2$ we get
$$
\cA_{\lambda_0}(v,v)\geq \frac{m'}{2}\big|\big|\frac{\mathrm{d}v}{\mathrm{dx}}\big|\big|^2
+(\lambda+\lambda_0)||v||^2_\cH-\frac{M'^2}{2m'^3}||v||^2,$$
and then, again by \eqref{eq:normeV},
$$
\cA_{\lambda_0}(v,v)\geq \frac{m'^3}{2}\big|\big|\frac{\mathrm{d}v}{\mathrm{dx}}\big|\big|_\cH^2+
\big( \lambda+\lambda_0-\frac{M'^4}{2m'^3} \big)||v||_\cH^2.
$$
Therefore it suffices to choose $\lambda_0>0$ s.t. $\lambda+\lambda_0>M'^4/2m'^3$ in order to define 
$\alpha_0=\min\{\frac{m'^3}{2},\lambda+\lambda_0-\frac{M'^4}{2m'^3}\}$ that satisfies \eqref{eq:Acoer}.
\end{proof}

\begin{proof}[Proof of Lemma \ref{lem:ut-L2}]
We consider a mollification $u_\tau(t,x)=u(t,x,\tau)$ of $u(t,x)$ (see for instance p22 in \cite{lieberman}).
We will show that 
\begin{equation}
\label{eq:ut-maj}
\int_0^T\int_\R|(u_\tau)'_t|^2\,dxdt\leq C_5
\end{equation}
with a constant $C_5$ not depending on $\tau$. Using a compactness argument this implies that there is an element 
$w\in L^2(0,T;L^2(\R))$ such that
 for any 
$\varphi\in C^{\infty,\infty}_{c,c}(E)$ the quantity $\int_0^T\int_\R(u_\tau)'_t\varphi\rho^{-1}\,dxdt$ converges to
$\int_0^T\int_\R w\varphi\rho^{-1}\,dxdt$, as $\tau\downarrow 0$. But using integration by parts with respect to the time variable we have
$$
\int_0^T\int_\R(u_\tau)'_t\varphi\rho^{-1}\,dxdt
=-\int_0^T\int_\R u_\tau\varphi'_t\rho^{-1}\,dxdt+\int_0^T\int_\R u_\tau\varphi\frac{\rho'_t}{\rho}\rho^{-1}\,dxdt.$$
But the right hand side term in the above expression converges (as $\tau\downarrow 0$) to
$$
-\int_0^T\int_\R u\varphi'_t\rho^{-1}\,dxdt+\int_0^T\int_\R u\varphi\frac{\rho'_t}{\rho}\rho^{-1}\,dxdt$$
(see Lemma 3.2 in \cite{lieberman}), which is nothing else than $\langle \dfrac{\mathrm{d} u}{\mathrm{d t}},\varphi\rangle_{\cV',\cV}$ (using the notations of Subsection \ref{ss:sol-faible}).
Therefore we will get the desired result.

Using Fubini type arguments and Lemma 3.3 in \cite{lieberman} we can see that $u_\tau$ is a weak solution of
$(\cP^\lambda_{\mathrm{div},T}(\rho_\tau,a_\tau,B_\tau))$, where $\rho_\tau=\rho(\cdot,\tau)$ (resp. $a_\tau$, $B_\tau$) is a mollified
version of $\rho$ (resp. $a$, $B$). Following \cite{lady1} the idea is to use $\varphi=(u_\tau)'_t\zeta_n^2\in H^{1,1}(E)$ as a test function
in \eqref{eq:sol-faible}, where $\zeta_n$ is some element of a sequence of cut-off functions $(\zeta_n)$ (this sequence can be defined for example in the same spirit as in the proof of Theorem VIII.6 in \cite{brezis}). In the following computations we drop any reference to the 
subscripts $\tau$ and $n$. But the function denoted by $u$ is smooth so that $u'_t$ and $u'_x$ exist in the classical sense. So that using integration
by parts w.r.t. the time variable in \eqref{eq:sol-faible} we first get
$$
\begin{array}{c}
-\int_0^T\int_\R|u'_t|^2\zeta^2\rho^{-1}+\frac 1 2 \int_0^T\int_\R au'_x(u''_{tx}\zeta^2+2u'_t\zeta\zeta'_x)
-\int_0^T\int_\R Bu'_xu'_t\zeta^2\rho^{-1}+\lambda\int_0^T\int_\R uu'_t\zeta^2\rho^{-1}\\
\\
=-\int_0^T\int_\R gu'_t\zeta^2\rho^{-1}.\\
\end{array}
$$
Using the relation 
$au'_xu''_{xt}\zeta^2=\frac 1 2\big( a|u'_x|^2\zeta^2 \big)'_t-\frac 1 2 a'_t|u'_x|^2\zeta^2-a|u'_x|^2\zeta\zeta'_t$ and 
$\zeta(0,\cdot)=\zeta(T,\cdot)=0$
we get
$$
\begin{array}{lll}
\int_0^T\int_\R|u'_t|^2\zeta^2\rho^{-1}&=&
-\frac 1 4 \int_0^T\int_\R a'_t|u'_x|^2\zeta^2-\frac 1 2 \int_0^T\int_\R a|u'_x|^2\zeta\zeta'_t+\int_0^T\int_\R au'_xu'_t\zeta\zeta'_x
-\int_0^T\int_\R Bu'_xu'_t\zeta^2\rho^{-1}\\
\\
&&+\lambda\int_0^T\int_\R uu'_t\zeta^2\rho^{-1}+\int_0^T\int_\R gu'_t\zeta^2\rho^{-1}.
\end{array}
$$
Remembering now that $\rho,a\in\Theta(m',M')$, $B\in\Xi(M')$ and $|a'_t|\leq C_3(k,\kappa,M)$, we get, using triangular
and Young inequalities, that
$$
\begin{array}{lll}
\frac{1}{M'}\int_0^T\int_\R|u'_t|^2\zeta^2&\leq&\frac 1 4C_3\int_0^T\int_\R|u'_x|^2\zeta^2+\frac{M'}{2}\int_0^T\int_\R|u'_x|^2\zeta\zeta'_t
+M'\delta\int_0^T\int_\R|u'_t|^2\zeta^2+\frac{M'}{\delta}\int_0^T\int_\R|u'_x|^2|\zeta'_x|^2
\\
\\
&&+\frac{M'}{m'}\delta\int_0^T\int_\R|u'_t|^2\zeta^2+\frac{M'}{m'}\frac 1 \delta\int_0^T\int_\R|u'_x|^2\zeta^2
+\frac{\lambda}{m'}\delta\int_0^T\int_\R|u'_t|^2\zeta^2+\frac{\lambda}{m'}\frac 1 \delta \int_0^T\int_\R u^2\zeta^2\\
\\
&&+\frac{\delta}{m'}\int_0^T\int_\R|u'_t|^2\zeta^2+\frac{1}{m'\delta}\int_0^T\int_\R g^2\zeta^2
\end{array}
$$
for any $\delta>0$. Adjusting now $\delta$ we get
$$
\int_0^T\int_\R|u_t|^2\zeta^2\leq C_4  \int_0^T\int_\R\big(|u'_x|^2+u^2+g^2 \big)\zeta^2
+C_4\int_0^T\int_\R|u'_x|^2(\zeta\zeta'_t+|\zeta'_x|^2)$$
with a constant $C_4$ depending on $m',M',M,k,\kappa,\delta,\lambda$. Using now $u,u'_x,g\in L^2(0,T;L^2(\R))$, Lemmas 3.2 and 3.3 in \cite{lieberman}, and the fact that
$|\zeta'_t|\leq c_4\frac 1 n$ and $|\zeta'_x|\leq c_4\frac 1 n$ ($c_4>0$ is some constant), we get \eqref{eq:ut-maj} with a constant
$C_5$ not depending on $\tau>0$, by letting $n$ tend to infinity.
\end{proof}

\section*{Acknowledgements}

 Both authors wish to acknowledge Faouzi Triki for fruitful discussions on the PDE aspects, and the anonymous referee for his very valuable comments.

%\section*{References}
\bibliography{../AAABIBLIOGRAPHIE/BIB.bib}
%\bibliography{BIB.bib}

\end{document}